\documentclass[12pt, a4paper]{article}

\usepackage{amsmath}
\usepackage{amsfonts}
\usepackage{amssymb}
\usepackage[francais,english]{babel}

\def\english{\selectlanguage{english}}

\providecommand\mathbb{\bf}

\newcommand\R{{\mathbb R}}

\newtheorem{thm}{Theorem}[section]
\newtheorem{lemma}{Lemma}[section]

\newtheorem{pro}{Proposition}[section]

\newtheorem{coro}{Corollary}[section]
\newtheorem{remark}{Remark}[section]

\newcounter{Remark}

\renewcommand\theRemark{\arabic{Remark}}

\newcounter{steps}
\newenvironment{proof}[1][]{%
\par\medbreak\setcounter{steps}{0}
{\noindent\bfseries Proof#1. }} {\hfill\fbox{\ }\medbreak}

\newcounter{substeps}[steps]

\newcommand{\vtz}[0]{
\frac{\left ((v_3 - \vtp)\frac{(z,0)}{|z|}, -|z|e_3  \right )}{\sqrt{|z|^2 + (v_3 - \vtp )^2}}}

\newcommand{\pz}[0]{
{^\perp z}}

\newcommand{\olvpolv}[0]{
\left (\frac{(\olv, 0)}{|\olv|},\frac{(\polv, 0)}{|\olv|}
\right )}

\newcommand{\polvolv}[0]{
\left (\frac{(\polv, 0)}{|\olv|},-\frac{(\olv, 0)}{|\olv|}   \right )}

\newcommand{\olvppolvp}[0]{
\left (\frac{(\olvp, 0)}{|\olvp|},\frac{(\polvp, 0)}{|\olvp|}
\right )}

\newcommand{\polvpolvp}[0]{
\left (\frac{(\polvp, 0)}{|\olvp|},-\frac{(\olvp, 0)}{|\olvp|}   \right )}

\newcommand{\z}[0]{
(\oc \olx + \polv ) - ( \oc \olxp + \polvp)}

\newcommand{\tD}[0]{
\tilde{D}}

\newcommand{\tDp}[0]{
\tilde{D}^\prime}

\newcommand{\tDt}[0]{
\tilde{D}_3}

\newcommand{\tDtp}[0]{
\tilde{D}^\prime _3}

\newcommand{\tO}[0]{
{^t {\cal O}}}

\newcommand{\musi}[0]{
\mu _{s s^\prime} ^2 \sigma _{s s^\prime} (|v - v^\prime|) |v - v^\prime |^3}

\newcommand{\inttpi}[0]{
\int _0 ^{2\pi}\!\!\!\!}

\newcommand{\avetpi}[0]{
\frac{1}{2\pi}\int _0 ^{2\pi}\!\!\!\!}

\newcommand{\tC}[0]{
\tilde{C}}

\newcommand{\ltvm}[0]{
{L^2 ( M^{-1}\mathrm{d}v  )}}

\newcommand{\ltm}[0]{
{L^2 _M}}

\newcommand{\D}[0]{
\mathrm{D}}

\newcommand{\ran}[0]{
\mathrm{Range\;}}

\newcommand{\fpl}[0]{
Fokker-Planck-Landau }

\newcommand{\fp}[0]{
Fokker-Planck }

\newcommand{\ol}[1]{
\overline{#1}}

\newcommand{\pol}[1]{
{^{\;\perp} \overline{#1}}}

\newcommand{\olv}[0]{
{\overline{v}}}

\newcommand{\polv}[0]{
{^{\; \perp} \overline{v}}}

\newcommand{\olx}[0]{
{\overline{x}}}

\newcommand{\oly}[0]{
{\overline{y}}}

\newcommand{\polx}[0]{
{^{\; \perp} \overline{x}}}

\newcommand{\olV}[0]{
{\overline{V}}}

\newcommand{\polV}[0]{
{^{\; \perp} \overline{V}}}

\newcommand{\olX}[0]{
{\overline{X}}}

\newcommand{\olE}[0]{
\overline{E}}

\newcommand{\polE}[0]{
{^{\; \perp} \overline{E}}}

\newcommand{\eps}[0]{
\varepsilon}

\newcommand{\oc}[0]{
\omega _c}

\newcommand{\fe}[0]{
f ^\varepsilon}

\newcommand{\fek}[0]{
f ^{\varepsilon _k}}

\newcommand{\Be}[0]{
B ^\varepsilon}

\newcommand{\fin}[0]{
f ^{\mathrm{in}}}

\newcommand{\fo}[0]{
f ^1}

\newcommand{\ftwo}[0]{
f ^2}

\newcommand{\divx}[0]{
\mathrm{div}_x}

\newcommand{\divv}[0]{
\mathrm{div}_v}

\newcommand{\divolx}[0]{
\mathrm{div}_{\overline{x}}}

\newcommand{\divoolx}[0]{
\mathrm{div}_{\omega _c \overline{x}}}

\newcommand{\divolv}[0]{
\mathrm{div}_{\overline{v}}}

\newcommand{\divxv}[0]{
\mathrm{div}_{x,v}}

\newcommand{\divoxv}[0]{
\mathrm{div}_{\omega _c x,v}}

\newcommand{\gradoxv}[0]{
\nabla _{\oc x,v}}

\newcommand{\gradoxvp}[0]{
\nabla _{\oc x,v^\prime}}

\newcommand{\gradxv}[0]{
\nabla _{x,v}}

\newcommand{\gradxvp}[0]{
\nabla _{x,v^\prime}}

\newcommand{\gradoxpvp}[0]{
\nabla _{\oc x^{\prime},v^{\prime}}}

\newcommand{\gradolx}[0]{
\nabla _{\olx}}

\newcommand{\gradolv}[0]{
\nabla _{\olv}}

\newcommand{\ave}[1]{
\left \langle #1 \right \rangle }

\newcommand{\avess}[1]{
\left \langle #1 \right \rangle _{\sigma S}}

\newcommand{\md}[0]{
\mathrm{d}}

\newcommand{\vp}[0]{
v^{\prime}}

\newcommand{\rp}[0]{
r^{\prime}}

\newcommand{\alphap}[0]{
\alpha^{\prime}}

\newcommand{\svvp}[0]{
s(v,v^{\prime})}

\newcommand{\intvp}[1]{
\int _{\R ^3} #1\;\mathrm{d}v^{\prime}  }

\newcommand{\intv}[1]{
\int _{\R ^3} #1\;\mathrm{d}v  }

\newcommand{\intx}[1]{
\int _{\R ^3} #1\;\mathrm{d}x  }

\newcommand{\intvm}[1]{
\int _{\R ^3} #1\;\frac{\mathrm{d}v}{M}  }

\newcommand{\qb}[0]{
Q_B }

\newcommand{\qbp}[0]{
Q_B ^+}

\newcommand{\qbm}[0]{
Q_B ^-}

\newcommand{\qfpl}[0]{
Q_{FPL} }

\newcommand{\qfp}[0]{
Q_{FP} }

\newcommand{\intvvp}[1]{
\int _{\R^3}\!\int _{\R^3} \!\!#1\;\mathrm{d}v^{\prime}\mathrm{d}v}

\newcommand{\intxv}[1]{
\int _{\R^3}\!\int _{\R^3} \!\!#1\;\mathrm{d}v\mathrm{d}x}

\newcommand{\intxvm}[1]{
\int _{\R^3}\!\int _{\R^3} \!\!#1\;\frac{\mathrm{d}x\mathrm{d}v}{M(v)}}

\newcommand{\intolxolv}[1]{
\int _{\R^2}\!\int _{\R^2} \!\!#1\;\mathrm{d}\olv\mathrm{d}\olx}

\newcommand{\intolxv}[1]{
\int _{\R^2}\!\int _{\R^3} \!\!#1\;\mathrm{d}v\mathrm{d}x_1\mathrm{d}x_2}

\newcommand{\intolxpvp}[1]{
\int _{\R^2}\!\int _{\R^3} \!\!#1\;\mathrm{d}v^{\prime}\mathrm{d}x^{\prime}_1\mathrm{d}x^{\prime}_2}

\newcommand{\intolypvp}[1]{
\int _{\R^2}\!\int _{\R^3} \!\!#1\;\mathrm{d}v^{\prime}\mathrm{d}y^{\prime}_1\mathrm{d}y^{\prime}_2}

\newcommand{\olvp}[0]{
\overline{v ^{\prime}}}

\newcommand{\polvp}[0]{
{\; ^\perp \overline{v ^{\prime}}}}

\newcommand{\olxp}[0]{
{\overline{x^\prime}}}

\newcommand{\olyp}[0]{
{\overline{y^\prime}}}

\newcommand{\vtp}[0]{
v^{\prime}_3}

\newcommand{\cocxv}[0]{
C^1 _c (\R^3 \times \R^3)}

\newcommand{\ind}[1]{
{\bf 1}_{\{#1\}}}

\newcommand{\loxv}[0]{
L^1( \R ^3 \times \R ^3)}

\newcommand{\ltxv}[0]{
L^2( \R ^3 \times \R ^3)}

\newcommand{\ltxvm}[0]{
L^2\left ( M^{-1}\mathrm{d}x\mathrm{d}v\right)}

\newcommand{\linftltxvm}[0]{
L^\infty \left ( \R_+,\ltxvm \right )}

\newcommand{\linftloxv}[0]{
L^\infty \left ( \R_+,L^1(\R^3 \times \R ^3) \right )}

\newcommand{\imww}[0]{
\left ( I - \frac{w \otimes w}{|w|^2} \right ) }

\newcommand{\lime}[0]{
\lim _{\varepsilon \searrow 0}}

\newcommand{\limk}[0]{
\lim _{k \to +\infty }}

\newcommand{\inttxv}[1]{
\int _{\R_+} \!\int _{\R ^3} \!\int _{ \R^3}\!\!#1\;\mathrm{d}v\mathrm{d}x\mathrm{d}t}

\newcommand{\inttxvm}[1]{
\int _{\R_+} \!\int _{\R ^3} \!\int _{ \R^3}\!\!#1\;\frac{\mathrm{d}v\mathrm{d}x}{M}\mathrm{d}t}

\newcommand{\sgn}[0]{
\mathrm{sgn}}

\newcommand{\fs}[0]{
f _s}

\newcommand{\fsp}[0]{
f _{s^\prime}}

\newcommand{\ffm}[0]{
\left ( \frac{f}{M} \right )}

\baselinestretch\renewcommand{\baselinestretch}{1.5}

\begin{document}
\english

\title{Finite Larmor radius approximation for the Fokker-Planck-Landau equation}

\author{
Mihai Bostan
\thanks{Laboratoire d'Analyse, Topologie, Probabilit\'es LATP, Centre de Math\'ematiques et Informatique CMI, UMR CNRS 7353, 39 rue Fr\'ed\'eric Joliot Curie, 13453 Marseille  Cedex 13
France. E-mail : {\tt bostan@cmi.univ-mrs.fr}}
\;\;
C\'eline Caldini
\thanks{Laboratoire de
Math\'ematiques de Besan{\c c}on, UMR CNRS 6623, Universit\'e de
Franche-Comt\'e, 16 route de Gray, 25030 Besan{\c c}on  Cedex
France. E-mail : {\tt celine.caldini@univ-fcomte.fr}} 
}

\date{ (\today)}

\maketitle

\begin{abstract}
The subject matter of this paper concerns the derivation of the finite Larmor radius approximation, when collisions are taken into account. Several studies are performed, corresponding to different collision kernels. The main motivation consists in computing the gyroaverage of the Fokker-Planck-Landau operator, which plays a major role in plasma physics. We show that the new collision operator enjoys the usual physical properties~; the averaged kernel balances  the mass, momentum, kinetic energy and dissipates the entropy.

\end{abstract}



\section{Introduction}
\label{Intro}
\indent

Many studies in plasma physics concern the energy production through thermonuclear fusion. In particular this reaction can be achieved by magnetic confinement {\it i.e.,} a tokamak plasma is controlled by applying a strong magnetic field. Large magnetic fields induce high cyclotronic frequencies corresponding to the fast particle dynamics around the magnetic lines. We concentrate on the linear problem, by neglecting the self-consistent electro-magnetic field. The external electro-magnetic field is supposed to be a given smooth field
\[
E = - \nabla _x \phi,\;\;\Be = \frac{B(x)}{\eps}\;b(x),\;\;|b| = 1
\]
when $\eps >0$ is a small parameter, destinated to converge to $0$, in order to describe strong magnetic fields. The scalar function $\phi$ stands for the electric potential, $B(x)>0$ is the rescaled magnitude of the magnetic field and $b(x)$ denotes its direction. 
As usual, we appeal to  the kinetic description for studying the evolution of the plasma. The notation $\fe = \fe (t,x,v) \geq 0$ stands for the presence density of a population of charged particles with mass $m$ and charge $q$. This density satisfies
\begin{equation}
\label{Equ1} \partial _t \fe + v \cdot \nabla _x \fe + \frac{q}{m} ( E + v \wedge \Be ) \cdot \nabla _v \fe = Q(\fe),\;\;(t,x,v) \in \R_+ \times \R ^3 \times \R ^3
\end{equation}
\begin{equation}
\label{Equ2}
\fe (0, x, v) = \fin (x,v),\;\;(x,v) \in \R^3 \times \R^3
\end{equation}
where $Q$ denotes a collision kernel. The interpretation of the density $\fe$ is straightforward~: the number of charged particles contained at time $t$ inside the infinitesimal volume $\mathrm{d}x\mathrm{d}v$ around the point $(x,v)$ of the position-velocity phase space is given by $\fe (t,x,v) \mathrm{d}x\mathrm{d}v$. The equation \eqref{Equ1} accounts for the fluctuation of the density $\fe$ due to the transport but also to the collisions. We analyze here the linear relaxation operator, but also the bilinear \fpl operator. 

When neglecting the collisions the limit model as $\eps \searrow 0$ comes by averaging with respect to the fast cyclotronic motion \cite{GolSai99, SaiRay02, Fre06, BosAsyAna, BosTraEquSin, BosGuiCen3D, BosNeg09, Bos11}. The problem reduces to homogenization analysis and can be solved using the notion of two-scale convergence \cite{FreSon98, FreSon01, FreMou10}.

We point out that a linearized and gyroaveraged collision operator has been written in \cite{XuRos91}, but the implementation of this operator seems very hard. We refer to \cite{Bri04, BriHah07} for a general guiding-center bilinear Fokker-Planck collision operator. Another difficulty lies in the relaxation of the distribution function towards a local Maxwellian equilibrium. Most of the available model operators, in particular those which are linearized near a Maxwellian, are missing this property. Very recently a set of model collision operators has been obtained in \cite{GarDifSarGra09}, based on entropy variational principles \cite{Bri00}. 

We study here the finite Larmor radius scaling {\it i.e.,} the typical perpendicular spatial length is of the same order as the Larmor radius and the parallel spatial length is much larger. We assume that the magnetic field is homogeneous and stationary
\[
\Be = \left ( 0, 0, \frac{B}{\eps} \right )
\]
for some constant $B>0$ and therefore \eqref{Equ1} becomes
\begin{equation}
\label{Equ3} \partial _t \fe + \frac{1}{\eps} ( v_1 \partial _{x_1} \fe + v_2 \partial _{x_2} \fe ) + v_3 \partial _{x_3} \fe  + \frac{q}{m} E \cdot \nabla _v \fe + \frac{\oc}{\eps} ( v_2 \partial _{v_1} \fe - v_1 \partial _{v_2} \fe ) = Q(\fe)
\end{equation}
where $\oc = qB/m$ stands for the rescaled cyclotronic frequency. The density $\fe$ is decomposed into a dominant density $f$ and fluctuations of orders $\eps, \eps ^2, ...$
\begin{equation}
\label{Equ4}
\fe = f + \eps \fo + \eps ^2 \ftwo + ...
\end{equation}
Combining \eqref{Equ3}, \eqref{Equ4} yields, with the notations $\olx = (x_1, x_2), \olv = (v_1, v_2), \polv = (v_2, - v_1)$
\begin{equation}
\label{Equ5}
{\cal T} f : = \olv \cdot \nabla _\olx f + \oc \polv \cdot \nabla _\olv f = 0
\end{equation}
\begin{equation}
\label{Equ6}
\partial _t f + v_3 \partial _{x_3} f + \frac{q}{m} E \cdot \nabla _v f + {\cal T} f^1 = Q(f)
\end{equation}
\[
\vdots
\]
The equation \eqref{Equ5} appears as a divergence constraint
\[
\divxv \{f (\olv,0, \oc \polv,0)\} = 0.
\]
Equivalently, \eqref{Equ5} says that at any time $t$ the density $f(t,\cdot,\cdot)$ remains constant along the flow associated to $\olv \cdot \nabla _\olx + \oc \polv \cdot \nabla _\olv$
\begin{equation}
\label{Equ7} \frac{\md \ol{X}}{\md s} = \ol{V} (s),\;\;\frac{\md X_3}{\md s} = 0,\;\;\frac{\md \ol{V}}{\md s} = \oc \pol{V} (s),\;\;\frac{\md V_3}{\md s} = 0
\end{equation}
and therefore, at any time $t$, the density $f(t,\cdot,\cdot)$ depends only on the invariants of \eqref{Equ7}
\[
f(t,x,v) = g\left ( t, x_1 + \frac{v_2}{\oc}, x_2 - \frac{v_1}{\oc}, x_3, r = |\olv|, v_3\right ).
\]
The time evolution for $f$ comes by eliminating $f^1$ in \eqref{Equ6}. For doing that, we project onto the kernel of ${\cal T}$, which is orthogonal to the range of ${\cal T}$. In order to get a explicit model for $f$ we need a simpler representation for the orthogonal projection on $\ker  {\cal T}$. Actually this projection appears as the average along the characteristic flow \eqref{Equ7}. Denoting by $\ave{\cdot}$ this projection, we obtain
\begin{equation}
\label{Equ9} 
\ave{\partial _t f + v_3 \partial _{x_3}f + \frac{q}{m} E \cdot \nabla _v f } = \ave{Q(f)},\;\;(t,x,v) \in \R_+ \times \R^3 \times \R^3.
\end{equation}
By one hand, averaging $\partial _t + v_3 \partial _{x_3} + \frac{q}{m} E \cdot \nabla _v $ leads to another transport operator
\[
\ave{\partial _t f + v_3 \partial _{x_3}f + \frac{q}{m} E \cdot \nabla _v f } = \partial _t f + \frac{\ave{\polE}}{B} \cdot \nabla _\olx f + v_3 \partial _{x_3} f + \frac{q}{m} \ave{E_3} \partial _{v_3} f .
\]
The key point here is to choose as new coordinates the invariants of \eqref{Equ7} and to observe that the partial derivatives with respect to these invariants commute with the average operator. More generally, for any smooth vector field $\xi = (\xi _x, \xi_v)$, we obtain the following commutation formula between the divergence and average operators, cf. Proposition \ref{ComDivAve}
\begin{eqnarray}
\ave{\divxv \xi} & = & \divolx  \left \{
\ave{\xi _{\olx} + \frac{^\perp \xi _\olv}{\oc}}
 + \ave{\xi _\olv \cdot \frac{\polv}{|\olv|}} 
 \frac{\olv}{\oc |\olv|} -  \ave{\xi _\olv \cdot \frac{\olv}{|\olv|}} 
 \frac{\polv}{\oc |\olv|} \right \}+ \partial _{x_3} \ave{\xi _{x_3}} \nonumber \\
& + & \divolv \left \{\ave{\xi _\olv \cdot \frac{\polv}{|\olv|}} 
 \frac{\polv}{|\olv|} +  \ave{\xi _\olv \cdot \frac{\olv}{|\olv|}} 
 \frac{\olv}{ |\olv|}    \right \}+ \partial _{v_3} \ave{\xi _{v_3}}. \nonumber
\end{eqnarray}
By the other hand we need to compute the average of the collision kernel $Q$ which is a more complicated task. It is convenient to focus first on the relaxation Boltzmann operator \cite{MarRinSch90}
\[
\qb(f(t,x,\cdot))(v) = \frac{1}{\tau} \intvp{\svvp\{M(v) f(t,x,\vp) - M(\vp) f(t,x,v)\}}
\]
where $\tau >0$ is the relaxation time, $\svvp$ is the scattering cross section and $M$ is the Maxwellian equilibrium with temperature $\theta$
\[
M(v) = \frac{1}{(2 \pi \theta /m) ^{3/2}} e ^ { - \frac{m|v|^2}{2\theta}},\;\;v \in \R^3.
\]
We need to average functions like $(x,v) \to \intvp{C(v,\vp)f(x,\vp)}$, where $C(v,\vp)$ is a given function. 
Since the invariants of the flow $(\ol{X}, \ol{V})$ combines $\olx$ and $\olv$, we get a position-velocity integral operator cf. Proposition \ref{FirstFormula} 
\begin{align*}
\ave{\intvp{C(v,\vp)f(x,\vp)}} (x,v) = \oc ^2 \intolxpvp{{\cal C}( |\olv|, v_3, |\olvp|, \vtp, z)f(\olxp, x_3, \vp)}
\end{align*}
with $z = \oc \olx + \polv - (\oc \olxp + \polvp)$. 
We prove that averaging $\qb$ will lead to a position-velocity integral operator of the same form
\begin{align*}
\ave{\qb}f (x,v) : & = \ave{\qb (f)}(x,v) \\
& = \frac{\oc ^2}{\tau} \!\!\intolxpvp{{\cal S} ( |\olv|, v_3, |\olvp|, \vtp,z) \{ M(v)f(\olxp, x_3, \vp) - M(\vp) f (x,v)\}\!}
\end{align*}
(see Theorem \ref{MainResultBol} for the definition of ${\cal S}$).  Observe that $\ave{\qb}$ is global in $(\olx, v)$, but remains local in $x_3$. In particular it satisfies only a global mass balance, which comes easily by Fubini theorem and the symmetry of ${\cal S}$ cf. Remark \ref{GlobalMass}
\[
\intxv{\ave{\qb}f(x,v)} = 0.
\]
In the case of the relaxation operator $\qb$ we obtain the limit model
\begin{thm}
\label{MainResultBol} Assume that the scattering cross section satisfies \eqref{Equ31}, \eqref{Equ35} and that $E(x) = - \nabla _x \phi (x),\phi \in W^{2,\infty} (\R^3)$. Let us consider $\fin \geq 0, \fin \in \loxv{} \cap \ltxvm{}$ and denote by $\fe$ the weak solution of \eqref{Equ3}, \eqref{Equ2} with $Q = \qb$ for any $\eps >0$. We assume that $(\fe)_{\eps >0}$ is bounded in $\linftltxvm{}$. Then the family $(\fe)_{\eps >0}$ converges weakly $\star$ in $\linftltxvm{}$ to the weak solution of
\begin{equation}
\label{Equ45}
\partial _t f + \frac{\ave{\polE}}{B} \cdot \nabla _\olx f + v_3 \partial _{x_3} f + \frac{q}{m} \ave{E_3} \partial _{v_3} f = \ave{\qb} f,\;\;(t,x,v) \in \R_+ \times \R ^3 \times \R ^3
\end{equation}
\begin{equation}
\label{Equ46} 
f(0,x,v) = \ave{\fin}(x,v),\;\;(x,v) \in \R ^3 \times \R^3
\end{equation}
where the averaged relaxation operator is given by
\[
\ave{\qb}f (x,v) = \frac{\oc ^2}{\tau} \intolxpvp{{\cal S} ( |\olv|, v_3, |\olvp|, \vtp,z) \{ M(v)f(\olxp, x_3, \vp) - M(\vp) f (x,v)\}}
\]
with $z = \oc \olx + \polv - (\oc \olxp + \polvp)$ and the averaged scattering cross section writes
\[
{\cal S} (r, v_3, \rp, \vtp, z) = \sigma(\sqrt{|z|^2 + (v_3 - \vtp)^2}\;) \;\chi (r, \rp, z)
\]
with 
\[
\chi (r, \rp, z)= \frac{\ind{|r - \rp| < |z| < r + \rp}}{\pi ^2\sqrt{|z|^2 - (r-\rp)^2} \sqrt{(r+\rp)^2 - |z|^2}},\;\;r, \rp \in \R_+, v_3, \vtp \in \R,\;\;z \in \R^2.
\]
\end{thm}
The averaging technique allows us to treat many different collision operators, for example the Fokker-Planck kernel (see Appendix \ref{FP} for details)
\[
\qfp (f) = \frac{\theta}{m \tau} \divv \left ( \nabla _v f + \frac{m}{\theta} v f \right ) = \frac{\theta}{m\tau} \divv \left \{M \nabla _v \left ( \frac{f}{M} \right )   \right \}.
\]
\begin{thm}
\label{MainResultFP}
The limit model when $\eps \searrow 0 $ of \eqref{Equ3}, \eqref{Equ2} with $Q = \qfp$ is given by
\begin{equation}
\label{EquLimModFP}
\partial _t f + \frac{\ave{\polE}}{B} \cdot \nabla _\olx f + v_3 \partial _{x_3} f + \frac{q}{m} \ave{E_3} \partial _{v_3} f = \ave{\qfp} f,\;\;(t,x,v) \in \R_+ \times \R ^3 \times \R ^3
\end{equation}
\begin{equation}
\label{EquLimModCI} 
f(0,x,v) = \ave{\fin}(x,v),\;\;(x,v) \in \R ^3 \times \R^3
\end{equation}
where the averaged \fp operator and the diffusion matrix ${\cal L}$ write
\[
\ave{\qfp}f (x,v) = \frac{\theta}{m \tau} \divoxv \left \{ M {\cal L} \nabla_{\omega _c x, v} \left ( \frac{f}{M} \right ) \right \}
\]
\[
{\cal L} = \left ( 
\begin{array}{lll}
2(I_3 - e_3 \otimes e_3)  &  -E\\
E  &\;\;\;   I_3
\end{array}
\right)
, \;\;
E = \left ( 
\begin{array}{rrrr}
0 & 1 & 0\\
-1& 0 & 0\\
0 & 0 & 0
\end{array}
\right).
\]
\end{thm}
Notice that the averaged \fp operator contains no derivatives with respect to $x_3$ since the diffusion matrix ${\cal L}$ has only zero entries on the third line and column; averaging the \fp operator leads to diffusion in velocity but also with respect to the perpendicular position coordinates.

Our main motivation concerns the bilinear \fpl equation, more exactly how to average kernels like
\[
\qfpl (f,f)(v) = \divv \left \{ \intvp{\sigma ( |v - \vp|) S(v - \vp) [f(\vp) \nabla _v f (v) - f(v) \nabla _{\vp} f (\vp)]}  \right \}
\]
where $\sigma$ denotes the scattering cross section and $S(w) = I - \frac{w\otimes w}{|w|^2}$ is the orthogonal projection on the plane of normal $w$, cf. \cite{HazMei03}. Recall that $\qfpl$ satisfies the mass, momentum and kinetic energy balances
\[
\intv{\qfpl (f, f)} = 0,\;\;\intv{v \qfpl (f, f)} = 0,\;\;\intv{\frac{|v|^2}{2}\qfpl (f, f)} = 0.
\]
Moreover it decreases the entropy $f \ln f$ since, by standard computations, we obtain
\begin{align*}
\intv{\!\!\ln f & \;\qfpl (f,f) } \\
& = - \frac{1}{2}\intvvp{\sigma  f(v) f (\vp) \frac{|(v - \vp) \wedge ( \nabla _v \ln f (v) - \nabla _{\vp}\ln f (\vp))|^2}{|v - \vp |^2}} \leq 0.
\end{align*}
We expect that the averaged \fpl operator satisfies the same properties. Nevertheless we will see that all of them hold true only globally in velocity and perpendicular position coordinates. Indeed, the averaged collision kernel will account for the interactions between Larmor circles (characterized by the center $\olx + \polv/\oc$ and the radius $|\olv|/|\oc|$) rather than between particles. We show that the averaged \fpl kernel has the form 
\begin{eqnarray}
\label{FPLExp}
\ave{\qfpl}(f,f):& = &  \ave{\qfpl (f,f)}(x,v) \nonumber \\
& = &  \divoxv \left \{\oc ^2 \intolxpvp{\sigma \chi  f(\olxp, x_3, \vp) A^+ \gradoxv f(x,v)}\right \}\nonumber \\
& -&  \divoxv \left \{\oc ^2 \intolxpvp{\sigma \chi f(x,v) A^- \gradoxpvp f(\olxp,x_3, \vp)}  \right \} \nonumber \\
& &
\end{eqnarray}
with $z = \oc \olx + \polv - (\oc \olxp + \polvp)$, $\sigma = \sigma (\sqrt{|z|^2 + (v_3 - \vtp)^2}\;)$, $\chi = \chi (|\olv|, |\olvp|, z)$  and
\[
\sigma \chi  A^+ (r, v_3, \rp, \vtp, z) = \sum _{i = 1} ^4 \xi ^i (\olx, v, \olxp, \vp) \otimes \xi ^i (\olx, v, \olxp, \vp)
\]
\[
\sigma \chi A^- (r, v_3, \rp, \vtp, z) = \sum _{i = 1} ^4 \eps _i \xi ^i (\olx, v, \olxp, \vp) \otimes \xi ^i (\olxp, \vp, \olx, v)
\]
for some vector fields $(\xi ^i)_{1\leq i \leq 4}$ and $\eps _1 = \eps _2 = - 1, \eps _3 = \eps _4 = 1$, see Proposition \ref{LFPKerT}. 
Actually $A^+, A^-$ have only zero entries on the third line and column and therefore, averaging the \fpl kernel leads to diffusion (and convolution) with respect to velocity but also perpendicular position coordinates. To the best of our knowledge, this is the first completely explicit result on this topic. In particular, the above collisional kernel decreases the entropy $f \ln f$ since, by standard computations we obtain (see Proposition \ref{ConserLaw})
\begin{align*}
\intolxv{\ln f & \ave{\qfpl}(f,f)  }   = - \frac{\oc ^2}{2}\sum _{i = 1} ^4  \intolxv{ \intolxpvp{ f f ^\prime \\
& \times \left (\xi ^i \cdot \nabla \ln f - \eps _i   (\xi ^i) ^\prime \cdot \nabla ^\prime \ln f ^\prime   \right )^2 }}\leq 0,\;x_3 \in \R.
\end{align*}
Here, for any $\xi, \eta \in \R ^6$, the notations $\xi \otimes \eta$ stands for the matrix whose entries are $(\xi \otimes \eta )_{kl} = \xi _k \eta _l$, $1\leq k, l \leq 6$. We obtain formally the following stability result
\begin{thm}
\label{MainResultLFP} Let us consider $\fin \geq 0 $, $( 1 + |\ln \fin |) \fin \in \loxv{}$ and denote by $\fe$ the solution of \eqref{Equ3}, \eqref{Equ2} with $Q = \qfpl$, for any $\eps >0$. Then the limit $f = \lime \fe$ satisfies
\begin{equation}
\label{Equ101}
\partial _t f + \frac{\ave{\polE}}{B} \cdot \nabla _\olx f + v_3 \partial _{x_3} f + \frac{q}{m} \ave{E_3} \partial _{v_3} f = \ave{\qfpl} (f,f),\;\;(t,x,v) \in \R_+ \times \R ^3 \times \R ^3
\end{equation}
\begin{equation}
\label{Equ102} 
f(0,x,v) = \ave{\fin}(x,v),\;\;(x,v) \in \R ^3 \times \R^3
\end{equation}
where the averaged \fpl operator is given by \eqref{FPLExp}.
\end{thm}

Our paper is organized as follows. In Section \ref{AveOpe} we introduce the average operator along a characteristic flow. Section \ref{AveDiff} is devoted to the commutation properties between average and first order differential operators. The average of the linear Boltzmann kernel is computed in Section \ref{Bol}. We establish its main properties and we prove the convergence result stated in Theorem \ref{MainResultBol}. Section \ref{fpl} is dedicated to the bilinear \fpl kernel, Theorem \ref{MainResultLFP}. We give a explicit form of its average and check the main physical properties. We prove the mass, momentum and total energy conservations for smooth solutions of the averaged \fpl equation coupled to the Poisson equation for the electric field. We also show that the mean Larmor circle center and power  (with respect to the origin) are left invariant. Up to our knowledge this has not been reported yet.

\section{Average operator}
\label{AveOpe}
We recall briefly the definition and properties of the average operator corresponding to the transport operator ${\cal T}$, whose definition in the $\ltxv{}$ setting is 
\[
{\cal T} u = \divxv (u \;b),\;\;b =  (\olv,0, \oc \polv, 0),\;\;\oc = \frac{qB}{m}
\]
for any function $u$ in the domain
\[
\D({\cal T}) = \{u(x,v)\in \ltxv{}\;:\;\divxv (u \;b) \in \ltxv{}\}.
\]
We denote by $\|\cdot\|$ the standard norm of $\ltxv{}$.The characteristics $(X,V)(s;x,v)$ associated to $\olv \cdot \nabla _\olx + \oc \polv \cdot \nabla _\olv$, see \eqref{Equ7}, satisfy 
\[
\frac{\md}{\md s} \left \{ \olX + \frac{\polV}{\oc} \right \} = 0,\;\;\frac{\md \olV}{\md s}= \oc \polV,\;\;\frac{\md X_3}{\md s } = 0,\;\;\frac{\md V_3}{\md s} = 0
\]
implying that
\[
\olV (s) = R(-\oc s) \olv,\;\;\olX  (s) = \olx + \frac{\polv}{\oc} - \frac{\polV (s)}{\oc},\;\;X_3(s) = x_3,\;\;V_3 (s) = v_3
\]
where $R(\alpha)$ stands for the rotation of angle $\alpha$
\begin{equation}\nonumber
R(\alpha) = \left (
\begin{array}{lll}
\cos \alpha  &  -\sin \alpha\\
\sin \alpha  &\;\;\;   \cos \alpha
\end{array}
\right ).
\end{equation}
All the trajectories are $T_c = 2\pi/\oc$ periodic and we introduce the average operator, see \cite{BosTraEquSin}, for any function $u \in \ltxv{}$
\begin{eqnarray}
\ave{u}(x,v) & = & \frac{1}{T_c} \int _0 ^{T_c} u (X(s;x,v), V(s;x,v))\;\md s \nonumber \\
& = & \frac{1}{2\pi} \int _0 ^{2\pi} u \left ( \olx + \frac{\polv}{\oc} - \frac{^\perp \{R(\alpha)\olv\}}{\oc}, x_3, R(\alpha)\olv, v_3  \right ) \;\md \alpha.\nonumber
\end{eqnarray}
It is convenient to introduce the notation $e ^{i\varphi}$ for the $\R^2$ vector $(\cos \varphi, \sin \varphi )$. Assume that the vector $\olv$ writes $\olv = |\olv| e ^{i \varphi}$. Then $R(\alpha) \olv =|\olv| e ^{i (\alpha + \varphi)}$ and the expression for $\ave{u}$ becomes
\begin{eqnarray}
\label{Equ22} \ave{u}(x,v) & = &  \frac{1}{2\pi} \int _0 ^{2\pi} u \left ( \olx + \frac{\polv}{\oc} - \frac{^\perp\{ |\olv| e ^{i (\alpha + \varphi)}\}}{\oc}, x_3, |\olv| e ^{i (\alpha + \varphi)}, v_3  \right ) \;\md \alpha \nonumber \\
& = & \frac{1}{2\pi} \int _0 ^{2\pi} u \left ( \olx + \frac{\polv}{\oc} - \frac{^\perp \{|\olv| e ^{i\alpha }\}}{\oc}, x_3, |\olv| e ^{i \alpha }, v_3  \right ) \;\md \alpha.
\end{eqnarray}
Notice that $\ave{u}$ depends only on the invariants $\olx + \frac{\polv}{\oc}, |\olv|, x_3, v_3$ and therefore belongs to $\ker {\cal T}$. The following two results are justified in \cite{BosGuiCen3D}, Propositions 2.1, 2.2. The first one states that averaging reduces to orthogonal projection onto the kernel of ${\cal T}$. The second one concerns the invertibility of ${\cal T}$ on the subspace of zero average functions and establishes a Poincar\'e inequality.
\begin{pro}
\label{AverageProp} The average operator is linear continuous. Moreover it coincides with the orthogonal projection on the kernel of ${\cal T}$ {\it i.e.,}
\begin{equation}
\label{EquVar}
\ave{u} \in \ker {\cal T}\;\mbox{ and } \;\intxv{(u - \ave{u}) \varphi } = 0,\;\;\forall \; \varphi \in \ker {\cal T}.
\end{equation}
\end{pro}
\begin{remark}
\label{Ave4D} Notice that $(\olX, \olV)$ depends only on $s$ and $(\olx, \olv)$ and thus the variational characterization in \eqref{EquVar} holds true at any fixed $(x_3, v_3) \in \R^2$. Indeed, for any $\varphi \in \ker {\cal T}$, $(x_3, v_3) \in \R^2$ we have
\begin{eqnarray}
\intolxolv{(u\varphi ) (x,v)} & = & \frac{1}{T_c} \int _0 ^{T_c} \intolxolv{u  (x,v) \varphi (\olX(-s;x,v),x_3, \olV(-s;x,v),v_3)}\md s\nonumber \\
& = &\frac{1}{T_c} \int _0 ^{T_c} \intolxolv{u (\olX(s;x,v),x_3, \olV(s;x,v),v_3) \varphi (x,v)} \md s \nonumber \\
& = & \intolxolv{\ave{u}(x,v)\varphi (x,v)}.\nonumber 
\end{eqnarray}
\end{remark}
We have the orthogonal decomposition of $\ltxv{}$ into invariant functions along the characteristics \eqref{Equ7} and zero average functions
\[
u = \ave{u} + (u - \ave{u}),\;\;\intxv{(u - \ave{u})\ave{u}} = 0.
\]
Notice that ${\cal T} ^\star = - {\cal T}$ and thus the equality $\ave{\cdot} = \mathrm{Proj}_{\ker {\cal T}}$ implies 
\[
\ker \ave{\cdot} = (\ker {\cal T})^\perp = (\ker {\cal T}^\star )^\perp = \overline{\ran {\cal T}}.
\]
In particular $\ran {\cal T} \subset \ker \ave{\cdot}$. Actually we show that $\ran {\cal T}$ is closed, which will give a solvability condition for ${\cal T} u = w$ (cf. \cite{BosGuiCen3D}, Propositions 2.2).
\begin{pro}
\label{TransportProp}
The restriction of ${\cal T}$ to $\ker \ave{\cdot}$ is one to one map onto $\ker \ave{\cdot}$. Its inverse belongs to ${\cal L }(\ker \ave{\cdot}, \ker \ave{\cdot} )$ and we have the Poincar\'e inequality
\[
\|u \| \leq \frac{2 \pi }{|\omega _c |} \|{\cal T} u \|,\;\;\omega _c= \frac{qB}{m}\neq 0
\]
for any $u \in \D({\cal T}) \cap \ker \ave{\cdot}$.
\end{pro}
The natural space when dealing with the linear Boltzmann kernel $\qb$  is $\ltxvm$ rather than $L^2(\md x \md v)$. Motivated by that we introduce the operator ${\cal T}_M : \mathrm{D}({\cal T}_M) \subset \ltxvm \to \ltxvm $ given by ${\cal T}_M u = \divxv (ub)$ for any function $u$ in the domain
\[
\mathrm{D}({\cal T}_M) = \{ u(x,v) \in \ltxvm \;:\; \divxv (ub) \in \ltxvm \}.
\]
Straightforward arguments show that $u \in \mathrm{D}({\cal T}_M)$ iff $u/\sqrt{M} \in \mathrm{D}({\cal T})$ and ${\cal T}_M (u) = \sqrt{M} {\cal T} (u/\sqrt{M})$ for any $u \in {\cal D}({\cal T}_M)$. In particular we have $\ker {\cal T}_M = \sqrt{M} \ker {\cal T}$. Notice that formula \eqref{Equ22} still defines a linear bounded operator on $\ltxvm$, denoted by $\ave{\cdot}_M$, which coincides with the orthogonal projection on the kernel of ${\cal T}_M$, with respect to the scalar product of $\ltxvm$. Indeed, taking into account that $M(v)$ is constant along the characteristic flow of \eqref{Equ7}, we have for any $u \in \ltxvm$
\[
\ave{u}_M = \sqrt{M} \ave{\frac{u}{\sqrt{M}}} \in \sqrt{M} \ker {\cal T} = \ker {\cal T}_M
\]
and for any $\varphi \in \ker {\cal T}_M$
\[
\intxvm{(u - \ave{u}_M) \;\varphi (x,v)}= \intxv{\left (\frac{u}{\sqrt{M}} - \ave{\frac{u}{\sqrt{M}}}    \right )\frac{\varphi}{\sqrt{M}}} = 0.
\]
The Poincar\'e inequality holds also true, with the same constant, since for any $u \in \mathrm{D}({\cal T}_M) \cap \ker \ave{\cdot}_M$ we can write
\[
\|u \|_{L^2(M^{-1})} = \left \|\frac{u}{\sqrt{M}}  \right \| \leq \frac{2\pi}{|\oc |} \left \|{\cal T } \left (\frac{u}{\sqrt{M}}   \right )   \right \| = \frac{2\pi}{|\oc |} \left \|\frac{{\cal T}_M u }{\sqrt{M}}  \right \| = \frac{2\pi}{|\oc |} \| {\cal T}_M u \|_{L^2(M^{-1})}.
\]
From now on, for the sake of simplicity, we will use only the notations ${\cal T}, \ave{\cdot}$, independently of acting on $L^2(\md x \md v)$ or $\ltxvm$.

\section{Average and first order differential operators}
\label{AveDiff}
We intend to average transport operators, see \eqref{Equ9}. Moreover, in order to handle the \fpl kernel we will need to average second order differential operators. For doing that it is convenient to identify derivations which leave invariant $\ker {\cal T}$. It turns out that these derivations are those along the invariants
\[
\psi _1 = x_1 + \frac{v_2}{\oc},\;\;\psi _2 = x_2 - \frac{v_1}{\oc},\;\;\psi _3 = x_3,\;\;\psi _4 = \sqrt{(v_1)^2 + (v_2)^2},\;\;\psi _5 = v_3.
\]
We introduce also $\psi _0 = - \frac{\alpha}{\oc}$, with $\olv = |\olv|e ^{i \alpha}$, $\alpha \in [0,2\pi[$. Notice that $\psi _0$ has a jump of $\frac{2\pi}{\oc}$ across $\olv \in \R _+ \times \{0\}$ but not its gradient with respect to $\olv$
\[
\nabla _\olv \alpha = - \frac{\polv}{|\olv |^2},\;\;\nabla _\olv \psi _0 = \frac{\polv}{\oc |\olv|^2},\;\;{\cal T} \psi _0 = 1. 
\]
The idea is to consider the fields $(b^i)_{0\leq i \leq 5}$ such that 
\[
b ^i \cdot \nabla _{x,v} \psi _j = \delta ^i _j,\;\;0 \leq i, j \leq 5.
\]
Indeed, the map $(x,v) \to (\psi _i (x,v))_{0 \leq i \leq 5}$ defines a change of coordinates
\[
x_1 = \psi _1 + \frac{\psi _4}{\oc}\sin (\oc \psi _0),\;\;x_2 = \psi _2 + \frac{\psi _4}{\oc}\cos (\oc \psi _0),\;\;x_3 = \psi _3
\]
\[
v_1 = \psi _4 \cos (\oc \psi _0),\;\;v_2 = - \psi _4 \sin (\oc \psi _0),\;\;v_3 = \psi _5.
\]
Therefore any function $u = u(x,v)$ can be written $u(x,v) = U(\psi (x,v))$, $\psi = (\psi _i)_{0 \leq i \leq 5}$ and thus, for any $i \in \{0,1,...,5\}$ we have
\[
b^i \cdot \nabla _{x,v} u  =  b^i \cdot \sum _{j = 0} ^5 \frac{\partial U }{\partial \psi _j} (\psi (x,v)) \nabla _{x,v} \psi _j = \frac{\partial U}{\partial \psi _i} (\psi (x,v)).
\]
In other words the derivations $b^i \cdot \nabla _{x,v}$ act like $\partial _{\psi _i}, 0 \leq i \leq 5$. In particular if $u \in \ker {\cal T}$, meaning that $U$ does not depend on $\psi _0$, then $b ^i \cdot \nabla _{x,v} u  = \partial _{\psi _i} U (\psi (x,v))$ does not depend on $\psi _0$, saying that $\ker {\cal T}$ is left invariant by $b^i \cdot \nabla _{x,v}$, $0 \leq i \leq 5$. The following result comes by direct computation and is left to the reader. 
For any smooth vector fields $\xi, \eta$ on $\R^6$, the notation $[\xi, \eta]$ stands for their Poisson bracket {\it i.e.,}
\[
[\xi, \eta] = ( \xi \cdot \nabla _{x,v} ) \eta - (\eta \cdot \nabla _{x,v} ) \xi.
\]
\begin{pro}
\label{Field}
The fields $(b^i)_{0 \leq i \leq 5}$ satisfying $b^i \cdot \nabla _{x,v} \psi _j = \delta ^i _j$, $0\leq i, j \leq 5$ are given by
\[
b^0 \cdot \nabla _{x,v} = \olv \cdot \nabla _\olx  + \oc \polv \cdot \nabla _\olv,\;\;b^1 \cdot \nabla _{x,v} = \partial _{x_1},\;\;b^2 \cdot \nabla _{x,v} = \partial _{x_2},\;\;b^3 \cdot \nabla _{x,v} = \partial _{x_3}
\]
\[
b^4 \cdot \nabla _{x,v} = - \frac{\polv}{\oc |\olv|} \cdot \nabla _\olx + \frac{\olv}{|\olv|} \cdot \nabla _\olv,\;\;b^5 \cdot \nabla _{x,v} = \partial _{v_3}. 
\]
Moreover the Poisson brackets between $(b^i)_{0\leq i \leq 5}$ vanishes or equivalently the derivations $b^i \cdot \nabla _{x,v}$, $0 \leq i \leq 5$ are commuting.
\end{pro}
\begin{remark}
\label{Div}
Notice that $(b^i)_{i\neq 4}$ are divergence free and $\divxv b^4 = \frac{1}{|\olv|}$.
\end{remark}
We claim that the operators $u \to \divxv(ub^i)$, with domain
\[
\D (\divxv(\cdot \;b^i)) = \{ u \in \ltxv{}\;:\; \divxv(u b^i) \in \ltxv\},\;\;0 \leq i \leq 5
\]
are commuting with the average operator. More generally we establish the following result.
\begin{pro}
\label{ComDerAve} Assume that the field $c \cdot \nabla _{x,v}$ is in involution with $b \cdot \nabla _{x,v} = \olv \cdot \nabla _\olx  + \oc \polv \cdot \nabla _\olv$ {\it i.e.,} $[c,b] = 0$. Then the operator $\divxv(\cdot \; c)$ is commuting with the average operator associated to the flow of $b\cdot \nabla _{x,v}$ that is, for any function $u \in \D (\divxv(\cdot \; c ))$ its average $\ave{u}$ belongs to $\D (\divxv(\cdot \;c ))$ and
\[
\divxv(\ave{u}c) = \ave{\divxv(uc)}.
\]
\end{pro}
\begin{proof}
Let us consider $u \in \mathrm{D}(\divxv (\cdot \;c))$. 
For any $\varphi \in \cocxv{} \cap \ker {\cal T}$ we have
\begin{equation}
\label{Equ26}
\intxv{\ave{\divxv(uc)}\varphi } = \intxv{\divxv(uc) \varphi} = - \intxv{u c\cdot \nabla _{x,v} \varphi }.
\end{equation}
But ${\cal T}(c \cdot \nabla _{x,v} \varphi) = c\cdot \nabla _{x,v} ({\cal T} \varphi ) = 0$ saying that $c\cdot \nabla _{x,v} \varphi \in \ker {\cal T}$ and thus
\begin{equation}
\label{Equ27} \intxv{u c \cdot \nabla _{x,v} \varphi } = \intxv{\ave{u}c \cdot \nabla _{x,v} \varphi }.
\end{equation}
Combining \eqref{Equ26}, \eqref{Equ27} we obtain for any $\varphi \in \cocxv{} \cap \ker {\cal T}$
\begin{equation}
\label{Equ28} \intxv{\ave{\divxv(uc)}\varphi} = - \intxv{\ave{u}c \cdot \nabla _{x,v} \varphi }.
\end{equation}
Actually the previous equality holds also true for smooth functions $\varphi \in \ker \ave{\cdot}$. Indeed, by Proposition \ref{TransportProp}, for any smooth function $\varphi \in \ker \ave{\cdot}$ there is $\psi \in \D ({\cal T}) \cap \ker \ave{\cdot}$ such that ${\cal T} \psi = \varphi $ and thus $ c \cdot \nabla _{x,v} \varphi = c\cdot \nabla _{x,v} ( {\cal T} \psi ) = {\cal T}(c \cdot \nabla _{x,v} \psi ) \in \ran {\cal T} = \ker \ave{\cdot}$. Using now the orthogonality between $\ker {\cal T}$ and $\ker \ave{\cdot}$ we deduce that 
\[
\intxv{\ave{\divxv(uc)}\varphi }= 0 = - \intxv{\ave{u} c \cdot \nabla _{x,v} \varphi},\;\;\varphi \in \cocxv{} \cap \ker \ave{\cdot}.
\]
Finally \eqref{Equ28} is verified for any smooth $\varphi$, implying that
\[
\ave{u} \in \D ( \divxv(\cdot \;c))\;\mbox{ and }\;\divxv(\ave{u}c) = \ave{\divxv(uc)}.
\]
\end{proof}
We want to average transport operators, which are written in conservative forms. In order to obtain averaged model still written in conservative form, it is worth to establish the following commutation formula between average and divergence. For the sake of simplicity we discard all difficulties related to the required minimal smoothness.
\begin{pro}
\label{ComDivAve} For any smooth field $\xi = (\xi_x, \xi _v)\in \R^6$ we have the equality
\begin{eqnarray}
\ave{\divxv \xi} & = & \divolx  \left \{
\ave{\xi _{\olx} + \frac{^\perp \xi _\olv}{\oc}}
 + \ave{\xi _\olv \cdot \frac{\polv}{|\olv|}} 
 \frac{\olv}{\oc |\olv|} -  \ave{\xi _\olv \cdot \frac{\olv}{|\olv|}} 
 \frac{\polv}{\oc |\olv|} \right \}+ \partial _{x_3} \ave{\xi _{x_3}} \nonumber \\
& + & \divolv \left \{\ave{\xi _\olv \cdot \frac{\polv}{|\olv|}} 
 \frac{\polv}{|\olv|} +  \ave{\xi _\olv \cdot \frac{\olv}{|\olv|}} 
 \frac{\olv}{ |\olv|}    \right \}+ \partial _{v_3} \ave{\xi _{v_3}}. \nonumber
\end{eqnarray}
In particular we have for any smooth field $\xi _x \in \R^3$
\[
\ave{\divx \xi _x }= \divx \ave{\xi _x}
\]
and for any smooth field $\xi _v \in \R^3$
\begin{eqnarray}
\ave{\divv \xi _v } & = & \divolx  \left \{
\ave{ \frac{^\perp \xi _\olv}{\oc}}
 + \ave{\xi _\olv \cdot \frac{\polv}{|\olv|}} 
 \frac{\olv}{\oc |\olv|} -  \ave{\xi _\olv \cdot \frac{\olv}{|\olv|}} 
 \frac{\polv}{\oc |\olv|} \right \} \nonumber \\
& + & \divolv \left \{\ave{\xi _\olv \cdot \frac{\polv}{|\olv|}} 
 \frac{\polv}{|\olv|} +  \ave{\xi _\olv \cdot \frac{\olv}{|\olv|}} 
 \frac{\olv}{ |\olv|}    \right \}+ \partial _{v_3} \ave{\xi _{v_3}}. \nonumber
\end{eqnarray}
\end{pro}
\begin{proof}
By construction we have $
\sum _{i = 0} ^5 b^i \otimes \gradxv \psi _i  = I
$
and thus 
\[
\xi = \sum _{i = 0} ^5 ( \xi \cdot \gradxv \psi _i )b^i.
\]
The main statement follows thanks to Proposition \ref{ComDerAve}, since we have
\[
\ave{\divxv \xi } = \ave{\sum _{i = 0} ^5 \divxv \{ ( \xi \cdot \gradxv \psi _i ) b^i \}} = \divxv \left \{ \sum _{i = 0} ^5 \ave{\xi \cdot \gradxv \psi _i } b^i \right \}.
\]
The other statements come by considering the fields $(\xi_x,0)$ and $(0,\xi_v)$.
\end{proof}
A direct consequence of Proposition \ref{ComDivAve} is the computation of the average for the transport operator in \eqref{Equ6}.
\begin{pro}
\label{AveTra} Assume that the electric field derives from a smooth potential {\it i.e.,} $E = - \nabla _x \phi$. Then for any $f \in \cocxv{} \cap \ker {\cal T}$ we have
\[
\ave{\partial _t f + v_3 \partial _{x_3} f + \frac{q}{m} E \cdot \nabla _v f + {\cal T}\fo } = \partial _t f + \frac{\ave{\polE}}{B} \cdot \nabla _\olx f + v_3 \partial _{x_3} f + \frac{q}{m} \ave{E_3} \partial _{v_3} f.
\]
\end{pro}
\begin{proof}
We can write 
\[
\ave{\partial _t f + v_3 \partial _{x_3} f + \frac{q}{m} E \cdot \nabla _v f  + {\cal T} f^1} =\partial _t f + \ave{v_3 \partial _{x_3} f }+ \frac{q}{m} \ave{E \cdot \nabla _v f } 
\]
since $\ave{{\cal T}\fo } = 0$ and $\ave{\partial _t f} = \partial _t \ave{f} = \partial _t f $. The average of $v_3 \partial _{x_3}f$ comes easily thanks to Proposition \ref{ComDerAve}
\[
\ave{v_3 \partial _{x_3} f} = \ave{\divxv \{fv_3 b^3\}} = \divxv\{\ave{fv_3}b^3\} = \divxv\{fv_3 b^3\} = v_3 \partial _{x_3} f.
\]
Observe that ${\cal T}(f \phi) = f \;\olv \cdot \gradolx \phi = - f \;\olv \cdot \olE$ and thus $\ave{f \;\olv \cdot \olE } = 0$. Thanks to Proposition \ref{ComDivAve} one gets
\begin{align*}
\ave{\divv \{fE\}} & = \divolx \ave{f \frac{\polE}{\oc}} + {\cal T} \ave{f \frac{\polv \cdot \olE}{\oc |\olv|^2}} + \partial _{v_3} \ave{f E_3}\\
& = \divolx \left \{f \ave{\frac{\polE}{\oc}}   \right \} + \partial _{v_3} \{f \ave{E_3}  \}
\end{align*}
implying that 
\[
\frac{q}{m} \ave{\divv \{f E\}} = \divolx \left \{f \frac{\ave{\polE}}{B}   \right \} + \frac{q}{m} \partial _{v_3} \{ f \ave{E_3}\}.
\]
Using again Proposition \ref{ComDerAve} notice that 
\[
\partial _{v_3} \ave{E_3} = \divxv \{ \ave{E_3} b^5\} = \ave{\divxv \{E_3 b^5\}} = \ave{\partial _{v_3} E_3} = 0
\]
and
\[
\divolx \ave{\polE} = \ave{\divolx \polE} = 0
\]
and our statement follows.
\end{proof}
\begin{remark}
\label{AveTraCom} We have proved that averaging the transport operator $a \cdot \gradxv := v_3 \partial _{x_3} + \frac{q}{m} E \cdot \nabla _v $ leads to $A\cdot \gradxv : = \frac{\ave{\polE}}{B} \cdot \gradolx + v _3 \partial _{x_3} + \frac{q}{m} \ave{E_3} \partial _{v_3}$ which verifies
\[
\ave{a \cdot \gradxv f } = A \cdot \gradxv f,\;\;f \in \cocxv{} \cap \ker {\cal T}.
\]
By construction, the operator $A \cdot \gradxv$ leaves invariant the subspace of smooth functions of $\ker {\cal T}$. By antisymmetry (since $\divxv A = 0$) it is easily seen that $A \cdot \gradxv $ also leaves invariant the subspace of smooth functions in $\ker \ave{\cdot}$. Indeed, consider $h$ a zero average smooth function and let us prove that $\ave{A \cdot \gradxv h } = 0$ For any smooth $f$ in $\ker {\cal T}$ we have
\[
\intxv{A \cdot \gradxv h \;f } = - \intxv{h A \cdot \gradxv f} = 0
\]
by the orthogonality between $\ker \ave{\cdot}$ and $\ker {\cal T}$, and thus $\ave{A \cdot \gradxv h} = 0$. Finally $A \cdot \gradxv $ is commuting with the average operator $\ave{A\cdot \gradxv f} = A \cdot \gradxv \ave{f}$ for any smooth $f$.
\end{remark}

\section{The relaxation collision operator}
\label{Bol}

In this section we analyze the linear Boltzmann collision kernel \cite{Pou92, PouSch91}
\begin{equation}
\label{Equ30} \qb (f) (x,v) = \frac{1}{\tau} \intvp{\svvp \{ M(v) f(x,\vp) - M(\vp) f(x,v)\} }
\end{equation}
where the scattering cross section satisfies
\begin{equation}
\label{Equ31} \svvp = s(\vp,v),\;\;0 < s_0 \leq \svvp \leq S_0 <+\infty,\;\;v, \vp \in \R^3.
\end{equation}
We recall the standard properties of this operator. Here $Q_B ^\pm $ denote the gain/loss relaxation collision operators
\[
\qbp (f) (v)= \frac{1}{\tau} \intvp{\svvp M(v) f (\vp)},\;\;\qbm (f) (v)= \frac{1}{\tau} \intvp{\svvp M(\vp) f (v)}.
\]
\begin{pro}
\label{LinBolOpe} Assume that the scattering cross section satisfies \eqref{Equ31}. Then \\
1. The gain/loss collision operators $Q_B ^\pm$
are linear bounded operators on $\ltvm$, with $\|Q_B ^ \pm\| \leq S_0 /\tau$, and symmetric with respect to the scalar product of $\ltvm$.\\
2. For any $f \in \ltvm $ we have
\[
\intvm{\qb (f) (v) f(v) }= - \frac{1}{2\tau} \intv{\intvp{\svvp M(v) M(\vp) \left [\frac{f(v)}{M(v)} - \frac{f(\vp)}{M(\vp)}   \right ]^2}}\leq 0.
\]
\end{pro}
We want to average $\qb (f)$ for functions $f$ satisfying the constraint \eqref{Equ5}. In this section the operators ${\cal T}$ and $\ave{\cdot}$ should be understood in the $\ltxvm$ framework. We need to compute the average of functions like $\intvp{C(v,\vp)f(x,\vp)}$ where $C(v,\vp)$ is a given function. The corresponding result in the bidimensional framework has been announced in \cite{CRASBosCal12}. We will see that we only need to consider functions $C$ which are left invariant by any rotation around $e_3 = (0,0,1)$. Therefore we assume that for any orthogonal matrix ${\cal O} \in {\cal M}_3 (\R)$ such that ${\cal O}e_3 = e_3$ we have
\begin{equation}
\label{Equ32} C(^t{\cal O} v, {^t{\cal O}} \vp) = C(v,\vp),\;\;v, \vp \in \R^3.
\end{equation}
These functions are precisely those depending only on $|\olv|, v_3, |\olvp|, \vtp$ and the angle between $\olv$ and $\olvp$
\[
C(v,\vp) = \tC (|\olv|, v_3, |\olvp|, \vtp, \varphi),\;\;\varphi = \arg \olvp - \arg \olv.
\]
\begin{pro}
\label{FirstFormula} Assume that the function $C(v,\vp)$ satisfies \eqref{Equ32} and belongs to the space $L^2(M^{-1}(v)M(\vp)\md v \md \vp)$. Then for any function $f \in \ker {\cal T}$ we have
\begin{equation}
\label{EquVariant}
\ave{\intvp{C(v,\vp)f(x,\vp)}} (x,v) = \oc ^2 \intolxpvp{{\cal C}( |\olv|, v_3, |\olvp|, \vtp, z)f(\olxp, x_3, \vp)}
\end{equation}
where $z = \oc \olx + \polv - (\oc \olxp + \polvp )$
\[
{\cal C} (r, v_3, \rp, \vtp, z) = \frac{\tC (r, v_3, \rp, \vtp, \varphi) + \tC (r, v_3, \rp, \vtp, -\varphi)}{2\pi ^2 \sqrt{|z|^2 - (r - \rp)^2}\sqrt{(r+\rp )^2 - |z|^2}}\;\ind{|r - \rp| < |z|< r + \rp}
\]
and for any $|z| \in (|r - \rp|, r + \rp)$, $\varphi \in (0,\pi)$ is the unique angle such that 
\[
|z|^2 = r^2 + (\rp)^2 - 2 r \rp \cos \varphi.
\]
\end{pro}
\begin{proof}
For $(x,v) \in \R ^3 \times \R ^3$ we have
\[
\left ( \intvp{C(v,\vp) f(x,\vp)} \right )^2 \leq \intvp{(C(v,\vp))^2 M(\vp)} \intvp{\frac{(f(x,\vp))^2}{M(\vp)}}
\]
implying that
\[
\left \| \intvp{C(v,\vp) f(x,\vp)}  \right \|_{\ltxvm} \leq \|C\|_{L^2(M^{-1}(v)M(\vp)\md v \md \vp)} \|f \|_{\ltxvm}.
\]
Therefore the function $(x,v) \to \intvp{C(v,\vp) f(x,\vp)}$ belongs to $\ltxvm{}$ and can be averaged in $\ltxvm{}$.
Consider $(x,v) \in \R ^3 \times \R ^3$. By formula \eqref{Equ22} we have
\begin{align*}
I:& = \ave{\intvp{C(v,\vp)f(x,\vp)}}(x,v) \\
& =\avetpi \intvp{C(|\olv|e ^{i\alpha}, v_3, \vp) f \left (\olx + \frac{\polv}{\oc} - \frac{^\perp \{|\olv| e ^{i\alpha}\}}{\oc}, x_3, \vp   \right ) }\md \alpha.
\end{align*}
For any fixed $\alpha \in [0,2\pi)$ we use the cylindrical coordinates
\[
\vp = ( \rp e ^{i(\varphi + \alpha)}, \vtp),\;\;\rp \in \R_+,\;\;\varphi \in [-\pi, \pi)
\]
and therefore
\begin{align}
I & = \avetpi \int _\R \int _{-\pi} ^\pi \int _{\R_+} C(|\olv| e ^{i\alpha}, v_3, \rp e ^{i(\varphi + \alpha)}, \vtp) \nonumber \\
& \times f \left (\olx + \frac{\polv}{\oc} - \frac{^\perp \{|\olv| e ^{i\alpha}\}}{\oc}, x_3, \rp e ^{i (\varphi + \alpha)},\vtp   \right )\rp \md \rp \md \varphi \md \vtp \md \alpha.
\end{align}
But $f \in \ker {\cal T}$ and thus there is $g$ such that
\begin{equation}
\label{EquRedPhaseSpace}
f(x,v) = g\left ( \olx + \frac{\polv}{\oc}, x_3, |\olv|, v_3\right )
\end{equation}
implying that 
\[
f \left (\olx + \frac{\polv}{\oc} - \frac{^\perp \{|\olv| e ^{i\alpha}\}}{\oc}, x_3, \rp e ^{i (\varphi + \alpha)},\vtp   \right ) = g \left (\olx + \frac{\polv}{\oc} - \frac{^\perp \{ |\olv| e ^{i \alpha}\}}{\oc} + \frac{^\perp \{\rp  e ^{i (\varphi + \alpha)}  \}}{\oc}, x_3, \rp, \vtp   \right ).
\]
By one hand notice that $\rp  e ^{i (\varphi + \alpha)} - |\olv| e ^{i \alpha} = l  e ^{i (\psi + \alpha)}$ where $l^2 = r ^2 + (\rp )^2 - 2 r \rp \cos \varphi $, $r = |\olv|$ and $\psi$ depends on $r, \rp, \varphi$ but not on $\alpha$. By the other hand, since $C$ is invariant by rotation around $e_3$ we deduce that 
\[
C(re ^{i\alpha}, v_3, \rp  e ^{i (\varphi + \alpha)}, \vtp) = \tC (r, v_3, \rp, \vtp, \varphi),\;\;\varphi = \arg \olvp - \arg \olv.
\]
The map $\varphi \to l(\varphi) = \sqrt{r ^2 + (\rp )^2 - 2 r \rp \cos \varphi}$ defines a coordinate change between $\varphi \in (0,\pi)$ and $l \in (|r - \rp|, r + \rp)$ and
\[
\md \varphi = \frac{2l\md l}{\sqrt{l^2 - (r - \rp)^2} \sqrt{(r + \rp )^2 - l^2}}.
\]
By Fubini theorem one gets
\begin{align*}
I & = \frac{1}{2\pi}\int _{\R}\int _{-\pi}^\pi \int _{\R_+} \inttpi \tC (r, v_3, \rp, \vtp, \varphi) g \left (\olx + \frac{\polv}{\oc} + \frac{^\perp \{l e ^{i (\psi + \alpha)}\}}{\oc}, x_3, \rp, \vtp\right )\rp\;\md \alpha \md \rp \md \varphi \md \vtp \\
& = \frac{1}{2\pi}\int _{\R}\int _{-\pi}^\pi \int _{\R_+} \inttpi \tC (r, v_3, \rp, \vtp, \varphi) g \left (\olx + \frac{\polv}{\oc} + \frac{^\perp \{l e ^{i  \alpha}\}}{\oc}, x_3, \rp, \vtp\right )\rp\;\md \alpha \md \rp \md \varphi \md \vtp \\
& = \frac{1}{2\pi}\int _{\R}\int _{\R_+} \inttpi \int _0 ^{\pi} \{\tC (r, v_3, \rp, \vtp, \varphi) + \tC(r, v_3, \rp, \vtp, -\varphi) \}\\
& \times  g \left (\olx + \frac{\polv}{\oc} + \frac{^\perp \{l(\varphi) e ^{i \alpha}\}}{\oc}, x_3, \rp, \vtp\right )\rp\;\md \varphi \md \alpha \md \rp \md \vtp \\
& = \frac{1}{2\pi}\int _{\R}\int _{\R_+} \inttpi \int _{|r - \rp|}^ {r + \rp} \{\tC (r, v_3, \rp, \vtp, \varphi(l)) + \tC(r, v_3, \rp, \vtp, -\varphi (l)) \}\\
& \times  g \left (\olx + \frac{\polv}{\oc} + \frac{^\perp \{l e ^{i \alpha}\}}{\oc}, x_3, \rp, \vtp\right )\frac{2l\rp\;\md l\md \alpha \md \rp  \md \vtp}{\sqrt{l^2 - (r - \rp)^2} \sqrt{(r+\rp)^2 - l^2}}.
\end{align*}
For any $\alphap \in [0,2\pi)$ we have
\[
g \left (\olx + \frac{\polv}{\oc} + \frac{^\perp \{l e ^{i \alpha}\}}{\oc}, x_3, \rp, \vtp\right ) = f \left (\olx + \frac{\polv}{\oc} + \frac{^\perp \{l e ^{i \alpha}\}}{\oc} - \frac{^\perp \{\rp e ^{i \alphap}\}}{\oc}, x_3, \rp e ^{i \alphap}, \vtp   \right )
\]
and performing the change of coordinates $\vp = (\rp e ^{i \alphap}, \vtp)$ leads to
\begin{align*}
I & = \frac{1}{2\pi ^2} \inttpi \int _{\R_+} \int _{\R} \inttpi \int _{\R_+} \{\tC (r, v_3, \rp, \vtp, \varphi(l)) + \tC(r, v_3, \rp, \vtp, -\varphi (l)) \} \\
& \times f \left (\olx + \frac{\polv}{\oc} + \frac{^\perp \{l e ^{i \alpha}\}}{\oc} - \frac{^\perp \{\rp e ^{i \alphap}\}}{\oc}, x_3, \rp e ^{i \alphap}, \vtp   \right )\frac{\ind{|r - \rp| < l < r + \rp}\rp \md \rp \md \alphap \md \vtp l \md l \md \alpha}{\sqrt{l^2 - (r - \rp)^2} \sqrt{(r+\rp)^2 - l^2}}\\
& = \frac{1}{2 \pi ^2} \inttpi \int _{\R_+} \int _{\R^3} \{\tC (r, v_3, |\olvp|, \vtp, \varphi(l)) + \tC(r, v_3, |\olvp|, \vtp, -\varphi (l))   \} \\
& \times f \left (\olx + \frac{\polv}{\oc} + \frac{^\perp \{l e ^{i \alpha}\}}{\oc} - \frac{\polvp}{\oc}, x_3, \vp   \right )\frac{\ind{|\;|\olv| - |\olvp|\;| < l < |\olv| + |\olvp|} \md \vp  l \md l \md \alpha}{\sqrt{l^2 - (|\olv| - |\olvp|)^2} \sqrt{(|\olv| + |\olvp|)^2 - l^2}}.
\end{align*}
Using the notation 
\[
{\cal C} (r, v_3, \rp, \vtp, z) = \frac{\tC (r, v_3, \rp, \vtp, \varphi) + \tC (r, v_3, \rp, \vtp, -\varphi)}{2\pi ^2 \sqrt{|z|^2 - (r - \rp)^2}\sqrt{(r+\rp )^2 - |z|^2}}\;\ind{|r - \rp| < |z|< r + \rp}
\]
where for any $|z| \in (|r - \rp|, r + \rp)$, $\varphi \in (0,\pi)$ is the unique angle such that 
\[
|z|^2 = r^2 + (\rp)^2 - 2 r \rp \cos \varphi
\]
one gets
\[
I = \int _{\R ^3} \inttpi \int _{\R_+} {\cal C}( |\olv|, v_3, |\olvp|, \vtp, - ^\perp \{l e ^{i\alpha}\}) f \left (\olx + \frac{\polv}{\oc} + \frac{^\perp \{l e ^{i \alpha}\}}{\oc} - \frac{\polvp}{\oc}, x_3, \vp   \right )\;l \md l \md \alpha \md \vp.
\]
We take as new coordinates 
\[
\olxp = \olx + \frac{\polv}{\oc} + \frac{^\perp \{ l e ^{i \alpha}\}}{\oc} - \frac{\polvp}{\oc}
\]
Observing that $\det \frac{\partial (x_1 ^\prime, x_2 ^ \prime)}{\partial (l, \alpha)} = \frac{l}{\oc ^2}$ we deduce that
\[
I = \oc ^2 \intolxpvp{{\cal C}( |\olv|, v_3, |\olvp|, \vtp, (\oc \olx + \polv) - (\oc \olxp + \polvp))f(x_1^\prime, x_2 ^\prime, x_3, \vp)}.
\]
\end{proof}
\begin{remark}
The constraint ${\cal T} f = 0$ allows us to reduce the right hand side of \eqref{EquVariant} to a four dimensional integral. Indeed, thanks to \eqref{EquRedPhaseSpace} we obtain, using the notation $\oly = \olx + \frac{\polv}{\omega _c}$
\begin{align*}
I & = \oc ^2 \intolxpvp{{\cal C}( |\olv|, v_3, |\olvp|, \vtp, (\oc \olx + \polv) - (\oc \olxp + \polvp))\;g\left (\olxp + \frac{\polvp}{\omega _c}, x_3, |\olvp|, \vtp\right )} \\
& = \omega _c ^2 \intolypvp{{\cal C}(|\olv|, v_3, |\olvp|, \vtp,\omega _c (\oly - \olyp)) \;g (\olyp, x_3, |\olvp|, \vtp) } \\
& = 2\pi \omega _c ^2 \int _{\R^2} \int _\R \int _{\R_+} {\cal C}(|\olv|, v_3, r^{\prime}, \vtp,\omega _c (\oly - \olyp)) \;g (\olyp, x_3, r^{\prime}, \vtp) \;r^{\prime} \md r ^{\prime} \md \vtp \md \olyp.
\end{align*}
We prefer to keep the five dimensional integral representation since, in the sequel, we will introduce similar integral terms, but with densities $f$ not satisfying the constraint ${\cal T } f = 0$.
\end{remark}
\begin{remark}
\label{EvenOdd} If the function $\tC(r, v_3, \rp, \vtp, \varphi)$ is odd with respect to $\varphi$, then ${\cal C}= 0$ and
\[
\ave{\intvp{C(v,\vp) f (x, \vp)}} = 0,\;\;f \in \ker {\cal T}.
\]
\end{remark}
\begin{remark}
\label{Normalize} Let $\chi$ be the function
\[
\chi (r, \rp, z) = \frac{\ind{|r - \rp| < |z|< r + \rp}}{\pi ^2 \sqrt{|z|^2 - (r - \rp)^2}\sqrt{(r+\rp )^2 - |z|^2}}
\]
for any $r, \rp \in \R_+, z \in \R^2$. Then for every $r, \rp \in \R_+$, $\chi (r,\rp, z)\md z$ is a probability measure on $\R^2$
\[
\int_{\R ^2} \chi (r, \rp, z)\;\md z = 1,\;\;r,\rp \in \R_+
\]
and $\ave{\intvp{C(v,\vp) f(x,\vp)}}$ appears as a convolution with respect to the invariants $\oc \olx +\polv$. Indeed, using the formula
\[
f(\olxp, x_3, \vp) = g \left ( \olxp + \frac{\polvp}{\oc}, x_3, |\olvp|, \vtp\right)
\]
we obtain
\begin{align*}
\ave{\intvp{C(v,\vp) f(x,\vp)}}(x,v)  & = \int _{\R^3}\;\;\int _{\R^2} {\cal C} (|\olv|, v_3, |\olvp|, \vtp, (\oc \olx + \polv) - (\oc \olxp + \polvp)) \\
& \times g \left ( \olxp + \frac{\polvp}{\oc}, x_3, |\olvp|, \vtp\right)\;\md (\oc \olxp + \polvp)\md \vp.
\end{align*}
\end{remark}
\begin{remark}
\label{LinfAve} The conclusion of Proposition \ref{FirstFormula} also holds true for bounded functions $f$ which are constant along the flow \eqref{Equ7}, provided that $C(v,\vp) \in L^\infty(\md v; L^1(\md \vp))$ and satisfies \eqref{Equ32}. Indeed, in this case $f \to \intvp{C(v,\vp) f(x,\vp)}$ is bounded on $L^\infty (\md x\md v)$
\[
\left \| \intvp{C(v,\vp) f(x,\vp)}  \right \|_{L^\infty (\md x\md v)} \leq \|C\|_{L^\infty(\md v; L^1(\md \vp))} \;\|f \|_{L^\infty (\md x\md v)}
\]
and using the $L^\infty$ version of the average operator, the same computations as those in the proof of Proposition \ref{FirstFormula} show that
\[
\ave{\intvp{C(v,\vp) f(x,\vp)}}(x,v) = \oc ^2 \intolxpvp{{\cal C} (|\olv|, v_3, |\olvp|, \vtp, z) f (\olxp, x_3, \vp)}.
\]
\end{remark}
\begin{coro}
\label{RelBolAve} Assume that the scattering cross section satisfies \eqref{Equ31} and 
\begin{equation}
\label{Equ35}
\svvp = \sigma (|v - \vp|),\;\;v,\vp \in \R^3
\end{equation}
for some function $\sigma : \R_+ \to \R_+$. Then for any $f \in \ker {\cal T}$ we have
\[
\ave{\qb f } (x,v)  = \frac{\oc ^2}{\tau} \intolxpvp{{\cal S} ( |\olv|, v_3, |\olvp|, \vtp,z) \{ M(v)f(\olxp, x_3, \vp) - M(\vp) f (x,v)\}}
\] 
with $z = \oc \olx + \polv - (\oc \olxp + \polvp)$ and
\[
{\cal S} (r, v_3, \rp, \vtp,z) = \sigma (\sqrt{|z|^2 + (v_3 - \vtp)^2}\;) \chi (r, \rp, z).
\]
\end{coro}
\begin{proof}
Clearly the function $C(v,\vp) = \sigma (|v - \vp|) M(v)$ satisfies \eqref{Equ32}, belongs to $L^2(M^{-1}(v)M(\vp)\md v \md \vp)$ and we have
\[
\tilde{s} (r, v_3, \rp, \vtp, \varphi) = \sigma ( \sqrt{r ^2 + (\rp )^2 - 2 r \rp \cos \varphi + (v_3 - \vtp)^2}\;)
\]
\[
{\cal S} (r, v_3, \rp, \vtp, z) = \sigma (\sqrt{|z|^2 + (v_3 - \vtp)^2}\;) \chi (r, \rp, z).
\]
Thanks to Proposition \ref{FirstFormula} we obtain, with $z = (\oc \olx + \polv) - (\oc \olxp + \polvp)$
\[
\ave{\intvp{\svvp M(v)f (x,\vp)}} = \oc ^2 \intolxpvp{{\cal S}(|\olv|, v_3, |\olvp|, \vtp, z) M(v) f(\olxp, x_3, \vp)}.
\]
Since $f$ belongs to $\ltxvm$ and remains constant along the flow \eqref{Equ7} we have
\[
\ave{\intvp{\svvp M(\vp) f (x,v)}} = f(x,v) \ave{\intvp{\svvp M(\vp)}}
\]
where the first average operator should be understood in the $\ltxvm{}$ setting and the second one in the $L^\infty(\md x \md v)$ setting. Remark \ref{LinfAve} applied with $C(v,\vp) = \svvp M(\vp) \in L^\infty(\md v; L^1(\md \vp))$ and the constant function $1 \in L^\infty(\md x \md v)$ yields
\[
\ave{\intvp{\svvp M(\vp)}} = \oc ^2 \intolxpvp{{\cal S}(|\olv|, v_3, |\olvp|, \vtp, z)M(\vp)}.
\]
Therefore we obtain 
\[
\ave{\intvp{\svvp M(\vp)f(x,v) }} = \oc ^2 \intolxpvp{{\cal S}(|\olv|, v_3, |\olvp|, \vtp, z)M(\vp) f(x,v)}
\]
and our statement follows immediately.
\end{proof}
We intend to extend the previous averaged collision operator to all densities $f$, not only those satisfying the constraint ${\cal T} f = 0$. Think that, when simulating numerically these models, the particle density may not satisfy exactly ${\cal T} f = 0$, and thus we need to construct such a extension. One possibility consists to appeal to the decomposition $f = \ave{f} + (f - \ave{f})$ and to neglect the fluctuation $f - \ave{f}$, leading to the operator
\begin{align*}
f \to & \tilde{Q}_Bf  := \ave{\qb \ave{f}}\\
& = \frac{\oc ^2}{\tau} \intolxpvp{{\cal S} ( |\olv|, v_3, |\olvp|, \vtp,z) \{ M(v)\ave{f}(\olxp, x_3, \vp) - M(\vp) \ave{f} (x,v)\}}
\end{align*}
for any $f \in \ltxvm{}$. Clearly $\tilde{Q}_B$ coincides with $\ave{\qb f }$ for any $f \in \ker {\cal T}$. Notice that for any $(x,v), (x_3 ^\prime, \vtp)$ the function
\[
(\olxp,\olvp) \to {\cal S} (|\olv|, v_3, |\olvp|, \vtp, \oc \olx + \polv - (\oc \olxp + \polvp)) M(v)
\]
depends only on $\oc \olxp + \polvp, |\olvp|$ and therefore, thanks to Remark \ref{Ave4D}, we obtain a simpler form
\[
\tilde{Q}_B f = \frac{\oc ^2}{\tau} \intolxpvp{{\cal S} ( |\olv|, v_3, |\olvp|, \vtp,z) \{ M(v){f}(\olxp, x_3, \vp) - M(\vp) \ave{f} (x,v)\}}.
\]
Nevertheless notice that it is not possible to remove the average in the loss part of the previous operator. Since $Q_B$ and $\ave{\cdot}$ are linear bounded operators on $\ltxvm{}$ we deduce that $\tilde{Q}_B$ is linear bounded on $\ltxvm{}$. The properties of $\tilde{Q}_B$ come immediately from the properties of $\qb$, cf. Proposition \ref{LinBolOpe}. For example we have for any $f \in \ltxvm{}$
\begin{align*}
\intxv{\tilde{Q}_B & f \frac{f}{M}}  = \intxv{\ave{Q_B \ave{f}} \frac{f}{M}} = \intxv{Q_B \ave{f} \frac{\ave{f}}{M}} \\
& = - \frac{1}{2\tau} \intxv{\intvp{\svvp M(v) M(\vp) \left [\frac{\ave{f}(x,v)}{M(v)} - \frac{\ave{f}(x,\vp)}{M(\vp)}  \right]^2}\!\!}\leq 0.
\end{align*}
Another possible extension is given by
\begin{equation}
\label{Equ37} \ave{Q_B}f := \frac{\oc ^2}{\tau} \intolxpvp{{\cal S} ( |\olv|, v_3, |\olvp|, \vtp,z) \{ M(v){f}(\olxp, x_3, \vp) - M(\vp) {f} (x,v)\}}
\end{equation}
for any $f \in \ltxvm{}$, which is very similar to the operator $\qb$ in \eqref{Equ30}. We keep this operator as extension for the operator in Corollary \ref{RelBolAve}. The properties of $\ave{\qb}$ are summarized below
\begin{pro}
\label{ProBolAve} Assume that the scattering cross section satisfies \eqref{Equ31}, \eqref{Equ35}. Then\\
1. The operator $\ave{Q_B}$ is linear bounded on $\ltxvm{}$ and symmetric with respect to the scalar product of $\ltxvm{}$.\\
2. For any $f \in \ltxvm{}$ we have
\begin{align*}
\intxv{\ave{\qb}(f) \frac{f}{M}} & = - \frac{\oc^2}{2\tau} \intxv{\intolxpvp{{\cal S}( |\olv|, v_3, |\olvp|, \vtp,z)  M(v) M(\vp) \\
& \times \left [\frac{f(x,v)}{M(v)} - \frac{f(\olxp, x_3,\vp)}{M(\vp)} \right ] ^2}\!}\leq 0.
\end{align*}
\end{pro}
\begin{proof}
1. The boundedness of the loss part follows easily since it is the multiplication by the bounded function (see Remark \ref{Normalize})
\begin{align*}
\frac{\oc ^2}{\tau} \intolxpvp{{\cal S}    ( |\olv|, v_3, |\olvp|, \vtp,z)M(\vp)} & = \frac{1}{\tau} \intvp{\!\int _{\R^2} {\cal S}( |\olv|, v_3, |\olvp|, \vtp, - z^\prime)M(\vp)\;\md z ^\prime\!\!}\\
& \leq \frac{S_0}{\tau}.
\end{align*}
For the gain part we use the inequalities
\begin{align*}
& \oc ^2 \left (\intolxpvp{{\cal S}  ( |\olv|, v_3, |\olvp|, \vtp,z)f (\olxp, x_3, \vp)    }\right )^2  \\
& \leq \intolxpvp{ {\cal S} \frac{f^2 (\olxp,x_3, \vp)}{M(\vp)} } \;\;\oc ^2 \intolxpvp{{\cal S} M(\vp)} \\
& \leq S_0 \intolxpvp{ {\cal S} \frac{f^2 (\olxp,x_3, \vp)}{M(\vp)} }.
\end{align*}
Thanks to Remark \ref{Normalize} we deduce, with $\ltm = \ltxvm{}$,  that 
\begin{align*}
& \left \|\frac{\oc ^2}{\tau} \intolxpvp{{\cal S} M(v) f(\olxp, x_3, \vp)}    \right \|^2 _\ltm  \leq \frac{\oc ^2}{\tau ^2} \intxv{M(v) S_0\\
& \times \intolxpvp{{\cal S} \frac{f^2 (\olxp,x_3, \vp)}{M(\vp)}}\!}\\
& = \frac{S_0}{\tau ^2} \intv{M(v) \int_{\R^3} \!\int _{\R^3} \frac{f^2 (\olxp, x_3, \vp)}{M(\vp)}\;\oc ^2 \int _{\R^2} {\cal S} \;\md x_1\md x_2 \md \vp \md x_1 ^\prime \md x_2 ^\prime \md x_3    }\\
& \leq \frac{S_0 ^2}{\tau ^2} \int _{\R^3} \!\int _{\R^3} \frac{f^2 (x_1 ^\prime, x_2 ^\prime, x_3, \vp)}{M(\vp)} \;\md \vp \md x_1 ^\prime \md x_2 ^\prime \md x_3 = \frac{S_0 ^2}{\tau ^2} \|f \| ^2 _{\ltm}.
\end{align*}
2. Interchanging $(\olxp, \vp)$ with $(\olx, v)$ and observing that this change leaves invariant ${\cal S}$, yield for any $f, g \in \ltxvm{}$
\begin{align*}
& ( \ave{\qb}f, g)_\ltm \\
&  = \frac{\oc ^2}{\tau} \intxv{\intolxpvp{{\cal S} M(v) M(\vp)\left [\frac{f(\olxp,x_3,\vp)}{M(\vp)} - \frac{f(x,v)}{M(v)}     \right ]}\frac{g(x,v)}{M(v)}}\\
&  = - \frac{\oc ^2}{\tau} \int_{\R^3}\!\int _{\R^3}{\intolxv{{\cal S} M(v) M(\vp)\left [\frac{f(\olxp,x_3,\vp)}{M(\vp)} - \frac{f(x,v)}{M(v)}     \right ]}} \\
& \times \frac{g(\olxp, x_3,\vp)}{M(\vp)}\md \vp \md x_1 ^\prime \md x_2 ^ \prime \md x_3\\
& = -\frac{\oc ^2}{2 \tau} \intxv{\intolxpvp{{\cal S} M(v) M(\vp)\left [\frac{f(\olxp,x_3,\vp)}{M(\vp)} - \frac{f(x,v)}{M(v)}     \right ]\left [\frac{g(\olxp,x_3,\vp)}{M(\vp)} - \frac{g(x,v)}{M(v)}     \right ]\\
& }\!}
\end{align*}
which justifies the symmetry of $\ave{\qb}$ and its negativity.
\end{proof}
\begin{remark}
\label{NonLocal} Contrary to $\qb$, the operator $\ave{\qb}$ is non local in space. The value of $\ave{\qb}f $ at the point $(x,v)$ depends on all the values of $f$ in the set 
\begin{align*}
A(x,v) & = \{ (x_1 ^\prime, x_2 ^\prime, x_3, \vp)\;:\; {\cal S} (|\olv|,v_3, |\olvp|, \vtp, (\oc \olx + \polv) - (\oc \olxp + \polvp)\;) >0\} \\
& = \{  (x_1 ^\prime, x_2 ^\prime, x_3, \vp)\;:\; |\;| \olv| - |\olvp| \;| < |(\oc \olx + \polv) - (\oc \olxp + \polvp)| < |\olv| + |\olvp|  \}.
\end{align*}
Observe that if we denote by $C_{x,v}$ the Larmor circle
\[
C_{x,v} = \{ (x_1 ^\prime, x_2 ^\prime, x_3)\;:\; |\;\oc \olxp - ( \oc \olx + \polv) \; | = |\olv|\;\}
\]
then we have
\[
C_{x,v} \times \{ \vp\;:\;\vp \in \R ^3\} \subset \overline{A(x,v)}
\]
where $\overline{X}$ stands for the adherence of $X$ in $\R^6$. In particular $ \{ x \} \times \R ^3  \subset \overline{A(x,v)}$.
\end{remark}
\begin{remark}
\label{GlobalMass} The gain/loss parts of $\ave{\qb}$ are bounded on $\loxv{}$ and for any $f \in \loxv{}$ we have the global balance of the mass $\intxv{\ave{\qb}f } = 0$.
Indeed, we have
\begin{align*}
\|\ave{\qb}^+ f \|_{L^1} & \leq \frac{\oc ^2}{\tau} \intxv{\intolxpvp{{\cal S}( |\olv|, v_3, |\olvp|, \vtp,z) M(v) |f (\olxp, x_3, \vp)|}\!} \\
& \leq \frac{S_0}{\tau} \int _{\R^3}\!\int _{\R ^3} \!\int _{\R^3} M(v) |f (\olxp, x_3, \vp)| \;\md \vp\md x_1 ^\prime \md x_2 ^\prime \md x_3 \md v \nonumber \\
& = \frac{S_0}{\tau} \|f \|_{L^1}
\end{align*}
and similarly 
\[
\|\ave{\qb} ^- f \|_{L^1} \leq \frac{S_0}{\tau} \|f \|_{L^1}.
\]
The global balance follows by interchanging $(\olxp, \vp) $ with $(\olx, v)$.
\end{remark}
Combining \eqref{Equ6}, Propositions \ref{AveTra}, \ref{RelBolAve} and \eqref{Equ37} yields the limit model in Theorem \ref{MainResultBol}.
\begin{proof} (of Theorem \ref{MainResultBol})
Clearly $0 \leq \fe \in \linftloxv{}$ and 
\[
\intxv{\fe (t,x,v)} = \intxv{\fin (x,v)},\;\;t \in \R_+, \eps >0.
\]
We consider a sequence $(\eps _k)_k \subset \R_+ ^\star$ converging to $0$ such that $\limk \fek = f$ weakly $\star$ in $\linftltxvm{}$. Using the weak formulation of \eqref{Equ3}, \eqref{Equ2} with test functions $\eta (t) \varphi (x,v)$, $\eta \in C^1 _c (\R_+), \varphi \in \cocxv{}$, we deduce, after multiplication by $\eps _k$ and letting $k \to \infty$, that the limit density satisfies the constraint 
\begin{equation}
\label{Equ39} {\cal T} f(t) = 0,\;\;t \in \R_+.
\end{equation}
Considering test functions like $\eta (t) \varphi (x,v)$ with $\eta \in C^1 _c (\R_+)$, $\varphi \in \cocxv{} \cap \ker {\cal T}$ one gets
\begin{align}
\label{Equ40}
\inttxv{\fek \{ \eta ^\prime \varphi + \eta v_3 \partial _{x_3} \varphi + \eta \frac{q}{m} E (x) \cdot \nabla _v \varphi \}} & + \intxv{\fin \eta (0) \varphi } \nonumber \\
& = - \inttxv{\eta \qb ( \fek) \varphi }.
\end{align}
The symmetry of $\qb$ cf. Proposition \ref{LinBolOpe} allows us to write 
\begin{align}
\label{Equ41}
\limk \inttxv{\eta & \qb (\fek)\varphi }  = \limk \inttxvm{\eta \fek \qb (\varphi M)} \nonumber \\
& = \inttxvm{\eta f \qb (\varphi M)}  = \inttxv{\eta \qb (f) \varphi }\nonumber \\
& = \inttxv{\eta \ave{\qb (f)}\varphi } = \inttxv{\eta \ave{\qb} (f) \varphi }
\end{align}
since $f(t) \in \ker {\cal T}, t \in \R_+$ and thus $\ave{\qb (f)} = \ave{\qb} (f)$. For the other terms in \eqref{Equ40} we obtain thanks to Proposition \ref{AveTra}, Remark \ref{AveTraCom}
\begin{align}
\label{Equ42}
\limk  \inttxv{\fek  ( \partial _t  + a \cdot \gradxv ) (\eta \varphi)  } & =  \inttxv{f ( \partial _t  + a \cdot \gradxv )(\eta \varphi ) } \nonumber \\
& = \inttxv{f ( \partial _t  + A \cdot \gradxv )( \eta \varphi ) } 
\end{align}
and
\begin{equation}
\label{Equ43} \intxv{\fin (x,v)\eta (0) \varphi (x,v)} = \intxv{\eta (0) \ave{\fin}(x,v) \varphi (x,v)}.
\end{equation}
Combining \eqref{Equ40}, \eqref{Equ41}, \eqref{Equ42}, \eqref{Equ43} yields for any smooth test function $\eta (t) \varphi (x,v)$ with $\varphi \in \ker {\cal T}$. 
\begin{align}
\label{Equ44} \inttxv{f ( \partial _t  + A \cdot \gradxv )( \eta \varphi ) } & + \intxv{\ave{\fin} \eta (0) \varphi (x,v)} \nonumber \\
& = - \inttxv{\ave{\qb} (f) \eta \varphi }.
\end{align}
By Remark \ref{AveTraCom} we know that $A \cdot \gradxv $ leaves invariant the subspace of zero average functions and therefore it is easily seen that \eqref{Equ44} is trivially satisfied for any test function $\eta (t) \psi (x,v)$, with $\psi \in \cocxv{} \cap \ker \ave{\cdot}$. Finally \eqref{Equ44} holds true for any smooth test function, saying that $f$ solves \eqref{Equ45}, \eqref{Equ46}. We are done if we prove the uniqueness for the solution of \eqref{Equ45}, \eqref{Equ46} (and in this case all the family $(\fe)_\eps$ will converge towards $f$, weakly $\star$ in $\linftltxvm{}$). Assume that $f$ solves \eqref{Equ45} with zero initial condition. By standard arguments one gets
\[
\partial _t |f| + \frac{\ave{\polE}}{B}\cdot \gradolx |f| + v_3 \partial _{x_3} |f| + \frac{q}{m} \ave{E_3} \partial _{v_3} |f| = \ave{\qb} (f) \;\sgn f
\]
implying that 
\[
\frac{\md }{\md t} \intxv{|f(t,x,v)| } = \intxv{\ave{\qb}(f)\; \sgn f(t,x,v)},\;\;t \in \R_+.
\]
Our conclusion comes by observing that 
\begin{align*}
& \intxv{ \ave{\qb} (f) \;\sgn f }  \\
& = \frac{\oc ^2 }{\tau} \intxv{\intolxpvp{{\cal S} \{M(v) f(t,\olxp, x_3, \vp) - M(\vp) f (t,x,v)     \}\sgn f (t,x,v)}\!} \\
& = \frac{\oc ^2 }{\tau} \intxv{\intolxpvp{{\cal S}M(\vp) \{ f(t,x,v) \sgn f(t,\olxp, x_3, \vp) - |f(t,x,v)| \}}}\leq 0.
\end{align*}
\end{proof}
\begin{remark}
\label{Propagation}
It is easily seen that the integro-differential operator in \eqref{Equ45} propagates the constraint ${\cal T} f = 0$. We are done if we prove that $f_s = f$ for any $s \in \R$, where $f_s (t,x,v) = f(t, X(s;x,v), V(s;x,v))$ and $(X,V)$ is the characteristic flow \eqref{Equ7}. A direct computation shows that ${\cal T}$ and $A \cdot \nabla _{x,v} = \frac{\ave{\polE}}{B} \cdot \nabla _\olx + v_3 \partial _{x_3} + \frac{q}{m} \ave{E_3} \partial _{v_3}$ commute, implying that
\[
A\cdot \nabla _{x,v} f_s = (A\cdot \nabla _{x,v} f )_s.
\]
Observe also that 
\[
\ave{\qb}^+ f_s = \ave{\qb} ^+ f = ( \ave{\qb} ^+ f)_s,\;\;\ave{\qb} ^- f_s = (\ave{\qb} ^- f )_s 
\]
and therefore $\ave{\qb}f_s = (\ave{\qb} f)_s$. Finally both $f, f_s$ satisfy \eqref{Equ45}, \eqref{Equ46} and our statement follows by the uniqueness that we have established before.

Clearly the transport equation \eqref{Equ45} can be written in the reduced phase space $(\oly = \olx + \frac{\polv}{\omega _c}, x_3, r = |\olv|, v_3)$ since, by the constraint ${\cal T} f = 0$, we know that $f(t,x,v) = g(t, \oly, x_3, r, v_3)$. We obtain
\begin{align*}
\partial _t g + \frac{\ave{\polE}}{B} \cdot \nabla _\oly g + v_3 \partial _{x_3} g & + \frac{q}{m} \ave{E_3} \partial _{v_3} g  = 2\pi \frac{\omega _c ^2}{\tau} \int _{\R^2} \int _\R \int _{\R_+} {\cal S} (r, v_3, \rp, \vtp, \omega _c (\oly - \olyp)) \\
& \times \{ M \; g(t, \olyp, x_3, \rp, \vtp) - M^\prime \;g (t, \oly, x_3, r, v_3)\}\;\rp \md \rp \md \vtp \md \olyp 
\end{align*}
where
\[
M = \frac{1}{(2\pi \theta/m)^{3/2}} e ^ {- \frac{m}{2\theta} ( r^2 + (v_3)^2)},\;\;M^\prime = \frac{1}{(2\pi \theta/m)^{3/2}} e ^ {- \frac{m}{2\theta} ( (\rp) ^2 + (\vtp)^2)}.
\]
\end{remark}
\begin{remark}
\label{ConfinPot}
The family $(\fe )_{0 < \eps \leq 1}$ remains bounded in $\linftltxvm{}$ for potentials of the form $\phi (x) = \overline{\phi}(\olx) + \phi _3 (x_3)$. Indeed, in this case observe that the energy function $W^\eps (x,v)  : = \frac{m|v|^2}{2} + q ( \eps \;\overline{\phi}(\olx) + \phi _3 (x_3))$ satisfies
\[
\partial _t W^\eps + \frac{1}{\eps} ( \olv \cdot \gradolx + \oc \polv \cdot \gradolv ) W^\eps + v_3 \partial _{x_3} W^\eps + \frac{q}{m} E(x) \cdot \nabla _v W^\eps = 0
\]
and therefore one gets
\begin{align*}
\left \{\partial _t  + \frac{1}{\eps} ( \olv \cdot \gradolx + \oc \polv \cdot \gradolv )  + v_3 \partial _{x_3}   + \frac{q}{m} E(x) \cdot \nabla _v  \right \}& \left ( \frac{(\fe ) ^2 }{2 M(v) \exp( {- \frac{q}{\theta} [ \eps \; \overline{\phi} ( \olx) + \phi _3( x_3)]})}\right ) \\
& = \frac{\qb (\fe) \fe }{ M(v) \exp( {- \frac{q}{\theta} [\eps \;\overline{\phi} ( \olx) + \phi _3( x_3)]})}.
\end{align*}
Integrating with respect to $(x,v)$ yields the bound
\[
\intxv{\frac{(\fe (t,x,v) )^2 }{ M(v) \exp( {- \frac{q}{\theta}[\eps \;\overline{\phi} ( \olx) + \phi _3( x_3)] })}}\leq \intxv{\frac{(\fin (x,v) )^2 }{ M(v) \exp( {- \frac{q}{\theta} [\eps \;\overline{\phi} ( \olx) + \phi _3( x_3)]})}}
\]
implying the uniform estimate
\[
\|\fe \|_{\linftltxvm{}} \leq \exp \left ( \frac{|q|\;(\|\overline{\phi}\|_{L^\infty} + \|\phi _3 \|_{L^\infty})}{\theta}  \right ) \|\fin \|_{\ltxvm{}},\;0 < \eps \leq 1.
\]
\end{remark}

\section{The \fpl operator}
\label{fpl}

In this section we focus on the \fpl equation \cite{CerIllPul, DevVilOne, DevVilTwo, Dev04}. The rate of change of the density $\fs$, corresponding to a population of charged particles of specie $s$, due to collisions with charge particles of specie $s^\prime$ writes
\begin{align*}
\qfpl (\fs, \fsp) & =  \frac{1}{m_s}  \divv \intvp{\musi \\
& \times S(v-\vp) \left (
\frac{1}{m_s} \fsp (\vp) (\nabla _v \fs )(v) - \frac{1}{m_{s^\prime}} \fs (v) (\nabla _{\vp} \fsp )(\vp)
\right )}
\end{align*}
where $\mu _{s s ^\prime} = \frac{m_s m_{s^\prime}}{m_s + m_{s^\prime}}$ is the reduced mass of the pair $\{m_s, m_{s^\prime}\}$, $\sigma _{s s^\prime} = \sigma _{s^\prime s}>0$ is the scattering cross section between species $\{s, s ^\prime\}$ and the matrix $S(w) = \imww$ corresponds to the orthogonal projection on the plane orthogonal to $w$. As the electron mass is much smaller than the ion mass, we consider only the collisions between ions, whose distribution function is denoted by $\fe$ and satisfies
\begin{equation}
\label{EquPert}
\partial _t \fe + \frac{1}{\eps} ( \olv \cdot \gradolx + \oc \polv \cdot \gradolv ) \fe + v _3 \partial _{x_3} \fe + \frac{q}{m} E \cdot \nabla _v \fe = \qfpl (\fe, \fe),\;(t,x,v) \in \R_+ \times \R ^3 \times \R ^3
\end{equation}
where $q >0$ is the ion charge and $m$ is the ion mass. As in the relaxation case, the limit model comes by averaging the collision kernel $\qfpl$. The treatment of the \fpl kernel is much elaborated. Therefore we content ourselves of formal computations. 
The main properties of the \fpl operator are summarized below
\begin{pro}
\label{ProLFP} Consider the \fpl kernel between ions
\[
\qfpl (f, f)  =  \frac{1}{4}  \divv \intvp{ \sigma _{ii} (|v - \vp|) |v - \vp |^3  S(v-\vp) \left (
 f (\vp) (\nabla _v f )(v) - f (v) (\nabla _{\vp} f )(\vp)
\right )}.
\]
Then the mass, momentum and kinetic energy balances hold true
\[
\intv{m \qfpl (f,f)} = 0,\;\;\intv{m v \qfpl (f,f) } = 0,\;\;\intv{m \frac{|v|^2}{2} \qfpl (f,f) } = 0.
\]
Moreover the entropy production $D:= - \intv{(1 + \ln f ) \qfpl (f,f)}$ is non negative
\[
D = \frac{1}{8} \intvvp{\sigma (|v - \vp|) |v- \vp| f(v) f(\vp) | (v - \vp) \wedge ( \nabla _v \ln f(v) - \nabla _{\vp} \ln f (\vp) ) |^2 } \geq 0.
\]
\end{pro}
\begin{proof}
All statements come easily by integration by parts, observing that $A_{ii} (v,\vp) + A_{ii} (\vp,v) = 0$, where $A_{ii} (v,\vp) = \sigma _{ii} (|v - \vp|) |v - \vp |^3 S(v - \vp) ( f (\vp) \nabla _v f (v) - f (v) \nabla _{\vp} f (\vp))$ and $S(v - \vp) (v-\vp) = 0$.
\end{proof}
With the notation $\sigma (|v - \vp|) = \frac{1}{4}\sigma _{ii}(|v - \vp|) |v - \vp |^3$ the collision kernel becomes
\[
\qfpl (f, f) (v) =    \divv \intvp{ \sigma  (|v - \vp|) 
S(v-\vp) \left (f(\vp)  (\nabla _v f )(v) -  f (v) (\nabla _{\vp} f )(\vp)\right )}.
\]

\subsection{Preliminary computations}
The \fpl operator combines convolution and differential operators in $v$. Therefore its average can be determined by studying the commutation properties between convolution and derivation with respect to the average. First we apply the commutation formula between divergence and average. Next we are looking for commutation between convolution and average. It is convenient to split $\qfpl$ into its gain and loss parts $\qfpl ^\pm$. We introduce the following notations, for any function $g$ and vectors $w_1, w_2$
\[
\ave{g}_{\sigma S} := \ave{\intvp{g(x, \vp) \sigma (|v - \vp|) S(v- \vp)}}
\]
\[
\ave{g, w_1}_{\sigma S} := \ave{\intvp{g(x, \vp) \sigma (|v - \vp|) S(v- \vp)w_1}}
\]
\[
\ave{g, w_1, w_2}_{\sigma S} := \ave{\intvp{g(x, \vp) \sigma (|v - \vp|) S(v- \vp): w_1\otimes w_2}}.
\]
Let us establish some useful formulae based on Proposition \ref{FirstFormula}. For any orthogonal matrix ${\cal O} \in {\cal M}_3 (\R)$ we consider the application $(v,\vp) \to S(^t {\cal O} v - {^t {\cal O}} \vp)$. It is easily seen that 
\[
S(^t {\cal O} v - {^t {\cal O}} \vp) = {^t {\cal O}} S(v - \vp) {\cal O}.
\]
Notice also that for any orthogonal matrix ${\cal O} \in {\cal M}_3 (\R)$ such that ${\cal O} e_3 = e_3$ we have $( \overline{\tO v}, 0) = \tO (\olv,0)$ and $({^\perp \overline{\tO v}}, 0) = \tO (\polv,0)$ for any $v \in \R^3$. 
\begin{lemma}
The following applications are left invariant by any rotation around $e_3$, that is they satisfy \eqref{Equ32}
\[
S(v-\vp) : (\olv,0) \otimes (\olv,0),\;\;S(v-\vp) : (\olv,0) \otimes (\polv,0),\;\;S(v-\vp) : (\polv,0) \otimes (\polv,0)
\]
\[
S : (\olvp,0) \otimes (\olv,0),\;S : (\olvp,0) \otimes (\polv,0),\;S : (\polvp,0) \otimes (\olv,0),\;S : (\polvp,0) \otimes (\polv,0).
\]
\end{lemma}
\begin{proof}
For any rotation ${\cal O}$ around $e_3$ we have
\begin{align*}
S(\tO v - \tO \vp) : (\overline{\tO v},0) \otimes (\overline{\tO v},0) & = \tO S(v - \vp) {\cal O} : \tO (\olv,0) \otimes \tO (\olv,0) \\
& = 
S(v-\vp) : (\olv,0) \otimes (\olv,0).
\end{align*}
The other invariances follow similarly.
\end{proof}
Therefore the formula in Proposition \ref{FirstFormula} applies to all previous functions and we obtain
\begin{pro}
\label{AveSca} For any $z \in \R ^2$ such that $|r - \rp| < |z| < r + \rp$ we denote by $\varphi \in (0,2\pi)$ the angle satisfying $r^2 + (\rp )^2 - 2 r \rp \cos \varphi = |z|^2$. Then for any function $f \in \ker {\cal T}$ we have, with the notation $z = (\oc \olx + \polv ) - (\oc \olxp + \polvp)$
\begin{enumerate}
\item 
\begin{align*}
\ave{f, ( \olv,0), (\olv, 0)}_{\sigma S} & = \oc ^2 \intolxpvp{\sigma( \sqrt{ |z|^2 + (v_3 - \vtp)^2 }\;) f(\olxp, x_3, \vp) \chi (|\olv|, |\olvp|, z) \\
& \times 
\left \{ r^2 - \frac{r^2 ( r - \rp \cos \varphi )^2}{|z|^2 + (v_3 - \vtp)^2}  \right \}}
\end{align*}
\item 
\begin{align*}
\ave{f, ( \olv,0), (\polv, 0)}_{\sigma S} = \ave{f, ( \polv,0), (\olv, 0)}_{\sigma S} = 0
\end{align*}
\item
\begin{align*}
\ave{f, ( \polv,0), (\polv, 0)}_{\sigma S} & = \oc ^2 \intolxpvp{\sigma( \sqrt{ |z|^2 + (v_3 - \vtp)^2 }\;) f(\olxp, x_3, \vp) \chi (|\olv|, |\olvp|, z) \\
& \times 
\left \{ r^2 - \frac{r^2 (\rp)^2 \sin ^2 \varphi }{|z|^2 + (v_3 - \vtp)^2}  \right \}}
\end{align*}
\item
\begin{align*}
\ave{f, ( \olvp,0), (\olv, 0)}_{\sigma S} & = \oc ^2 \intolxpvp{\sigma( \sqrt{ |z|^2 + (v_3 - \vtp)^2 }\;) f(\olxp, x_3, \vp) \chi (|\olv|, |\olvp|, z) \\
& \times 
\left \{ r \rp \cos \varphi  - \frac{r \rp  (r \cos \varphi - \rp)(r - \rp \cos \varphi )  }{|z|^2 + (v_3 - \vtp)^2}  \right \}}
\end{align*}
\item 
\begin{align*}
\ave{f, ( \olvp,0), (\polv, 0)}_{\sigma S} = \ave{f, ( \polvp,0), (\olv, 0)}_{\sigma S} = 0
\end{align*}
\item
\begin{align*}
\ave{f, ( \polvp,0), (\polv, 0)}_{\sigma S} & = \oc ^2 \intolxpvp{\sigma( \sqrt{ |z|^2 + (v_3 - \vtp)^2 }\;) f(\olxp, x_3, \vp) \chi (|\olv|, |\olvp|, z) \\
& \times 
\left \{ r \rp \cos \varphi  - \frac{r^2(\rp)^2  \sin ^2 \varphi   }{|z|^2 + (v_3 - \vtp)^2}  \right \}}.
\end{align*}
\end{enumerate}
\end{pro}
\begin{proof}
We need to compute the functions $\tC, {\cal C}$ defined in Proposition \ref{FirstFormula}. In each case we have\\
1.
\[
C(v,\vp) = \sigma (|v - \vp|)S(v - \vp) : (\olv, 0) \otimes (\olv, 0) = \sigma (|v - \vp|)\left \{ |\olv|^2 - \frac{[(\olv - \olvp) \cdot \olv ]^2}{|v - \vp|^2}   \right \}
\]
\[
\tC (r, v_3, \rp, \vtp, \varphi ) = \sigma( \sqrt{r ^2 + (\rp )^2 - 2 r \rp \cos \varphi + (v_3 - \vtp)^2}\;)\left \{ r^2 - \frac{ ( r^2 - r \rp \cos \varphi )^2}{|z|^2 + (v_3 - \vtp)^2}  \right \}
\]
\[
{\cal C} (r, v_3, \rp, \vtp, z) = \chi (r, \rp, z) \sigma(\sqrt{|z|^2 + (v_3 - \vtp)^2}\;)\left \{ r^2 - \frac{ r^2( r -  \rp \cos \varphi )^2}{|z|^2 + (v_3 - \vtp)^2}  \right \}.
\]
2.
\begin{align*}
C(v,\vp) & = \sigma (|v - \vp|)S(v - \vp) : (\olv, 0) \otimes (\polv, 0) \\
& = - \sigma (|v - \vp|) \frac{[(\olv - \olvp) \cdot \olv ]\;[(\olv - \olvp) \cdot \polv]}{|v - \vp|^2}   
\end{align*}
\[
\tC (r, v_3, \rp, \vtp, \varphi ) = - \sigma( \sqrt{|z|^2+ (v_3 - \vtp)^2}\;)\frac{ ( r^2 - r \rp \cos \varphi ) r \rp \sin \varphi }{|z|^2 + (v_3 - \vtp)^2}  
\]
\[
{\cal C} = 0.
\]
3.
\[
C(v,\vp) = \sigma (|v - \vp|)S(v - \vp) : (\polv, 0) \otimes (\polv, 0) = \sigma (|v - \vp|)\left \{ |\olv|^2 - \frac{(\olvp \cdot \polv )^2}{|v - \vp|^2}   \right \}
\]
\[
\tC (r, v_3, \rp, \vtp, \varphi ) = \sigma( \sqrt{|z|^2 + (v_3 - \vtp)^2}\;)\left \{ r^2 - \frac{ r^2 (\rp)^2  \sin ^2 \varphi }{|z|^2 + (v_3 - \vtp)^2}  \right \}
\]
\[
{\cal C} (r, v_3, \rp, \vtp, z) = \chi (r, \rp, z) \sigma(\sqrt{|z|^2 + (v_3 - \vtp)^2}\;)\left \{ r^2 - \frac{ r^2(\rp)^2 \sin ^2 \varphi }{|z|^2 + (v_3 - \vtp)^2}  \right \}.
\]
4.
\begin{align*}
C(v,\vp) & = \sigma (|v - \vp|)S(v - \vp) : (\olvp, 0) \otimes (\olv, 0) \\
& =   \sigma (|v - \vp|) \left \{ \olvp \cdot \olv - \frac{[(\olv - \olvp) \cdot \olvp ]\;[(\olv - \olvp) \cdot \olv]}{|v - \vp|^2}\right \}   
\end{align*}
\[
\tC (r, v_3, \rp, \vtp, \varphi ) =  \sigma\left \{ r \rp \cos \varphi -  \frac{ ( r\rp \cos \varphi  - (\rp )^2 )( r^2 -  r \rp \cos \varphi ) }{|z|^2 + (v_3 - \vtp)^2} \right \} 
\]
\[
{\cal C} (r, v_3, \rp, \vtp, z) = \chi (r, \rp, z) \sigma \left \{ r \rp \cos \varphi -  \frac{ r \rp ( r \cos \varphi  - \rp  )( r -   \rp \cos \varphi ) }{|z|^2 + (v_3 - \vtp)^2} \right \}.
\]
5.
\begin{align*}
C(v,\vp) & = \sigma (|v - \vp|)S(v - \vp) : (\olvp, 0) \otimes (\polv, 0) \\
& =   \sigma (|v - \vp|) \left \{ \olvp \cdot \polv - \frac{[(\olv - \olvp) \cdot \olvp ]\;[(\olv - \olvp) \cdot \polv]}{|v - \vp|^2}\right \}   
\end{align*}
\[
\tC (r, v_3, \rp, \vtp, \varphi ) =  \sigma( \sqrt{|z|^2+ (v_3 - \vtp)^2}\;)\left \{ -r \rp \sin \varphi -  \frac{ ( r\rp \cos \varphi  - (\rp )^2 )r \rp \sin \varphi  }{|z|^2 + (v_3 - \vtp)^2} \right \} 
\]
\[
{\cal C} (r, v_3, \rp, \vtp, z) = 0.
\]
6.
\begin{align*}
C(v,\vp) & = \sigma (|v - \vp|)S(v - \vp) : (\polvp, 0) \otimes (\polv, 0) \\
& =   \sigma (|v - \vp|) \left \{ \olv \cdot \olvp - \frac{[(\olv - \olvp) \cdot \polvp ]\;[(\olv - \olvp) \cdot \polv]}{|v - \vp|^2}\right \}   
\end{align*}
\[
\tC (r, v_3, \rp, \vtp, \varphi ) =  \sigma( \sqrt{|z|^2+ (v_3 - \vtp)^2}\;)\left \{ r \rp \cos \varphi -  \frac{ r^2 (\rp)^2  \sin ^2 \varphi  }{|z|^2 + (v_3 - \vtp)^2} \right \}
\]
\[
{\cal C} (r, v_3, \rp, \vtp, z) = \chi (r, \rp, z) \sigma(\sqrt{|z|^2 + (v_3 - \vtp)^2}\;)\left \{ r \rp \cos \varphi -  \frac{ r^2 (\rp)^2  \sin ^2 \varphi  }{|z|^2 + (v_3 - \vtp)^2} \right \}.
\]
\end{proof}
We also need to compute the averages
\begin{equation}
\label{FourAve}
\ave{f, (\olv, 0)}_{\sigma S},\;\;\ave{f, (\polv, 0)}_{\sigma S},\;\;\ave{f, (\olvp, 0)}_{\sigma S},\;\;\ave{f, (\polvp, 0)}_{\sigma S}.
\end{equation}
Notice that the functions $\sigma S (v - \vp) (\olv,0), \sigma S (v - \vp)(\polv,0), \sigma S (v - \vp)(\olvp,0), \sigma S (v - \vp)(\polvp,0)$ writes
\begin{align*}
D(v,\vp) = ( \tD \;\olv + \tDp \;\olvp, \tDt\;v_3 + \tDtp \;\vtp)
\end{align*}
for some scalar functions $\tD, \tDp, \tDt, \tDtp$ depending on $|\olv|, v_3, |\olvp|, \vtp$ and $\varphi$, the angle between $(\olv,0), (\olvp,0)$. Performing the same steps as in the proof of Proposition \ref{FirstFormula} we obtain (see Appendix \ref{A} for proof details)
\begin{pro}
\label{SecondFormula} Consider $\tD = \tD (|\olv|, v_3, |\olvp|, \vtp, \varphi), \tDp = \tDp (|\olv|, v_3, |\olvp|, \vtp, \varphi)$ two functions and
\[
D(v,\vp) = \tD \;\olv + \tDp \;\olvp,\;\;D_3(v,\vp) = \tD \;v_3 + \tDp \;\vtp.
\]
Then for any $f \in \ker {\cal T}$ we have
\begin{enumerate}
\item 
\begin{align*}
\ave{\intvp{D(v,\vp) f(x, \vp)}} & = \oc ^2 \intolxpvp{{\cal D}(|\olv|, v_3, |\olvp|, \vtp, \z)\\
& \times f(\olxp, x_3, \vp)}
\end{align*}
where 
\begin{align*}
{\cal D}(r, v_3, \rp, \vtp, z) & = \frac{
\left (
\begin{array}{rrr}
z_2  &  z_1\\
-z_1  &\;\;\;   z_2
\end{array}
\right )
}{2|z|} [\tD(r, v_3, \rp, \vtp, \varphi) r e ^{-i\psi} + \tD (r, v_3, \rp, \vtp, -\varphi) r e ^{i\psi}   \\
& + \tDp(r, v_3, \rp, \vtp, \varphi) \rp e ^{i(\varphi - \psi)} + \tDp (r, v_3, \rp, \vtp, -\varphi) \rp e ^{-i(\varphi - \psi)} 
]\\
& \times \chi (r, \rp, z).
\end{align*}
\item 
\begin{align*}
\ave{\intvp{D_3(v,\vp) f(x, \vp)}} & = \oc ^2 \intolxpvp{{\cal D}_3(|\olv|, v_3, |\olvp|, \vtp, \z)\\
& \times f(\olxp, x_3, \vp)}
\end{align*}
where 
\begin{align*}
{\cal D}_3(r, v_3, \rp, \vtp, z) & = \frac{1
}{2} [(\tD(r, v_3, \rp, \vtp, \varphi)  + \tD (r, v_3, \rp, \vtp, -\varphi) )v_3   \\
& + (\tDp(r, v_3, \rp, \vtp, \varphi)  + \tDp (r, v_3, \rp, \vtp, -\varphi) ) \vtp
] \chi (r, \rp, z).
\end{align*}
\end{enumerate}
The angles $\varphi, \psi \in (0,\pi)$ are such that $|z|^2 = r^2 + (\rp)^2 - 2 r \rp \cos \varphi $, $(\rp)^2 = r^2 + |z|^2 + 2 r |z| \cos \psi$.
\end{pro}
It is worth analyzing the case of even/odd coefficients $\tD, \tDp$.
\begin{pro}
\label{SecondEven} With the same notations as in Proposition \ref{SecondFormula} assume that the functions $\tD, \tDp$ are even with respect to $\varphi$. Then we have
\begin{enumerate}
\item
\[
{\cal D} (r, v_3, \rp, \vtp, z) = [(\tD (\varphi)+ \tDp (\varphi)) \;r \cos \psi + |z| \tDp (\varphi)] \chi (r, \rp, z) \frac{^\perp z}{|z|}
\]
\item
\[
{\cal D}_3 (r, v_3, \rp, \vtp, z) = [(\tD (\varphi)+ \tDp (\varphi)) \;v_3 +  \tDp (\varphi) (\vtp - v_3)] \chi (r, \rp, z).
\]
\end{enumerate}
\end{pro}
\begin{proof}
1. Clearly we have
\[
\frac{1}{2}[\tD (r, v_3, \rp, \vtp, \varphi)\;r e ^{-i\psi} + \tD (r, v_3, \rp, \vtp, -\varphi)\;r e ^{i\psi}]= \tD (r, v_3, \rp, \vtp, \varphi)\;r(\cos \psi, 0)
\]
and
\[
\frac{1}{2}[\tDp (r, v_3, \rp, \vtp, \varphi)  \;\rp e ^{i (\varphi -\psi)} + \tDp (r, v_3, \rp, \vtp, - \varphi)  \;\rp e ^{-i (\varphi -\psi)} ] = \tDp   \;\rp (\cos (\psi - \varphi ),0).
\]
Consider now the triangle of vertices $O = (0,0), A = (r,0), A^\prime = \rp e ^{i \varphi}$ in $\R^2$. The definitions for $\varphi, \psi$ assure that $|z| = |AA^\prime|$ and that $\psi$ is the supplement of the angle opposite to $OA^\prime$. Applying the cosine theorem with respect to the angle opposite to $OA$ one gets
\begin{equation}
\label{Equ63} r^2 = (\rp)^2 + |z|^2 - 2 \rp |z| \cos (\psi - \varphi ).
\end{equation}
Combining with the definition of $\psi$ yields
\[
0 = 2|z|^2 - 2 \rp |z| \cos (\psi - \varphi ) + 2 r |z| \cos \psi
\]
implying 
\begin{equation}
\label{Equ64} r \cos \psi - \rp \cos (\psi - \varphi ) + |z| = 0.
\end{equation}
Finally one gets
\begin{align*}
{\cal D} (r, v_3, \rp, \vtp, z) & =  [ \tD (\varphi)\;r\cos \psi + \tDp (\varphi)  \;\rp \cos (\psi - \varphi )]\chi (r, \rp, z) \frac{^\perp z}{|z|}\\
& = [(\tD (\varphi) + \tDp (\varphi) )\;r \cos \psi  + \tDp (\varphi ) ( \rp \cos (\psi - \varphi ) - r \cos \psi )]
\chi (r, \rp, z)\frac{^\perp z}{|z|} \\
& = [(\tD (\varphi)+ \tDp (\varphi)) \;r \cos \psi + |z| \tDp (\varphi)] \chi (r, \rp, z) \frac{^\perp z}{|z|}.
\end{align*}
2. It follows immediately observing that
\begin{align*}
{\cal D}_3 (r, v_3, \rp, \vtp, z) & = ( \tD (\varphi) \;v_3 + \tDp (\varphi) \;\vtp ) \chi (r, \rp, z) \\
& = [(\tD (\varphi)+ \tDp (\varphi)) \;v_3 +  \tDp (\varphi) (\vtp - v_3)] \chi (r, \rp, z).
\end{align*}
\end{proof}
\begin{pro}
\label{SecondOdd} With the same notations as in Proposition \ref{SecondFormula} assume that the functions $\tD, \tDp$ are odd with respect to $\varphi$. Then we have
\begin{enumerate}
\item
\[
{\cal D} (r, v_3, \rp, \vtp, z) = - [\tD (\varphi)+ \tDp (\varphi)] \;r \sin \psi  \;\chi (r, \rp, z) \;\frac{ z}{|z|}
\]
\item
\[
{\cal D}_3 (r, v_3, \rp, \vtp, z) = 0.
\]
\end{enumerate}
\end{pro}
\begin{proof}
1. We have
\[
\frac{1}{2}[\tD (r, v_3, \rp, \vtp,\varphi)\;r e ^{-i\psi} + \tD (r, v_3, \rp, \vtp, -\varphi)\;r e ^{i\psi}]= - (0, \tD (r, v_3, \rp, \vtp, \varphi)\;r \sin \psi)
\]
and
\[
\frac{1}{2}[\tDp (r, v_3, \rp, \vtp, \varphi)  \;\rp e ^{i (\varphi -\psi)} + \tDp (r, v_3, \rp, \vtp, - \varphi)  \;\rp e ^{-i (\varphi -\psi)} ] = - (0, \tDp   \;\rp \sin  (\psi - \varphi )).
\]
The sine theorem applied in the triangle of vertices $O = (0,0), A = (r,0), A^\prime = \rp e ^{i \varphi}$ implies
\[
r \sin \psi = \rp \sin (\psi - \varphi).
\]
We deduce that 
\begin{align*}
{\cal D} (r, v_3, \rp, \vtp, z) & =  - [ \tD (\varphi)\;r\sin \psi + \tDp (\varphi)  \;\rp \sin (\psi - \varphi)] \;\chi (r, \rp, z) \frac{ z}{|z|}\\
& =  - [\tD (\varphi)+ \tDp (\varphi)] \;r \sin \psi  \;\chi (r, \rp, z) \;\frac{ z}{|z|}.
\end{align*}
2. Clearly we have $\tD (\varphi) + \tD (- \varphi) = \tDp (\varphi) + \tDp (- \varphi) =0$ and therefore ${\cal D}_3 = 0$.
\end{proof}
The averages in \eqref{FourAve} come immediately appealing to Propositions \ref{SecondEven}, \ref{SecondOdd}.
\begin{coro}
\label{Cal1} With the notations in Proposition \ref{SecondFormula} we have for any function $f \in \ker {\cal T}$
\begin{enumerate}
\item 
\[
\ave{f, (\olv, 0)}_{\sigma S} = - \oc ^2 \intolxpvp{\sigma f \chi  \frac{r^2 - r \rp \cos \varphi }{|z| ^2 + (v_3 - \vtp )^2}
\left ( \frac{(v_3 - \vtp )^2}{|z|^2} \;^\perp z , v_3 - \vtp \right ) }
\]
\item 
\[
\ave{f, (\olvp, 0)}_{\sigma S} =  \oc ^2 \intolxpvp{\sigma f \chi  \frac{(\rp)^2 - r \rp \cos \varphi }{|z| ^2 + (v_3 - \vtp )^2}
\left ( \frac{(v_3 - \vtp )^2}{|z|^2} \;^\perp z , v_3 - \vtp \right ) }
\]
\item 
\[
\ave{f, (\polv, 0)}_{\sigma S} =  \oc ^2 \intolxpvp{\sigma f \chi  \frac{r^2 - r \rp \cos \varphi }{|z| ^2 }
\left (  z , 0 \right ) }
\]
\item 
\[
\ave{f, (\polvp, 0)}_{\sigma S} =  -\oc ^2 \intolxpvp{\sigma f \chi  \frac{(\rp)^2 - r \rp \cos \varphi }{|z| ^2 }
\left (  z , 0 \right ) }.
\]
\end{enumerate}
\end{coro}
\begin{proof}
1. We consider the function $D(v,\vp) = \sigma (|v - \vp|) S(v - \vp) ( \olv, 0) = (\tD \;\olv + \tDp \;\olvp, \tDp (\vtp - v_3))$ where
\[
\tD (r, v_3, \rp, \vtp, \varphi) = \sigma \left ( 1 - \frac{r^2 - r \rp \cos \varphi }{r^2 + (\rp)^2 - 2 r \rp \cos \varphi + (v_3 - \vtp)^2}\right )
\]
\[
\tDp (r, v_3, \rp, \vtp, \varphi) = \sigma \frac{r^2 - r \rp \cos \varphi }{r^2 + (\rp)^2 - 2 r \rp \cos \varphi + (v_3 - \vtp)^2}.
\]
Thanks to Proposition \ref{SecondEven} and the identity $r \cos \psi = - \frac{r^2 - r \rp \cos \varphi }{|z|}$ we obtain
\begin{align*}
\ave{\intvp{(\tD \;\olv + \tDp \;\olvp) f (x, \vp)}} & = \oc ^2 \intolxpvp{\sigma f(x_1 ^\prime, x_2 ^\prime, x_3, \vp) \chi \\
& \times \left (r \cos \psi + |z| \frac{r^2 - r \rp \cos \varphi }{|z|^2 + (v_3 - \vtp)^2}   \right )\frac{^\perp z}{|z|}} \\
& = - \oc ^2\intolxpvp{\sigma f(x_1 ^\prime, x_2 ^\prime, x_3, \vp) \chi \\
& \times  \frac{r^2 - r \rp \cos \varphi }{|z|^2 + (v_3 - \vtp)^2}   \frac{(v_3 - \vtp)^2}{|z|}\frac{^\perp z}{|z|}}
\end{align*}
and also 
\begin{align*}
\ave{\intvp{\tDp (\vtp - v_3) f(x,\vp)}} & = \oc ^2 \intolxpvp{\sigma f(x_1 ^\prime, x_2 ^\prime, x_3, \vp) \chi \\
& \times  \frac{r^2 - r \rp \cos \varphi }{|z|^2 + (v_3 - \vtp)^2}   (\vtp - v_3)}
\end{align*}
which justifies the first statement.\\
2. We take $D(v,\vp) = \sigma (|v - \vp|) S(v - \vp) ( \olvp, 0) = (\tD \;\olv + \tDp \;\olvp, \tD (v_3- \vtp))$ where
\[
\tD (r, v_3, \rp, \vtp, \varphi) = \sigma \frac{(\rp)^2 - r \rp \cos \varphi }{r^2 + (\rp)^2 - 2 r \rp \cos \varphi + (v_3 - \vtp)^2}
\]
\[
\tDp (r, v_3, \rp, \vtp, \varphi) = \sigma \left ( 1 - \frac{(\rp)^2 - r \rp \cos \varphi }{r^2 + (\rp)^2 - 2 r \rp \cos \varphi + (v_3 - \vtp)^2}\right ) .
\]
Notice that by \eqref{Equ63}, \eqref{Equ64} we have
\[
r \cos \psi + |z| = \rp \cos (\psi - \varphi ) = \frac{(\rp)^2 + |z|^2 - r^2}{2|z|} = \frac{(\rp)^2 - r \rp \cos \varphi }{|z|}
\]
and in this case we obtain
\begin{align*}
\ave{\intvp{(\tD \;\olv + \tDp \;\olvp) f (x, \vp)}} & = \oc ^2 \intolxpvp{\sigma f(x_1 ^\prime, x_2 ^\prime, x_3, \vp) \chi \\
& \times \left [r \cos \psi + |z| \left ( 1 - \frac{(\rp)^2 - r \rp \cos \varphi }{|z|^2 + (v_3 - \vtp)^2}   \right )\right ]\frac{^\perp z}{|z|}} \\
& =  \oc ^2\intolxpvp{\sigma f(x_1 ^\prime, x_2 ^\prime, x_3, \vp) \chi \\
& \times  \frac{(\rp)^2 - r \rp \cos \varphi }{|z|^2 + (v_3 - \vtp)^2}   \frac{(v_3 - \vtp)^2}{|z|}\frac{^\perp z}{|z|}}
\end{align*}
and 
\begin{align*}
\ave{\intvp{\tD (v_3 - \vtp) f(x,\vp)}} & = \oc ^2 \intolxpvp{\sigma f(x_1 ^\prime, x_2 ^\prime, x_3, \vp) \chi \\
& \times  \frac{(\rp)^2 - r \rp \cos \varphi }{|z|^2 + (v_3 - \vtp)^2}   (v_3 - \vtp)}
\end{align*}
justifying the second statement.\\
3. We take $D(v,\vp) = - \sigma (|v - \vp|) \frac{(v - \vp) \otimes (v - \vp)}{|v - \vp|^2} ( \polv, 0) = (\tD \;\olv + \tDp \;\olvp, \tD v_3 + \tDp \vtp)$ where
\[
\tD (r, v_3, \rp, \vtp, \varphi) = -\sigma \frac{r \rp \sin \varphi }{r^2 + (\rp)^2 - 2 r \rp \cos \varphi + (v_3 - \vtp)^2} = - \tDp.
\]
By Proposition \ref{SecondOdd} we deduce that $\ave{\intvp{D(v,\vp)f(x,\vp)}} = 0$. Therefore 
\[
\ave{f, (\polv, 0)}_{\sigma S} = \ave{\intvp{\sigma (|v - \vp|)(\polv, 0)f(x,\vp)}}.
\]
Applying now Proposition \ref{SecondEven} with $\tD = \sigma, \tDp = 0$ we obtain
\[
\ave{\intvp{\sigma (|v - \vp|)f(x,\vp) \olv }} = \oc ^2 \intolxpvp{\sigma f \chi \;r \cos \psi \frac{^\perp z}{|z|} }
\]
and finally 
\begin{align*}
\ave{f, (^\perp \olv, 0)} _{\sigma S} & = - \oc ^2 \intolxpvp{\sigma f \chi \; r \cos \psi \frac{(z,0)}{|z|}} \\
& = \oc ^2 \intolxpvp{\sigma f \chi \frac{r ^2 - r \rp \cos \varphi }{|z|^2} (z,0)}.
\end{align*}
4. As before, by Proposition \ref{SecondOdd} we have
\[
\ave{\intvp{\sigma (|v - \vp|) f (x,\vp) \frac{(v - \vp) \otimes (v - \vp)}{|v - \vp |^2} ( ^\perp \olvp, 0)}} = 0
\]
and by Proposition \ref{SecondEven} we obtain
\begin{align*}
\ave{\intvp{\sigma (|v - \vp|)f(x,\vp) \olvp }} & = \oc ^2 \intolxpvp{\sigma f \chi \;(r \cos \psi + |z|)\frac{^\perp z}{|z|} }
\\
& = \oc ^2 \intolxpvp{\sigma f \chi \;\frac{(\rp)^2 - r \rp \cos \varphi }{|z|^2} \;{^\perp z}}. 
\end{align*}
At the end one gets
\[
\ave{f, (\polvp, 0)}_{\sigma S} =  -\oc ^2 \intolxpvp{\sigma f \chi  \frac{(\rp)^2 - r \rp \cos \varphi }{|z| ^2 }
\left (  z , 0 \right ) }.
\]
\end{proof}
The last average we will need is $\ave{f}_{\sigma S} = \ave{\intvp{f(x,\vp) \sigma (|v - \vp|) S(v - \vp) }}$. By similar computations as those in the proofs of Propositions \ref{FirstFormula}, \ref{SecondFormula} we obtain (see Appendix \ref{A} for details)
\begin{pro}
\label{ThirdFormula}
For any function $f \in \ker {\cal T}$ we have
\[
\ave{f}_{\sigma S} =\oc ^2 \intolxpvp{\sigma (\sqrt{|z|^2 + (v_3 - \vtp)^2}\;) f (\olxp, x_3, \vp) \chi (|\olv|, |\olvp|, z) S(\;({^\perp z}, \vtp - v_3)\;)}.
\]
\end{pro}

\subsection{The averaged \fpl operator}
We are ready to determine the average of the \fpl kernel. For the sake of presentation we treat separately the gain and loss parts. Recall that the \fpl gain part appears as a velocity diffusion, where the diffusion matrix is a convolution in velocity
\[
\qfpl ^+ (f,f) = \divv \left \{\intvp{\sigma ( |v - \vp|) S(v - \vp) f(\vp) }\;\nabla _v f(v)   \right \}.
\]
The averaged \fpl kernel will keep the same structure, nevertheless diffusion and convolution have to be considered both in velocity and space perpendicular directions, as we have already observed in the relaxation case (see Remark \ref{Normalize}). The proof is postponed to Appendix \ref{B}.
\begin{pro}
\label{GainLFP} For any function $f = f(x,v)$ satisfying the constraint ${\cal T} f = 0$ we have
\begin{align}
\label{Equ70} \ave{\qfpl ^+ (f,f)} & = \divoxv  \{\oc ^2 \intolxpvp{\sigma(\sqrt{|z|^2 + (v _3 - \vtp)^2}\;)f (x_1 ^\prime, x_2 ^\prime, x_3, \vp) \;\chi (|\olv|, |\olvp|, z) \nonumber \\
& \times A^+ \gradoxv f(x,v)} \}
\end{align}
with $\mathrm{div} _{\oc x} = \frac{1}{\oc} \divx$, $\nabla _{\oc x} = \frac{1}{\oc} \nabla _x$
\begin{align*}
& A^+ (r, v_3, \rp, \vtp, z)  = \frac{(\rp )^2 \sin ^2 \varphi \;(v _3 - \vtp )^2}{|z|^2 [ |z|^2 + (v_3 - \vtp ) ^2]}\left ( \frac{(\olv,0)}{|\olv|}, \frac{(\polv, 0)}{|\olv|}   \right ) ^{\otimes 2} \\
& + \left [ \frac{r - \rp \cos \varphi }{|z|}\left ( \frac{(\olv,0)}{|\olv|}, \frac{(\polv, 0)}{|\olv|}   \right ) + \left (\frac{({^\perp z}, 0)}{|z|}, 0   \right ) \right ] ^{\otimes 2}  + \frac{(\rp )^2 \sin ^2 \varphi }{|z|^2} \left ( \frac{(\polv,0)}{|\olv|}, -\frac{(\olv, 0)}{|\olv|}   \right ) ^{\otimes 2} \\
& + \left [\frac{(\rp \cos \varphi - r) (v_3 - \vtp)}{|z| \sqrt{|z|^2 + (v_3 - \vtp )^2}}    \left ( \frac{(\polv,0)}{|\olv|}, -\frac{(\olv, 0)}{|\olv|}   \right )   
+ \frac{\left ((v_3 - \vtp)\frac{(z,0)}{|z|}, -|z|e_3  \right )}{\sqrt{|z|^2 + (v_3 - \vtp )^2}} \right ]^{\otimes 2}
\end{align*}
where $z = \z$ and for any $r, \rp \in \R_+, z \in \R^2$ such that $|r - \rp| < |z| < r + \rp$, the angle $\varphi \in (0,\pi)$ is given by $|z|^2 = r ^2 + (\rp)^2 - 2 r \rp \cos \varphi$. 
\end{pro}
\begin{remark}
\label{KerAGain} Clearly $A^+$ is symmetric and positive. Notice also that the vectors $(e_3,0)$ and $(\;(z,0),(-\pz, v_3 - \vtp)\;)$ are orthogonal on
\[
\olvpolv,\;\;\polvolv,\;\;
\left (\frac{(\pz, 0)}{|z|},0  \right ),\;\;\vtz
\]
saying that 
\[
A^+(e_3,0)  = A^+ (\;(z,0),(-{^\perp z}, v_3 - \vtp)\;)= 0.
\]
Actually we have for any $z \neq 0$
\[
\ker A^+ (r, v_3, \rp, \vtp, z) = \mathrm{span}\{(e_3,0),\;\;(\;(z,0),(-{^\perp z}, v_3 - \vtp)\;)\}.
\]
\end{remark}
A similar result can be carried out for the loss part $\qfpl ^-$ (see Appendix \ref{B} for the proof).
\begin{pro}
\label{LossLFP} For any function $f = f(x,v)$ satisfying the constraint ${\cal T} f = 0$ we have
\begin{align}
\label{Equ80} \ave{\qfpl ^- (f,f)} & = \divoxv  \{\oc ^2 \intolxpvp{\sigma(\sqrt{|z|^2 + (v _3 - \vtp)^2}\;)f (x,v) \;\chi (|\olv|, |\olvp|, z) \nonumber \\
& \times A^- \gradoxpvp f(x_1 ^\prime, x_2 ^\prime, x_3, \vp)} \}
\end{align}
with 
\begin{align*}
& A^- (r, v_3, \rp, \vtp, z)  = \frac{r \rp  \sin ^2 \varphi \;(v _3 - \vtp )^2}{|z|^2 [ |z|^2 + (v_3 - \vtp ) ^2]} \olvpolv \otimes \olvppolvp  \\
& + \left [ \frac{r - \rp \cos \varphi }{|z|}\left ( \frac{(\olv,0)}{|\olv|}, \frac{(\polv, 0)}{|\olv|}   \right ) + \left (\frac{({^\perp z}, 0)}{|z|}, 0   \right ) \right ] \\
& \otimes \left [ \frac{r \cos \varphi - \rp  }{|z|}\left ( \frac{(\olvp,0)}{|\olvp|}, \frac{(\polvp, 0)}{|\olvp|}   \right ) + \left (\frac{({^\perp z}, 0)}{|z|}, 0   \right ) \right ] \\
& + \frac{r \rp  \sin ^2 \varphi }{|z|^2} \left ( \frac{(\polv,0)}{|\olv|}, -\frac{(\olv, 0)}{|\olv|}   \right ) \otimes  \left ( \frac{(\polvp,0)}{|\olvp|}, -\frac{(\olvp, 0)}{|\olvp|}  \right )\\
& + \left [\frac{(\rp \cos \varphi - r) (v_3 - \vtp)}{|z| \sqrt{|z|^2 + (v_3 - \vtp )^2}}    \left ( \frac{(\polv,0)}{|\olv|}, -\frac{(\olv, 0)}{|\olv|}   \right )   + \vtz \right ] \\
& \otimes 
\left [\frac{(\rp - r \cos \varphi ) (v_3 - \vtp)}{|z| \sqrt{|z|^2 + (v_3 - \vtp )^2}}    \left ( \frac{(\polvp,0)}{|\olvp|}, -\frac{(\olvp, 0)}{|\olvp|}   \right )   + \vtz \right ]
\end{align*}
where $z = \z$ and for any $r, \rp \in \R_+, z \in \R^2$ such that $|r - \rp| < |z| < r + \rp$, the angle $\varphi \in (0,\pi)$ is given by $|z|^2 = r ^2 + (\rp)^2 - 2 r \rp \cos \varphi$. 
\end{pro}
For any $\olx = (x_1, x_2), \olxp = (x_1 ^\prime, x_2 ^\prime)\in \R^2$, $v, \vp \in \R^3$ we introduce the fields
\begin{align*}
\xi ^1 (\olx, v, \olxp, \vp) & = \{  \sigma \chi \}^{1/2} \frac{\rp \sin  \varphi \;(v_3 - \vtp )}{|z| \sqrt{|z|^2 + (v_3 - \vtp )^2 }}\;\olvpolv\\
\xi ^2 (\olx, v, \olxp, \vp) & = \{  \sigma \chi \}^{1/2}\left [\frac{r - \rp \cos \varphi }{|z|}\olvpolv + \left ( \frac{({^\perp z},0)}{|z|}, 0\right )     \right ] \\
\xi ^3 (\olx, v, \olxp, \vp) & = \{  \sigma \chi \}^{1/2}\frac{ \rp  \sin  \varphi }{|z| } \polvolv \\
\frac{\xi ^4 (\olx, v, \olxp, \vp)}{\{  \sigma \chi \}^{1/2}} & =  
 \frac{(\rp \cos \varphi - r)(v_3 - \vtp)}{|z|\sqrt{|z|^2 + (v _3 - \vtp)^2}}\polvolv + 
\vtz 
\end{align*}
where $r = |\olv|, \rp = |\olvp|, z = \z, \sigma = \sigma{\sqrt{|z|^2 + (v _3 - \vtp)^2}}, \chi = \chi (r, \rp, z)$ and $\varphi \in (0,\pi)$ is given by $|z|^2 = r ^2 + (\rp )^2 - 2 r \rp \cos \varphi$, $|r - \rp| < |z| < r + \rp$. Thanks to Propositions \ref{GainLFP}, \ref{LossLFP} we obtain the representation formula
\begin{pro}
\label{LFPKerT} Consider a function $f = f(x,v)$ satisfying the constraint ${\cal T} f = 0$. Then\\
1.
The averaged \fpl operator writes
\begin{align}
\label{Equ90} & \oc ^{-2} \ave{\qfpl (f,f)} (x,v) = \nonumber \\
&  \divoxv \left \{  \intolxpvp{\;\sum _{i = 1} ^4 f(\olxp, x_3, \vp) \xi ^i (\olx, v, \olxp, \vp) \otimes \xi ^i (\olx, v, \olxp, \vp) \gradoxv f (x,v)}  \right \} \nonumber \\
- & \divoxv \left \{  \intolxpvp{\;\sum _{i = 1} ^4 f(x,v) \xi ^i (\olx, v, \olxp, \vp) \otimes \eps _i \xi ^i (\olxp, \vp, \olx, v) \gradoxpvp f (\olxp, x_3, \vp)}  \right \}
\end{align}
where $\eps _1 = \eps _2 = -1, \eps _3 = \eps _4 = 1$.\\
2. The following properties hold true for any fixed $x_3 \in \R$
\[
\intolxv{\ave{\qfpl (f,f)}(x,v) } = 0,\;\intolxv{v \ave{\qfpl (f,f)}(x,v) } = 0
\]
\[
\intolxv{\!\frac{|v|^2}{2}\!\ave{\qfpl (f,f)}(x,v) } = 0,\intolxv{\;\ln f \ave{\qfpl (f,f)}(x,v) } \leq 0.
\]
\end{pro}
\begin{proof}
1. By Proposition \ref{GainLFP} we know that 
\begin{align}
\label{Equ91} 
 \oc ^{-2} & \ave{\qfpl ^+(f,f)}  (x,v) =\\
  & \divoxv \left \{  \intolxpvp{\;\sum _{i = 1} ^4 f(\olxp, x_3, \vp) (\xi ^i)^{\otimes 2} (\olx, v, \olxp, \vp)  \gradoxv f (x,v)}  \right \}.\nonumber 
\end{align}
Observe that we have $\chi (\rp, r, -z) = \chi (r, \rp, z)$. Therefore the permutation $(\olx, v) \longleftrightarrow (\olxp, \vp)$ leads to
\begin{align*}
\xi ^1 (\olxp, \vp, \olx, v) & = \{  \sigma \chi \}^{1/2} \frac{r\sin  \varphi \;(\vtp - v_3 )}{|z| \sqrt{|z|^2 + (v_3 - \vtp )^2 }}\;\olvppolvp
\\
\xi ^2 (\olxp, \vp, \olx, v) & = \{  \sigma \chi \}^{1/2}\left [\frac{\rp - r \cos \varphi }{|z|}\olvppolvp - \left ( \frac{({^\perp z},0)}{|z|}, 0\right )     \right ]
\\
\xi ^3 (\olxp, \vp, \olx, v) & = \{  \sigma \chi \}^{1/2}\frac{ r  \sin  \varphi }{|z| } \polvpolvp
\\
\frac{\xi ^4 (\olxp, \vp, \olx, v)}{\{  \sigma \chi \}^{1/2}} & = 
\frac{(r \cos \varphi - \rp)(\vtp - v_3)}{|z|\sqrt{|z|^2 + (v _3 - \vtp)^2}}\polvpolvp + 
\vtz  
\end{align*}
where $z = \z$. By Proposition \ref{LossLFP} one gets
\begin{align}
\label{Equ92} 
& \oc ^{-2}  \ave{\qfpl ^-(f,f)}  (x,v) =\\
& \divoxv \left \{  \intolxpvp{\;\sum _{i = 1} ^4 f(x,v) \xi ^i (\olx, v, \olxp, \vp) \otimes \eps _i \xi ^i (\olxp, \vp, \olx, v) \gradoxpvp f (\olxp, x_3, \vp)}  \right \}\nonumber
\end{align}
and the first statement follows combining \eqref{Equ91} and \eqref{Equ92}.\\
2. The mass, third momentum component and kinetic energy balances for the averaged \fpl operator come by the corresponding properties of the \fpl kernel. Indeed, since $1, v_3, \frac{|v|^2}{2}$ belong to $\ker {\cal T}$, we can write for any $x_3 \in \R$, thanks to Remark \ref{Ave4D}
\begin{align*}
\intolxv{\left \{1, v_3, \frac{|v|^2}{2}\right \} \ave{\qfpl (f,f)}} & = \intolxv{\left \{1, v_3, \frac{|v|^2}{2}\right \} {\qfpl (f,f)}} \\
& =(0,0,0).
\end{align*}
Similarly since $\ln f \in \ker {\cal T}$ one gets
\begin{align}
\label{EquDegen} & \intolxv{\ln f \ave{\qfpl (f,f)}}  = \intolxv{\ln f \;{\qfpl (f,f)}}  \\
  = & -\frac{1}{2} \intolxv{\intvp{\sigma (|v - \vp|) f (x,v) f (x, \vp) \frac{|(v - \vp) \wedge ( \nabla _v \ln f - \nabla _{\vp} \ln f )|^2}{|v - \vp |^2}}\!} \leq 0.\nonumber
\end{align}
Finally observe that $\ave{v_1} = \ave{v_2} = 0$, $\ave{\qfpl (f,f)} \in \ker {\cal T}$ and thus trivially 
\[
\intolxv{(v_1, v_2) \ave{\qfpl (f,f)}} = (0,0).
\]
\end{proof}
We establish formally the limit model stated in Theorem \ref{MainResultLFP}.
\begin{proof} (of Theorem \ref{MainResultLFP}) Plugging the Ansatz $\fe = f + \eps f^1 + \eps ^2 f ^2 +...$ into \eqref{EquPert} we obtain
\begin{align*}
\left ( \partial _t + v_3 \partial _{x_3} + \frac{q}{m} E \cdot \nabla _v  + \frac{1}{\eps} {\cal T} \right )  (f + \eps f^1 + ...)  & = \qfpl (f,f) \\
& + \eps ( \qfpl (f, f^1) + \qfpl (f^1, f) ) + ...
\end{align*}
implying that 
\[
{\cal T} f = 0,\;\;\partial _t f + v_3 \partial _{x_3} f + \frac{q}{m} E \cdot \nabla _v f + {\cal T} f^1 = \qfpl (f,f).
\]
Applying the average operator, we deduce by Propositions \ref{AveTra}, \ref{LFPKerT} that $f$ satisfies
\[
\partial _t f + \frac{\ave{\polE}}{B} \cdot \gradolx f + v_3 \partial _{x_3} f + \frac{q}{m} \ave{E_3} \partial _{v_3} f = \ave{\qfpl} (f,f).
\]
\end{proof}
As for the relaxation Boltzmann operator, we are searching for extensions of the averaged \fpl operator to the whole space of densities $f = f (x,v)$, not necessarily in the kernel of ${\cal T}$. One possibility is to consider the extension $\ave{\qfpl}$ obtained thanks to \eqref{Equ90}, that is for any $f$
\begin{align}
\label{Equ93} & \oc ^{-2} \ave{\qfpl } (f,f) = \nonumber \\
&  \divoxv \left \{  \intolxpvp{\;\sum _{i = 1} ^4 f(\olxp, x_3, \vp) \xi ^i (\olx, v, \olxp, \vp) \otimes \xi ^i (\olx, v, \olxp, \vp) \gradoxv f (x,v)}  \right \} \nonumber \\
- & \divoxv \left \{  \intolxpvp{\;\sum _{i = 1} ^4 f(x,v) \xi ^i (\olx, v, \olxp, \vp) \otimes \eps _i \xi ^i (\olxp, \vp, \olx, v) \gradoxpvp f (\olxp, x_3, \vp)}  \right \}.
\end{align}
What is remarkable is that this extension still satisfies the mass, third momentum component, kinetic energy balances and decreases the entropy $f \ln f$, globally in $(\olx, v)$.
\begin{pro}
\label{ConserLaw} Consider two functions $f = f(x,v), \varphi = \varphi (x,v)$. For any $x_3 \in \R$ we have
\begin{align}
\label{EquWeakForm} & \intolxv{\ave{\qfpl}(f,f) \varphi}= - \frac{\oc ^2}{2}  \times \\
& \sum _{i = 1} ^4 \intolxv{\intolxpvp{f f^\prime ( \xi ^i \cdot \nabla \ln f - \eps _i (\xi ^i)^\prime \nabla ^\prime \ln f ^\prime) ( \xi ^i \cdot \nabla  \varphi - \eps _i (\xi ^i)^\prime \nabla ^\prime \varphi ^\prime)}}\nonumber
\end{align}
where 
\[
f= f(x,v),\;f^\prime = f ^\prime(x_1 ^\prime, x_2 ^\prime, x_3, \vp),\;\nabla \varphi = \gradoxv \varphi (x,v),\;
\nabla ^\prime \varphi ^\prime = \gradoxpvp \varphi (x_1 ^\prime, x_2 ^\prime, x_3, \vp)
\]
\[
\xi ^i = \xi ^i (x_1, x_2, v, x_1 ^\prime, x_2 ^\prime, \vp),\;\;(\xi ^i)^\prime = \xi ^i (x_1 ^\prime, x_2 ^\prime, \vp, x_1, x_2, v).
\]
In particular the averaged \fpl operator satisfies the mass, third momentum component, kinetic energy  balances (globally in $(x_1, x_2, v)$)
\[
\intolxv{\left ( 1, v_3, \frac{|v|^2}{2}\right ) \ave{\qfpl}(f,f) } = (0, 0, 0)
\]
and decreases the entropy $f \ln f$ (globally in $(x_1, x_2, v)$)
\[
\intolxv{\ln f \ave{\qfpl}(f,f) } \leq 0.
\]
\end{pro}
\begin{proof}
Notice that for any $1 \leq i \leq 4$ we have $\xi ^i \cdot (e_3, 0) = 0$ and therefore the operator $\divoxv$ acts only in $(x_1, x_2, v)$. Thus, for any fixed $x_3 \in \R$ we can perform integration by parts with respect to $(x_1, x_2, v)$.
\begin{align}
\label{Equ94} & \intolxv{\ave{\qfpl}(f,f) \varphi }  = - \sum _{i = 1} ^4 \oc ^2 \intolxv{ \intolxpvp{ f f ^\prime \\
& \times \left \{(\xi ^i \cdot \nabla \varphi ) ( \xi ^i \cdot \nabla \ln f )  - \eps _i (\xi ^i \cdot \nabla \varphi ) ( (\xi ^i)^\prime \cdot \nabla ^\prime \ln f ^\prime)   \right \}}}.\nonumber
\end{align}
Performing the change of variables $(x_1 ^\prime, x_2 ^\prime, \vp) \leftrightarrow (x_1, x_2, v)$ yields
\begin{align}
\label{Equ95} & \intolxv{\ave{\qfpl}(f,f) \varphi }  = - \sum _{i = 1} ^4 \oc ^2 \intolxpvp{ \intolxv{ f f ^\prime \\
& \times \left \{((\xi ^i )^\prime \cdot \nabla ^\prime  \varphi ^\prime) ( (\xi ^i ) ^\prime \cdot \nabla ^\prime  \ln f ^\prime)  - \eps _i ((\xi ^i) ^\prime \cdot \nabla ^\prime \varphi ^\prime) ( \xi ^i \cdot \nabla  \ln f )   \right \}}}.\nonumber
\end{align}
Combining \eqref{Equ94}, \eqref{Equ95} one gets by Fubini theorem
\[
\intolxv{\ave{\qfpl}(f,f) \varphi }  = - \frac{\oc ^2}{2}\sum _{i = 1} ^4  \intolxv{ \intolxpvp{ f f ^\prime T^i }}
\]
where 
\[
T^i = \left (\xi ^i \cdot \nabla \varphi - \eps _i   (\xi ^i) ^\prime \cdot \nabla ^\prime \varphi ^\prime   \right ) \;\left (\xi ^i \cdot \nabla \ln f - \eps _i   (\xi ^i) ^\prime \cdot \nabla ^\prime \ln f ^\prime \right),\;\;1 \leq i \leq 4.
\]
Clearly, the divergence form of $\ave{\qfpl}$ guarantees the mass conservation and \eqref{EquWeakForm} applied with $\varphi = \ln f$ ensures that the entropy $f\ln f$ is decreasing
\begin{align*}
\intolxv{\ln f & \ave{\qfpl}(f,f)  }   = - \frac{\oc ^2}{2}\sum _{i = 1} ^4  \intolxv{ \intolxpvp{ f f ^\prime \\
& \times \left (\xi ^i \cdot \nabla \ln f - \eps _i   (\xi ^i) ^\prime \cdot \nabla ^\prime \ln f ^\prime   \right )^2 }}\leq 0,\;x_3 \in \R.
\end{align*}
It remains to show the kinetic energy and third momentum component balances. Thanks to formula \eqref{EquWeakForm} it is sufficient to check that 
\[
\xi ^i \cdot \nabla \frac{|v|^2}{2} - \eps _i (\xi ^i )^\prime \cdot \nabla ^\prime \frac{|\vp|^2}{2} = 0,\;\;1 \leq i \leq 4.
\]
The above condition is trivially satisfied for $i \in \{1,2\}$. For $i = 3$ we have
\[
\xi ^3 \cdot \nabla \frac{|v|^2}{2} - \eps _3 (\xi ^3 )^\prime \cdot \nabla ^\prime \frac{|\vp|^2}{2} = - \{ \sigma \chi \} ^{1/2} \frac{\rp \sin \varphi }{|z|} r + \{ \sigma \chi \} ^{1/2} \frac{r \sin \varphi }{|z|} \rp = 0. 
\]
Finally, when $i = 4$ we obtain
\begin{align*}
& \xi ^4 \cdot \nabla \frac{|v|^2}{2} - \eps _4 (\xi ^4 )^\prime \cdot \nabla ^\prime \frac{|\vp|^2}{2} \\
& =\{ \sigma \chi \} ^{1/2} \left \{- \frac{(\rp \cos \varphi - r) (v_3 - \vtp)r  + |z|^2 v_3}{|z| \sqrt{|z|^2 + (v_3 - \vtp)^2}} + \frac{(r \cos \varphi - \rp) (\vtp - v_3)\rp  + |z|^2 \vtp}{|z| \sqrt{|z|^2 + (v_3 - \vtp)^2}}     \right \} \\
& = \{ \sigma \chi \} ^{1/2}\frac{v_3 - \vtp}{\sqrt{|z|^2 + (v_3 - \vtp)^2}} \left [  r^2 + (\rp)^2 - 2 r \rp \cos \varphi \right ] = 0.
\end{align*}
Notice also that for any $i\in \{1,2,3,4\}$ we have
\[
\xi ^i \cdot \nabla v_3 - \eps _i (\xi ^i )^\prime \cdot \nabla ^\prime \vtp = 0
\]
saying that $\ave{\qfpl}$ satisfies 
\[
\intolxv{v_3 \ave{\qfpl}(f,f) } = 0.
\]
\end{proof}
\begin{remark}
\label{OtherBalance} By the formula \eqref{EquWeakForm} we deduce that the positive smooth functions $f$ satisfying $\ave{\qfpl}(f,f) = 0$ are those verifying
\begin{align}
\label{Equ111} 
\xi ^i \cdot \nabla \ln f - \eps _i (\xi ^i )^\prime \cdot \nabla ^\prime \ln f ^\prime = 0,\;\;i \in \{1,2,3,4\}.
\end{align}
In particular \eqref{Equ111} holds true for any Maxwellian $f$ which belongs to $\ker {\cal T}$, since in that case
\[
\ave{\qfpl}(f,f) = \ave{\qfpl (f,f)} = \ave{0} = 0.
\]
We deduce that 
\begin{align}
\label{Equ112} 
\xi ^i \cdot \nabla \varphi - \eps _i (\xi ^i )^\prime \cdot \nabla ^\prime \varphi  ^\prime = 0,\;\;i \in \{1,2,3,4\}
\end{align}
for any function $\varphi (x,v) = \alpha (x) |v|^2 + \beta (x) \cdot v + \gamma (x)$ satisfying ${\cal T} \varphi = 0$, and in particular for the functions
\[
x_1 + \frac{v_2}{\oc},\;x_2 - \frac{v_1}{\oc},\;x_3, \;v_3,\; |\olx|^2 + 2  \olx \cdot \frac{\polv}{\oc} = \left |\olx + \frac{\polv}{\oc} \right |^2 - \frac{|\olv|^2}{\oc ^2}.
\]
Notice that $\olx + \frac{\polv}{\oc}$ is the center of the circle obtained by projecting the Larmor circle $C_{x,v}$ onto $(x_1 ^\prime, x_2 ^\prime)$ and $|\olx|^2  + 2 \olx \cdot \frac{\polv}{\oc}$ is the power of the origin $(0,0)$ with respect to the same circle. 
We obtain, thanks to \eqref{EquWeakForm} and \eqref{Equ112}, the balances
\[
\intolxv{\left (\olx + \frac{\polv}{\oc} \right)  \ave{\qfpl}(f,f) } = (0,0)
\]
\[
\intolxv{\left ( |\olx|^2 + 2  \olx \cdot \frac{\polv}{\oc} \right ) \ave{\qfpl}(f,f) } = 0
\]
for any smooth function $f$.
\end{remark}
The previous identities allow us to establish the mass, momentum,  total energy, Larmor circle center and power conservations for smooth solutions of \eqref{Equ101}, \eqref{Equ102} coupled with the Poisson equation for the electric field
\begin{align}
\label{Equ113} E = - \nabla _x \varphi,\;\;\eps _0 \;\divx  E(t,x) = q \intv{f(t,x,v)} = : \rho (t,x),\;(t,x) \in \R_+ \times \R ^3.
\end{align}
\begin{thm}
\label{Conservation} Assume that $(f, E)$ is a smooth solution for
\[
\partial _t f + \frac{\ave{\polE}}{B} \cdot \gradolx f + v_3 \partial _{x_3} f + \frac{q}{m} \ave{E_3} \partial _{v_3} f = \ave{\qfpl} (f,f),\;\;{\cal T} f = 0,\;\;(t,x,v) \in \R_+ \times \R^3 \times \R^3
\]
\[
\eps _0 \divx E(t,x) = q \intv{f(t,x,v)},\;\;E = - \nabla _x \phi,\;\;(t,x) \in \R_+ \times \R ^3.
\]
Then we have
\[
\frac{\md }{\md t} \intxv{f} = 0,\;\;\frac{\md }{\md t} \intxv{\olx f} = 0,\;\;\frac{\md }{\md t} \intxv{v f} = 0
\]
\[
\frac{\md }{\md t} \left \{ \frac{\eps _0}{2} \int _{\R^3}{|E|^2\;\md x} + \intxv{m \frac{|v|^2}{2}f}\right \} = 0,\;\;\frac{\md }{\md t} \intxv{\left (\frac{|\olx|^2}{2} + \olx \cdot \frac{\polv}{\oc}   \right ) f} = 0.
\]
\end{thm}
\begin{proof}
The mass conservation comes by the conservative formulation of the Vlasov equation, which writes
\begin{align}
\label{Equ114} 
\partial _t f + \divolx \left \{ f \frac{\ave{\polE}}{B}  \right \} + \partial _{x_3} \{ fv_3\} + \frac{q}{m} \partial _{v_3} \{f \ave{E_3}  \} = \ave{\qfpl} (f,f).
\end{align}
Multiplying \eqref{Equ114} by $v$ and integrating with respect to $(x,v)$ yield
\[
\frac{\md }{\md t} \intxv{v f} - \intxv{\frac{q}{m} f \ave{E_3}e_3 } = 0
\]
since for functions in $\ker {\cal T}$ we can write
\begin{align*}
\intxv{v \ave{\qfpl} (f,f)} & = \intxv{v \ave{\qfpl (f,f)}} = \intxv{\ave{v} \qfpl (f,f) } \\
& = \intxv{(0,0,v_3) \qfpl (f,f)} = 0.
\end{align*}
Using the Poisson equation and the identity
\begin{align}
\label{Equ115} \divx E \;E = \divx ( E \otimes E ) - \frac{1}{2} \nabla _x |E|^2 
\end{align}
we obtain, taking into account that $f \in \ker {\cal T}$
\begin{align*}
\intxv{\frac{q}{m}f \ave{E_3}} & = \frac{q}{m} \intxv{f E_3} = \frac{\eps _0}{m}\intx{E_3 \;\divx E} \\
& = \frac{\eps _0}{m}\intx{\left \{ \divx (E_3E ) - \frac{1}{2}\partial _{x_3} |E|^2\right \}} = 0
\end{align*}
and thus $\frac{\md}{\md t} \intxv{\;vf} = 0$. Multiplying \eqref{Equ114} by $\frac{m|v|^2}{2}$ and integrating with respect to $(x,v)$ yield
\[
\frac{\md }{\md t} \intxv{\frac{m|v|^2}{2}f }- \intxv{q f v_3 \ave{E_3}} = 0.
\]
Thanks to the continuity equation
\[
\partial _t \intv{f} + \divolx \intv{f \frac{\ave{\polE}}{B}} + \partial _{x_3} \intv{f v_3} = 0
\]
we have
\begin{align*}
& - q \intxv{f v_3 \ave{E_3}}  = - q \intxv{f v_3 {E_3}} =  q \intxv{f v_3 \partial _{x_3} \phi } \\
& = - q \intx{\phi \;\partial _{x_3} \intv{f v_3 }} 
= \intx{\phi \left \{\partial _t \rho + q\;\divolx\intv{f \frac{\ave{\polE}}{B}}    \right \}}\\
& = \intx{\phi \;\partial _t ( \eps _0 \divx E)} - q \;\intxv{\nabla _x \phi \cdot \frac{\ave{\polE}}{B}f } \\
& = - \intx{\nabla _x \phi \cdot \partial _t ( \eps _0  E)} + q \;\intx{\ave{E} \cdot \frac{\ave{\polE}}{B}f}\\
& = \frac{\eps _0}{2} \frac{\md }{\md t} \intx{|E|^2}
\end{align*}
implying 
\[
\frac{\md }{\md t} \left \{ \frac{\eps _0}{2} \int _{\R^3}{|E|^2\;\md x} + \intxv{m \frac{|v|^2}{2}f}\right \} = 0.
\]
By Remark \ref{OtherBalance} we know that
\begin{align*}
\intxv{\olx \ave{\qfpl}(f,f)} & = \intxv{\left ( \olx + \frac{\polv}{\oc} \right ) \ave{\qfpl}(f,f)} \\
& - \intxv{\frac{\polv}{\oc}\ave{\qfpl}(f,f)}= 0
\end{align*}
and therefore, multiplying \eqref{Equ114} by $\olx$ one gets
\[
\frac{\md }{\md t} \intxv{f \olx} - \intxv{f\frac{\ave{\polE}}{B}} = 0.
\]
Appealing to the Poisson equation we have
\begin{align*}
q\intxv{f \ave{\olE}} & = q\intxv{f {\olE}} = \int{\rho \olE } = \eps _0 \intx{\divx E \olE } = \\
& = \eps _0 \intx{\left (\divx (\olE \otimes E ) - \frac{1}{2} \gradolx |E|^2   \right )} = 0
\end{align*}
and therefore $\frac{\md}{\md t} \intxv{\;f \olx } = 0$. In particular the mean Larmor circle center is left invariant
\[
\frac{\md}{\md t} \left \{\frac{\intxv{\;f \left (\olx + \frac{\polv}{\oc}   \right ) }}{\intxv{\;f}}    \right \} = 0.
\]
By Remark \ref{OtherBalance} we know that 
\[
\intxv{\left (\oc ^2 |\olx|^2 +2 \oc  \olx \cdot \polv \right )\ave{\qfpl}(f,f) } = 0 
\]
and for any $\psi (x) \in C^1 _c (\R^3)$ we can write
\begin{align*}
\intx{\psi (x) \;\divolx \intv{f \olv }} & = - \intxv{f \;\olv \cdot \gradolx \psi } \\
& = - \intxv{f \;{\cal T} \psi } = \intxv{\psi \;{\cal T}f  } = 0
\end{align*}
saying that $\divolx \intv{f \olv} = 0$. Therefore we deduce 
\begin{align*}
& \frac{\md }{\md t} \intxv{\left (\frac{|\olx|^2}{2} + \olx \cdot \frac{\polv}{\oc}    \right ) f}  =\intxv{f \frac{\ave{\polE}}{B} \cdot \left ( \olx + \frac{\polv}{\oc} \right )} \\
& = \intxv{f \frac{{\polE}}{B} \cdot \left (  \olx + \frac{\polv}{\oc} \right ) } = - \intxv{f \frac{{\olE}}{B} \cdot  \polx } + \intx{\frac{\phi (x)}{B \oc} \divolx \intv{\!\!f \olv }} \\
& = - \frac{\eps _0}{qB} \intx{\polx \cdot \olE \;\divx E } = - \frac{\eps _0}{qB} \intx{\polx \cdot \left (\divx (\olE \otimes E ) - \frac{1}{2} \gradolx |E|^2   \right )}\\
& = - \frac{\eps _0}{qB} \intx{\left \{x_2 \sum _{j = 1} ^3 \partial _{x_j} (E_1 E_j) - x_1 \sum _{j = 1} ^3 \partial _{x_j} (E_2 E_j)\right \}} \\
& = \frac{\eps _0}{qB} \intx{\sum _{j = 1} ^3  \left \{ \frac{\partial {x_2}}{\partial {x_j}} E_1 E_j - \frac{\partial {x_1}}{\partial {x_j}} E_2 E_j \right \}
}\\
&= 0
\end{align*}
implying that the mean Larmor circle power (with respect to the origin) is left invariant.
\end{proof}


\appendix
\section{Proof of Theorem \ref{MainResultFP}}
\label{FP}
\begin{proof}
The \fp kernel being a second order differential operator, we appeal twice to Proposition \ref{ComDivAve}. For any $f \in \ker {\cal T}$, taking $\xi _v = M \nabla _v \left ( f/M\right )$ yields
\begin{align}
\label{EquFP1}
& \ave{\divv \left ( M \nabla _v \ffm \right ) } \\
&  = \divoolx \left \{\ave{M\; {^\perp \nabla_{\olv}} \ffm } + \ave{M \;\nabla _\olv \ffm \cdot \frac{\polv}{|\olv|}} \frac{\olv}{|\olv|}  - \ave{M \;\nabla _\olv \ffm \cdot \frac{\olv}{|\olv|}} \frac{\polv}{|\olv|}\right \} \nonumber \\
& + \divolv \left \{\ave{M \;\nabla _\olv \ffm \cdot \frac{\polv}{|\olv|}} \frac{\polv}{|\olv|} + \ave{M \;\nabla _\olv \ffm \cdot \frac{\olv}{|\olv|}} \frac{\olv}{|\olv|}\right \}  + \partial _{v_3} \ave{ M \partial _{v_3} \ffm }.\nonumber
\end{align}
Since $\partial _{v_3}$ commutes with $\ave{\cdot}$ (cf. Proposition \ref{ComDerAve}) we deduce
\[
\partial _{v_3} \ave{ M \partial _{v_3} \ffm } = \partial _{v_3} \left \{M \ave{\partial _{v_3} \ffm}    \right \} = \partial _{v_3} \left \{M \partial _{v_3} \ave{\frac{f}{M}}   \right \} = \partial _{v_3} \left \{M \partial _{v_3} \ffm   \right \}.
\]
It remains to compute the averages
\[
\ave{M \; {^\perp \nabla} _\olv \ffm },\;\;\ave{ M \; \nabla _\olv \ffm \cdot \frac{\polv}{|\olv|}},\;\;\ave{M \;  \nabla _\olv \ffm \cdot \frac{\olv}{|\olv|}}.
\]
These averages come easily, thanks to Proposition \ref{ComDivAve}, observing that
\[
\partial _{v_1} \ffm = \divv \left ( \frac{f}{M}, 0, 0\right ), \;\;\partial _{v_2} \ffm = \divv \left ( 0, \frac{f}{M}, 0\right )
\]
\[
\nabla _\olv \ffm \cdot \polv = \divolv \left ( \frac{f}{M} \polv \right ),\;\;\nabla _\olv \ffm \cdot \olv = \divolv \left ( \frac{f}{M} \olv \right ) - 2 \frac{f}{M}.
\]
We obtain
\begin{align}
\label{EquFP2} \ave{M \; {^\perp \nabla} _\olv \ffm } = M \;\nabla _{\omega _c \olx} \ffm
\end{align}
\begin{align}
\label{EquFP3} \ave{ M \; \nabla _\olv \ffm \cdot \frac{\polv}{|\olv|}} = M \;\olvpolv \cdot \nabla _{\omega _c x, v} \ffm
\end{align}
\begin{align}
\label{EquFP4} \ave{ M \; \nabla _\olv \ffm \cdot \frac{\olv}{|\olv|}} = - M \;\polvolv \cdot \nabla _{\omega _c x, v} \ffm.
\end{align}
Our conclusion follows by combining \eqref{EquFP1}, \eqref{EquFP2}, \eqref{EquFP3}, \eqref{EquFP4}. The diffusion matrix ${\cal L}$ is positive and for any $\xi = (\xi _x, \xi _v) \in \R^6$ we have
\[
{\cal L} \xi \cdot \xi = |\xi _\olx |^2 + |\xi _\olx - {^\perp \xi }_\olv |^2 + (\xi _{v_3} )^2 \geq 0
\]
with equality iff $\xi _\olx = \xi _\olv =(0, 0)$ and $\xi _{v_3} = 0$.
\end{proof}

\section{Proofs of Propositions \ref{SecondFormula}, \ref{ThirdFormula}}
\label{A}
\begin{proof} (of Proposition \ref{SecondFormula})
We follow the same arguments as those in the proof of Proposition \ref{FirstFormula}. The details are left to the reader.\\
1. Using cylindrical coordinates we obtain 
\begin{align}
J & :=\ave{\intvp{D(v,\vp)f(x,\vp)}} = 
\avetpi \int _\R \int _{-\pi} ^\pi \int _{\R_+} D(|\olv| e ^{i\alpha}, v_3, \rp e ^{i(\varphi + \alpha)}, \vtp) \nonumber \\
& \times f \left (\olx + \frac{\polv}{\oc} - \frac{^\perp \{|\olv| e ^{i\alpha}\}}{\oc}, x_3, \rp e ^{i (\varphi + \alpha)},\vtp   \right )\rp \md \rp \md \varphi \md \vtp \md \alpha \nonumber \\
& = \avetpi \int _\R \int _{-\pi} ^\pi \int _{\R_+} D(|\olv| e ^{i\alpha}, v_3, \rp e ^{i(\varphi + \alpha)}, \vtp) \nonumber \\
& \times g \left (\olx + \frac{\polv}{\oc} - \frac{^\perp \{ |\olv| e ^{i \alpha}\}}{\oc} + \frac{^\perp \{\rp  e ^{i (\varphi + \alpha)}  \}}{\oc}, x_3, \rp, \vtp   \right )
\rp \md \rp \md \varphi \md \vtp \md \alpha 
\end{align}
where in the last equality we have used the constraint $f \in \ker {\cal T}$ {\it i.e.}, there is $g$ such that
\[
f(x,v) = g\left ( \olx + \frac{\polv}{\oc}, x_3, |\olv|, v_3\right ),\;\;(x,v) \in \R ^3 \times \R ^3.
\]
We have $\rp  e ^{i (\varphi + \alpha)} - |\olv| e ^{i \alpha} = l  e ^{i (\psi + \alpha)}$ where $l^2 = r ^2 + (\rp )^2 - 2 r \rp \cos \varphi $, $r = |\olv|$ and $(\rp)^2 = r^2 + l^2 + 2 r l \cos \psi$. Notice that $\psi \in (0,\pi)$ if $\varphi \in (0,\pi)$ and $\psi \in (-\pi, 0)$ if $\varphi \in (-\pi, 0)$. Also $\psi = \psi (\varphi)$ is odd with respect to $\varphi$ that is $\psi (- \varphi) = - \psi (\varphi)$. 
By hypothesis we deduce that 
\[
D(re ^{i\alpha}, v_3, \rp  e ^{i (\varphi + \alpha)}, \vtp) = \tD (r, v_3, \rp, \vtp, \varphi)\;r e ^{i\alpha} + \tDp (r, v_3, \rp, \vtp, \varphi)  \;\rp e ^{i (\varphi + \alpha)}
\]
and thus, if we denote by $R(\alpha)$ the rotation of angle $\alpha$ in $\R^2$ one gets
\begin{align*}
J & =\avetpi \int _\R \int _{-\pi} ^\pi \int _{\R_+} R(\alpha + \psi) [\tD (r, v_3, \rp, \vtp, \varphi)\;r e ^{-i\psi} + \tDp (r, v_3, \rp, \vtp, \varphi)  \;\rp e ^{i (\varphi -\psi)}  ] \nonumber \\
& \times g \left (\olx + \frac{\polv}{\oc} + \frac{^\perp \{ l e ^{i (\alpha + \psi)}\}}{\oc}, x_3, \rp, \vtp   \right )
\rp \md \rp \md \varphi \md \vtp \md \alpha \\
= & \avetpi \int _\R \int _{-\pi} ^\pi \int _{\R_+} R(\alpha ) [\tD (r, v_3, \rp, \vtp, \varphi)\;r e ^{-i\psi} + \tDp (r, v_3, \rp, \vtp, \varphi)  \;\rp e ^{i (\varphi -\psi)}  ] \nonumber \\
& \times g \left (\olx + \frac{\polv}{\oc} + \frac{^\perp \{ l e ^{i \alpha }\}}{\oc}, x_3, \rp, \vtp   \right )
\rp \md \rp \md \varphi \md \vtp \md \alpha.
\end{align*}
Using the symmetry of $\psi$ with respect to $\varphi$ and changing $\varphi$ against $l$ yield
\begin{align*}
J & = \avetpi \int _\R \int _{0} ^\pi \int _{\R_+} R(\alpha ) [\tD (r, v_3, \rp, \vtp, \varphi)\;r e ^{-i\psi} + \tD (r, v_3, \rp, \vtp, -\varphi)\;r e ^{i\psi} \\
&  + \tDp (r, v_3, \rp, \vtp, \varphi)  \;\rp e ^{i (\varphi -\psi)} + \tDp (r, v_3, \rp, \vtp, - \varphi)  \;\rp e ^{-i (\varphi -\psi)} ] \nonumber \\
& \times g \left (\olx + \frac{\polv}{\oc} + \frac{^\perp \{ l e ^{i \alpha }\}}{\oc}, x_3, \rp, \vtp   \right )
\rp \md \rp \md \varphi \md \vtp \md \alpha \\
& = \frac{1}{\pi} \int _0 ^{2\pi}\!\!\!\!  \int _\R \int _{|r - \rp|} ^{(r + \rp)} \int _{\R_+} R(\alpha ) [\tD (r, v_3, \rp, \vtp, \varphi)\;r e ^{-i\psi} + \tD (r, v_3, \rp, \vtp, -\varphi)\;r e ^{i\psi} \\
&  + \tDp (r, v_3, \rp, \vtp, \varphi)  \;\rp e ^{i (\varphi -\psi)} + \tDp (r, v_3, \rp, \vtp, - \varphi)  \;\rp e ^{-i (\varphi -\psi)} ] \nonumber \\
& \times g \left (\olx + \frac{\polv}{\oc} + \frac{^\perp \{ l e ^{i \alpha }\}}{\oc}, x_3, \rp, \vtp   \right )
\frac{\rp \md \rp \;l  \md l  \;\md \vtp \md \alpha}{\sqrt{l^2 - (r - \rp)^2\;} \sqrt{(r + \rp ) ^2 - l^2}}.
\end{align*}
Notice that $R(\alpha) = e_1 \otimes e ^{-i \alpha} - e_2 \otimes {^\perp e ^{-i \alpha}}$ and for any $\alphap \in [0,2\pi)$ we have
\[
g \left (\olx + \frac{\polv}{\oc} + \frac{^\perp \{l e ^{i \alpha}\}}{\oc}, x_3, \rp, \vtp\right ) = f \left (\olx + \frac{\polv}{\oc} + \frac{^\perp \{l e ^{i \alpha}\}}{\oc} - \frac{^\perp \{\rp e ^{i \alphap}\}}{\oc}, x_3, \rp e ^{i \alphap}, \vtp   \right ).
\]
Performing the change of coordinates $\vp = (\rp e ^{i \alphap}, \vtp)$ and $- z = {^\perp \{ l e ^{i \alpha}\}}$ leads to
\begin{align*}
J & = \frac{1}{2\pi ^2} \inttpi \int _{\R_+} \int _{\R} \inttpi \int _{\R_+} \{e_1 \otimes e ^{-i \alpha} - e_2 \otimes {^\perp e ^{-i \alpha}}\}[
\tD ( \varphi)\;r e ^{-i\psi} + \tD ( -\varphi)\;r e ^{i\psi} \\
&  + \tDp (\varphi)  \;\rp e ^{i (\varphi -\psi)} + \tDp ( - \varphi)  \;\rp e ^{-i (\varphi -\psi)} 
] \\
& \times f \left (\olx + \frac{\polv}{\oc} + \frac{^\perp \{l e ^{i \alpha}\}}{\oc} - \frac{^\perp \{\rp e ^{i \alphap}\}}{\oc}, x_3, \rp e ^{i \alphap}, \vtp   \right )\chi (r, \rp, - {^\perp \{l e ^{i\alpha}\}})\rp \md \rp \md \alphap \md \vtp l \md l \md \alpha\\
& = \int _{\R^2} \int _{\R^3} {\cal D}( |\olv|, v_3, |\olvp|, \vtp, z)f \left (\olx + \frac{\polv}{\oc} -\frac{z}{\oc} - \frac{\polvp}{\oc}, x_3, \vp   \right )\;\md \vp \md z \\
& = \oc ^2 \intolxpvp{{\cal D}( |\olv|, v_3, |\olvp|, \vtp, (\oc \olx + \polv) - (\oc \olxp + \polvp))f(x_1^\prime, x_2 ^\prime, x_3, \vp)}.
\end{align*}
The statement in 2. follows similarly.
\end{proof}
\begin{proof} (of Proposition \ref{ThirdFormula}) 
Observe that 
\[
\ave{f}_{\sigma S} = \ave{\intvp{f(x,\vp) \sigma (|v - \vp|)}}\;I - \ave{\intvp{f(x,\vp) \sigma (|v - \vp|)\frac{(v - \vp ) \otimes (v - \vp)}{|v - \vp|^2}}}.
\]
Applying Proposition \ref{FirstFormula} with $C = 1$ one gets $\ave{\intvp{f \sigma}} = \oc ^2 \intolxpvp{\;\sigma f \chi}$. It remains to compute the second average. Using cylindrical coordinates and the constraint ${\cal T } f = 0$ {\it i.e.,} $f(x,v) = g \left ( \olx + \frac{\polv}{\oc}, x_3, |\olv|, v_3\right )$ we obtain 
\begin{align*}
K & :=\ave{\intvp{f(x,\vp)\sigma (|v - \vp|)\frac{(v - \vp ) \otimes (v - \vp)}{|v - \vp|^2}}} \\
& = 
\avetpi \int _\R \int _{-\pi} ^\pi \int _{\R_+} \frac{\sigma(\;| \;(|\olv| e ^{i \alpha} - \rp e ^{i (\varphi + \alpha )}, v_3 - \vtp     )\;|\;)}{|\; (|\olv| e ^{i \alpha} - \rp e ^{i (\varphi + \alpha )}, v_3 - \vtp     )\;|^2}\left (|\olv| e ^{i \alpha} - \rp e ^{i (\varphi + \alpha )}, v_3 - \vtp    \right ) ^{\otimes 2}
\\
& \times f \left (\olx + \frac{\polv}{\oc} - \frac{^\perp \{|\olv| e ^{i\alpha}\}}{\oc}, x_3, \rp e ^{i (\varphi + \alpha)},\vtp   \right )\rp \md \rp \md \varphi \md \vtp \md \alpha  \\
& = \avetpi \int _\R \int _{-\pi} ^\pi \int _{\R_+}
\frac{\sigma(\;| \;(|\olv| e ^{i \alpha} - \rp e ^{i (\varphi + \alpha )}, v_3 - \vtp     )\;|\;)}{|\; (|\olv| e ^{i \alpha} - \rp e ^{i (\varphi + \alpha )}, v_3 - \vtp     )\;|^2}\left (|\olv| e ^{i \alpha} - \rp e ^{i (\varphi + \alpha )}, v_3 - \vtp    \right ) ^{\otimes 2}\\
& \times g \left (\olx + \frac{\polv}{\oc} - \frac{^\perp \{ |\olv| e ^{i \alpha}\}}{\oc} + \frac{^\perp \{\rp  e ^{i (\varphi + \alpha)}  \}}{\oc}, x_3, \rp, \vtp   \right )
\rp \md \rp \md \varphi \md \vtp \md \alpha.
\end{align*}
We introduce $l, \psi $ such that 
$\rp  e ^{i (\varphi + \alpha)} - |\olv| e ^{i \alpha} = l  e ^{i (\psi + \alpha)}$. We have the relations $l^2 = r ^2 + (\rp )^2 - 2 r \rp \cos \varphi $, $r = |\olv|$ and $(\rp)^2 = r^2 + l^2 + 2 r l \cos \psi$. Since $l, \psi$ are not depending on $\alpha$ one gets
\begin{align*}
K & =\avetpi \int _\R \int _{-\pi} ^\pi \int _{\R_+}  
\frac{\sigma(\sqrt{l^2 +( v_3 - \vtp     )^2 }\;)}{l^2 +( v_3 - \vtp     )^2}\left (l e ^{i (\alpha + \psi)},  \vtp - v_3    \right ) ^{\otimes 2}
\\
& \times g \left (\olx + \frac{\polv}{\oc} + \frac{^\perp \{ l e ^{i (\alpha + \psi)}\}}{\oc}, x_3, \rp, \vtp   \right )
\rp \md \rp \md \varphi \md \vtp \md \alpha \\
= & \avetpi \int _\R \int _{-\pi} ^\pi \int _{\R_+}
\frac{\sigma(\sqrt{l^2 +( v_3 - \vtp     )^2 }\;)}{l^2 +( v_3 - \vtp     )^2}\left (l e ^{i \alpha },  \vtp - v_3    \right ) ^{\otimes 2}
\\
& \times g \left (\olx + \frac{\polv}{\oc} + \frac{^\perp \{ l e ^{i \alpha }\}}{\oc}, x_3, \rp, \vtp   \right )
\rp \md \rp \md \varphi \md \vtp \md \alpha.
\end{align*}
Since $l$ is even with respect to $\varphi$ we obtain, after changing  $\varphi $ against $l$
\begin{align*}
K & = \frac{1}{\pi} \int _0 ^{2\pi}\!\! \!\!\int _\R \int _{0} ^\pi \int _{\R_+} \frac{\sigma(\sqrt{l^2 +( v_3 - \vtp     )^2 }\;)}{l^2 +( v_3 - \vtp     )^2}\left (l e ^{i \alpha },  \vtp - v_3    \right ) ^{\otimes 2}\\
& \times g \left (\olx + \frac{\polv}{\oc} + \frac{^\perp \{ l e ^{i \alpha }\}}{\oc}, x_3, \rp, \vtp   \right )
\rp \md \rp \md \varphi \md \vtp \md \alpha \\
& = \frac{2}{\pi} \int _0 ^{2\pi}\!\!\!\!  \int _\R \int _{|r - \rp|} ^{(r + \rp)} \int _{\R_+} 
\frac{\sigma(\sqrt{l^2 +( v_3 - \vtp     )^2 }\;)}{l^2 +( v_3 - \vtp     )^2}\left (l e ^{i \alpha },  \vtp - v_3    \right ) ^{\otimes 2}\\
& \times g \left (\olx + \frac{\polv}{\oc} + \frac{^\perp \{ l e ^{i \alpha }\}}{\oc}, x_3, \rp, \vtp   \right )
\frac{\rp \md \rp \;l  \md l  \;\md \vtp \md \alpha}{\sqrt{l^2 - (r - \rp)^2\;} \sqrt{(r + \rp ) ^2 - l^2}}.
\end{align*}
For any $\alphap \in [0,2\pi)$ we have
\[
g \left (\olx + \frac{\polv}{\oc} + \frac{^\perp \{l e ^{i \alpha}\}}{\oc}, x_3, \rp, \vtp\right ) = f \left (\olx + \frac{\polv}{\oc} + \frac{^\perp \{l e ^{i \alpha}\}}{\oc} - \frac{^\perp \{\rp e ^{i \alphap}\}}{\oc}, x_3, \rp e ^{i \alphap}, \vtp   \right ).
\]
Performing the change of coordinates $\vp = (\rp e ^{i \alphap}, \vtp)$ and $ z = -{^\perp \{ l e ^{i \alpha}\}}$ leads to
\begin{align*}
K & =  \inttpi \int _{\R_+} \int _{\R} \inttpi \int _{\R_+} \frac{\sigma(\sqrt{l^2 +( v_3 - \vtp     )^2 }\;)}{l^2 +( v_3 - \vtp     )^2}\left (l e ^{i \alpha },  \vtp - v_3    \right ) ^{\otimes 2}\\
& \times f \left (\olx + \frac{\polv}{\oc} + \frac{^\perp \{l e ^{i \alpha}\}}{\oc} - \frac{^\perp \{\rp e ^{i \alphap}\}}{\oc}, x_3, \rp e ^{i \alphap}, \vtp   \right )\chi (r, \rp, - {^\perp \{l e ^{i\alpha}\}})\rp \md \rp \md \alphap \md \vtp l \md l \md \alpha\\
& = \int _{\R^2} \int _{\R^3} \frac{\sigma (\sqrt{|z|^2 + (v _3 - \vtp)^2}\;)}{|z|^2 + (v _3 - \vtp)^2} ( \;{^\perp z}, \vtp - v_3 ) ^{\otimes 2}
 f \left (\olx + \frac{\polv}{\oc} -\frac{z}{\oc} - \frac{\polvp}{\oc}, x_3, \vp   \right )\chi \;\md \vp \md z \\
& = \oc ^2 \intolxpvp{\sigma(\sqrt{|z|^2 + (v _3 - \vtp)^2}\;)f (x_1 ^\prime, x_2 ^\prime, x_3, \vp) \;\chi \; \frac{ ( \;{^\perp z}, \vtp - v_3 ) ^{\otimes 2}}{|z|^2 + (v _3 - \vtp)^2}}.
\end{align*}
Finally we obtain
\[
\ave{f}_{\sigma S} = \oc ^2 \intolxpvp{\sigma(\sqrt{|z|^2 + (v _3 - \vtp)^2}\;)f (x_1 ^\prime, x_2 ^\prime, x_3, \vp) \;\chi \;S( \;({^\perp z}, \vtp - v_3 )\;)}.
\]
\end{proof}

\section{Proofs of Propositions \ref{GainLFP}, \ref{LossLFP}}
\label{B}

\begin{proof} (of Proposition \ref{GainLFP}) 
Let us introduce the notation
\[
\xi _v(x,v) = \intvp{\sigma (|v - \vp|) S(v - \vp)f (x, \vp) \nabla _v f (x,v)}.
\]
Thanks to Proposition \ref{ComDivAve} we have
\begin{align*}
\ave{\qfpl ^+ (f,f)} & = \ave{\divv \;\xi _v} = \frac{1}{\oc} \divolx {\left \{ \ave{{^\perp \xi _\olv}} + \ave{\xi _\olv \cdot \frac{\polv}{|\olv|}} \frac{\olv}{|\olv|} - 
\ave{\xi _\olv \cdot \frac{\olv}{|\olv|}} \frac{\polv}{|\olv|}  \right \}}  \\
& + \divolv\left \{\ave{\xi _\olv \cdot \frac{\polv}{|\olv|}} \frac{\polv}{|\olv|} + 
\ave{\xi _\olv \cdot \frac{\olv}{|\olv|}} \frac{\olv}{|\olv|}
\right \} + \partial _{v_3} \ave{\xi _{v_3}}.
\end{align*}
We need to compute $\ave{\xi _\olv}, \ave{\xi _\olv \cdot \olv}, \ave{\xi _\olv \cdot \polv}$. By Proposition \ref{Field} we know that 
$\sum _{i = 0} ^5 b ^i \otimes \gradxv \psi _i = I$ and thus
\begin{align*}
\partial _{v_1} f = \sum _{i = 0} ^5 \partial _{v_1} \psi _i \;b ^i \cdot \gradxv f = \frac{v_2}{\oc |\olv|^2} \;b ^0 \cdot \gradxv f - \frac{1}{\oc} b ^2 \cdot \gradxv f + \frac{v_1}{|\olv|} \;b ^4 \cdot \gradxv f.
\end{align*} 
Similarly we have
\begin{align*}
\partial _{v_2} f = \sum _{i = 0} ^5 \partial _{v_2} \psi _i \;b ^i \cdot \gradxv f = -\frac{v_1}{\oc |\olv|^2} \;b ^0 \cdot \gradxv f + \frac{1}{\oc} b ^1 \cdot \gradxv f + \frac{v_2}{|\olv|} \;b ^4 \cdot \gradxv f
\end{align*} 
leading to
\[
\nabla _v f = b ^0 \cdot \gradxv f\;\frac{(\polv, 0)}{\oc |\olv|^2} + \left (- \frac{^\perp \gradolx f}{\oc}, \partial _{v_3} f     \right ) + b ^4 \cdot \gradxv f\;\frac{(\olv, 0)}{ |\olv|}. 
\]
Taking into account that all derivations $b^i \cdot \gradxv, 0\leq i \leq 5$ leave invariant $\ker {\cal T}$, cf. Proposition \ref{Field}, we obtain 
\begin{align*}
& \ave{\xi _v}   = \ave{f, (\polv, 0)}_{\sigma S} \frac{b ^0 \cdot \gradxv f}{\oc |\olv|^2} + \ave{f}_{\sigma S} \left (- \frac{^\perp \gradolx f}{\oc}, \partial _{v_3} f     \right ) + \ave{f, (\olv, 0)}_{\sigma S} \frac{b ^4 \cdot \gradxv f}{ |\olv|} \\
& =  \left \{  \frac{\ave{f, (\polv, 0)}_{\sigma S} }{|\olv|}\otimes \left ( \frac{(\olv,0)}{|\olv|}, \frac{(\polv, 0)}{|\olv|}   \right )  
- \frac{\ave{f, (\olv, 0)}_{\sigma S} }{|\olv|} \otimes \left ( \frac{(\polv,0)}{|\olv|}, -\frac{(\olv, 0)}{|\olv|}   \right ) \right \}\gradoxv f \\
& -  \ave{f}_{\sigma S} (E, - e_3 \otimes e_3)\gradoxv f 
\end{align*}
where the lines of the matrix $E \in {\cal M}_3 (\R)$ are $e_2, - e_1, 0$. Similarly, thanks to the identities $\ave{f, (\olv, 0), (\polv, 0)}_{\sigma S}  = \ave{f, (\polv, 0), (\olv, 0)}_{\sigma S} = 0$ we obtain
\begin{align*}
& \ave{\xi _\olv \cdot \frac{\olv}{|\olv|}}  = \ave{\xi _v \cdot \frac{(\olv, 0)}{|\olv|}} \\
& = - \frac{\ave{f, (\olv, 0)}_{\sigma S}}{|\olv|}\cdot \left (  \frac{^\perp \gradolx f}{\oc}, -\partial _{v_3} f \right ) 
+ \frac{\avess{f, (\olv,0), (\olv, 0)}}{|\olv |^2} \;b ^4 \cdot \gradxv f \\
& = \left \{-\frac{\ave{f, (\olv, 0), (\olv, 0)}_{\sigma S}}{|\olv|^2}  \left (\frac{(\polv, 0)}{|\olv|},-\frac{(\olv, 0)}{|\olv|}   \right )
+ \left ( E \frac{\avess{f, (\olv, 0)}}{|\olv|}, e_3 \otimes e_3 \frac{\avess{f, (\olv, 0)}}{|\olv|}\right ) 
\right \}\\
& \cdot \gradoxv f
\end{align*}
and
\begin{align*}
& \ave{\xi _\olv \cdot \frac{\polv}{|\olv|}}  = \ave{\xi _v \cdot \frac{(\polv, 0)}{|\olv|}} \\
& =  \frac{\avess{f, (\polv,0), (\polv, 0)}}{|\olv |^2} \;\frac{b ^0 \cdot \gradxv f}{\oc |\olv|} 
- \frac{\ave{f, (\polv, 0)}_{\sigma S}}{|\olv|}\cdot \left (  \frac{^\perp \gradolx f}{\oc}, -\partial _{v_3} f \right )   \\
& = 
\left \{\frac{\ave{f, (\polv, 0), (\polv, 0)}_{\sigma S}}{|\olv|^2}  \left (\frac{(\olv, 0)}{|\olv|},\frac{(\polv, 0)}{|\olv|}   \right )
+ \left ( E \frac{\avess{f, (\polv, 0)}}{|\olv|}, 0 \right ) 
\right \}\cdot \gradoxv f.
\end{align*}
In the last equality we have taken into account that
\[ 
e_3 \otimes e_3 \frac{\avess{f, (\polv, 0)}}{|\olv|} = 0.
\]
Clearly $\ave{\qfpl ^+ (f, f)}$ has the form in \eqref{Equ70} with
$
A^+ = \left (
\begin{array}{rrr}
A_{xx} ^+  &  A_{xv} ^+\\
A_{vx}^+  &\;\;\;   A_{vv}^+
\end{array}
\right )
$
where
\begin{align}
\label{Equ72}
(A_{xx} ^+, A_{xv} ^+)  & = \frac{r - \rp \cos \varphi }{|z|} \;\frac{({^\perp z}, 0)}{|z|} \otimes  \olvpolv \nonumber \\
& - 
\frac{(r - \rp \cos \varphi ) (v_3 - \vtp )^2}{|z|[\;|z|^2 + (v_3 - \vtp)^2\;]}\;\frac{(z,0)}{|z|} \otimes \polvolv \nonumber \\
& + \left ( 1 - \frac{(\rp )^2 \sin ^2 \varphi }{|z|^2 + (v_3 - \vtp)^2}  \right )  \frac{(\olv, 0)}{|\olv|} \otimes \olvpolv \nonumber \\
& + \frac{r - \rp \cos \varphi }{|z|} \;\frac{(\olv, 0)}{|\olv|} \otimes \left (\frac{({^\perp z}, 0)}{|z|}, 0 \right ) \nonumber \\
& +  \left ( 1 - \frac{(r - \rp \cos \varphi)^2 }{|z|^2 + (v_3 - \vtp)^2}  \right )  \frac{(\polv, 0)}{|\olv|} \otimes \polvolv \nonumber \\
& - \frac{(r - \rp \cos \varphi)(v_3 - \vtp) }{|z|\sqrt{|z|^2 + (v_3 - \vtp)^2}}  \;  \frac{(\polv, 0)}{|\olv|} 
\otimes \vtz \nonumber \\
& + ( {^t E} S (\;({^\perp z}, \vtp - v_3)\; ) E, ES(\;({^\perp z}, \vtp - v_3)\; ) e_3 \otimes e_3 )
\end{align}
and
\begin{align}
\label{Equ73}
(A_{vx} ^+, A_{vv} ^+)  & =   \left ( 1 - \frac{(\rp )^2 \sin ^2 \varphi }{|z|^2 + (v_3 - \vtp)^2}  \right )  \frac{(\polv, 0)}{|\olv|} \otimes \olvpolv \nonumber \\
& + \frac{r - \rp \cos \varphi }{|z|} \;\frac{(\polv, 0)}{|\olv|} \otimes \left (\frac{({^\perp z}, 0)}{|z|}, 0 \right ) \nonumber \\
& - \left ( 1 - \frac{( r - \rp \cos \varphi )^2}{|z|^2 + (v_3 - \vtp)^2}  \right )  \frac{(\olv, 0)}{|\olv|} \otimes \polvolv \nonumber \\
& + \frac{(r - \rp \cos \varphi)(v_3 - \vtp) }{|z|\sqrt{|z|^2 + (v_3 - \vtp)^2}}  \;  \frac{(\olv, 0)}{|\olv|} 
\otimes \vtz \nonumber \\
& +
\frac{(r - \rp \cos \varphi ) (v_3 - \vtp )}{|z|^2 + (v_3 - \vtp)^2}\;e_3 \otimes \polvolv \nonumber \\
& + ( - e_3 \otimes e_3  S (\;({^\perp z}, \vtp - v_3)\; )  E, e_3 \otimes e_3S(\;({^\perp z}, \vtp - v_3)\; ) e_3 \otimes e_3 ).
\end{align}
It is easily seen that the matrix $A^+$ writes
\begin{align}
\label{Equ74}
A^+ & = \left ( 1 - \frac{(\rp )^2 \sin ^2 \varphi }{|z|^2 + (v_3 - \vtp)^2}  \right ) \olvpolv ^{\otimes 2}\nonumber \\
& + \left ( 1 - \frac{( r - \rp \cos \varphi )^2}{|z|^2 + (v_3 - \vtp)^2}  \right ) \polvolv ^{\otimes 2}\nonumber \\
& + \frac{r - \rp \cos \varphi }{|z|} \left ( \frac{({^\perp z}, 0)}{|z|}, 0\right )  \otimes  \olvpolv \nonumber \\
& + \frac{r - \rp \cos \varphi }{|z|} \olvpolv \otimes \left (\frac{({^\perp z}, 0)}{|z|}, 0 \right ) \nonumber \\
& - 
\frac{(r - \rp \cos \varphi ) (v_3 - \vtp )}{|z|\sqrt{|z|^2 + (v_3 - \vtp)^2}}\vtz \otimes \polvolv \nonumber \\
& - 
\frac{(r - \rp \cos \varphi ) (v_3 - \vtp )}{|z|\sqrt{|z|^2 + (v_3 - \vtp)^2}}\polvolv \otimes \vtz\nonumber \\
& + B^+ = A_1 ^+ + A_2 ^+ +A_3 ^+ +A_4 ^+ +A_5 ^+ +A_6 ^+ +B ^+
\end{align}
where 
\[
B^+ = \left ( 
\begin{array}{lll}
{^t E} S (\;({^\perp z}, \vtp - v_3)\; )  E  &  \;\;ES(\;({^\perp z}, \vtp - v_3)\; ) e_3 \otimes e_3\\
- e_3 \otimes e_3  S (\;({^\perp z}, \vtp - v_3)\; )  E  & \;\;e_3 \otimes e_3S(\;({^\perp z}, \vtp - v_3)\; ) e_3 \otimes e_3
\end{array}
\right).
\]
Observe that for any $z \in \R^2, v_3, \vtp \in \R$ the family 
\[
\frac{(z,0)}{|z|},\;\;\left (\frac{v_3 - \vtp}{\sqrt{|z|^2 + (v _3 - \vtp)^2}}\frac{^\perp z}{|z|}, \frac{|z|}{\sqrt{|z|^2 + (v _3 - \vtp)^2}}  \right ),\;\;\frac{({^\perp z}, \vtp - v_3)}{\sqrt{|z|^2 + (v _3 - \vtp)^2}}
\]
is a orthonormal basis of $\R^3$. Therefore we have
\begin{align*}
S(\;({^\perp z}, \vtp - v_3)\;)= \frac{(z,0)}{|z|}\otimes \frac{(z,0)}{|z|} + \left (\frac{v_3 - \vtp}{\sqrt{|z|^2 + (v _3 - \vtp)^2}}\frac{^\perp z}{|z|}, \frac{|z|}{\sqrt{|z|^2 + (v _3 - \vtp)^2}}  \right )^{\otimes 2}
\end{align*}
implying that for any $\xi = (\xi _x, \xi _v) \in \R^6$
\begin{align*}
& (B^+ \xi, \xi)  = S(\;({^\perp z}, \vtp - v_3)\;)  : ({^\perp \xi }_{\olx}, - \xi _{v_3}) \otimes  ({^\perp \xi} _{\olx}, - \xi _{v_3}) \\
& =  \left [\left ( \frac{({^\perp z},0)}{|z|}, 0\right ) ^{\otimes 2}     + 
\left (\frac{v_3 - \vtp}{\sqrt{|z|^2 + (v _3 - \vtp)^2}}\frac{( z,0)}{|z|}, -\frac{|z|e_3}{\sqrt{|z|^2 + (v _3 - \vtp)^2}}  \right )^{\otimes 2}
\right ] : \xi \otimes \xi 
\end{align*}
and thus
\begin{equation}
\label{Equ77} B^+ = \left ( \frac{({^\perp z},0)}{|z|}, 0\right ) ^{\otimes 2}     + 
\left (\frac{v_3 - \vtp}{\sqrt{|z|^2 + (v _3 - \vtp)^2}}\frac{( z,0)}{|z|}, -\frac{|z|e_3}{\sqrt{|z|^2 + (v _3 - \vtp)^2}}  \right )^{\otimes 2} = : B_1 ^ + + B_2 ^+.
\end{equation}
Observe that 
\begin{align}
\label{Equ78} A_1 ^+ + A_3 ^+ + A_4 ^+ + B_1 ^+ & = \frac{(\rp )^2 \sin ^2 \varphi \;(v_3 - \vtp )^2}{|z|^2 [\;|z|^2 + (v_3 - \vtp )^2 \;]}\;\olvpolv ^{\otimes 2} \nonumber \\
& + \left [\frac{r - \rp \cos \varphi }{|z|}\olvpolv + \left ( \frac{({^\perp z},0)}{|z|}, 0\right )     \right ] ^{\otimes 2}
\end{align}
since $(\frac{r - \rp \cos \varphi }{|z|} )^2 = 1 - \frac{(\rp )^2 \sin ^2 \varphi }{|z|^2 }$ and
\begin{align}
\label{Equ79} & A_2 ^+ + A_5 ^+ + A_6 ^+ + B_2 ^+  = \frac{(\rp )^2 \sin ^2 \varphi }{|z|^2 } \polvolv ^ {\otimes 2}  + \nonumber \\
& \left [ \frac{(r - \rp \cos \varphi )(v_3 - \vtp)}{|z|\sqrt{|z|^2 + (v _3 - \vtp)^2}}\polvolv - 
\vtz \right ] ^{\otimes 2}.
\end{align}
Our conclusion follows by combining \eqref{Equ74}, \eqref{Equ77}, \eqref{Equ78}, \eqref{Equ79}.
\end{proof}
\begin{proof} (of Proposition \ref{LossLFP}) 
We consider
\[
\xi _v(x,v) = \intvp{\sigma (|v - \vp|) S(v - \vp)f (x, v) \nabla _{\vp} f (x,\vp)}.
\]
By Proposition \ref{ComDivAve} we have
\begin{align*}
\ave{\qfpl ^- (f,f)} & = \ave{\divv \;\xi _v} = \frac{1}{\oc} \divolx {\left \{ \ave{{^\perp \xi _\olv}} + \ave{\xi _\olv \cdot \frac{\polv}{|\olv|}} \frac{\olv}{|\olv|} - 
\ave{\xi _\olv \cdot \frac{\olv}{|\olv|}} \frac{\polv}{|\olv|}  \right \}}  \\
& + \divolv\left \{\ave{\xi _\olv \cdot \frac{\polv}{|\olv|}} \frac{\polv}{|\olv|} + 
\ave{\xi _\olv \cdot \frac{\olv}{|\olv|}} \frac{\olv}{|\olv|}
\right \} + \partial _{v_3} \ave{\xi _{v_3}} \\
& = \mathrm{div}_{\oc x} \left \{ \ave{E \xi _v}   
+\ave{\xi _v \cdot \frac{(\polv, 0)}{|\olv|}} \frac{(\olv,0)}{|\olv|} - \ave{\xi _v \cdot \frac{(\olv,0)}{|\olv|}} \frac{(\polv,0)}{|\olv|} 
\right \} \\
& + \divv \left \{ 
\ave{\xi _v \cdot \frac{(\polv,0)}{|\olv|}} \frac{(\polv,0)}{|\olv|} + \ave{\xi _v \cdot \frac{(\olv,0)}{|\olv|}} \frac{(\olv,0)}{|\olv|}
+ \ave{e_3 \otimes e_3 \xi _v}\right \}.
\end{align*}
As in the proof of Proposition \ref{GainLFP} we obtain
\[
\nabla _{\vp} f (x,\vp)= b ^0 \cdot \gradxvp f\;\frac{(\polvp, 0)}{\oc |\olvp|^2} + \left ( -\frac{^\perp \gradolx f}{\oc}, \partial _{\vtp} f     \right ) + b ^4 \cdot \gradxvp f\;\frac{(\olvp, 0)}{ |\olvp|}
\]
and therefore 
\begin{align*}
& \ave{\xi _v}   = f(x,v)\avess{\frac{b ^0 \cdot \gradxvp f}{\oc |\olvp|^2}, (\polvp, 0)}  - f(x,v) \avess{\frac{\partial _{x_2} f }{\oc}, e_1} + f(x,v) \avess{\frac{\partial _{x_1} f }{\oc}, e_2} \\
& + f(x,v) \avess{ \partial _{\vtp} f, e_3  }
 + f (x,v) \avess{\frac{b ^4 \cdot \gradxvp f}{ |\olvp|}, (\olvp, 0)}  \\
& = - \oc ^2 \intolxpvp{\sigma f (x,v)\chi \frac{\rp - r \cos \varphi }{|z|} \frac{(z,0)}{|z|} \otimes \olvppolvp \gradoxvp f } \\
& - \oc ^2 \intolxpvp{\sigma f (x,v)\chi \frac{(\rp - r \cos \varphi ) ( v_3 - \vtp)}{|z|^2 + (v_3 - \vtp)^2}\left ( \frac{v_3 - \vtp}{|z|^2} \;^\perp z, 1\right ) \\
& \otimes \polvpolvp \gradoxvp f } \\
& - \oc ^2 \intolxpvp{\sigma f (x,v)\chi S ( \;({^\perp z}, \vtp - v_3)\;) (E, - e_3 \otimes e_3) \gradoxvp f }.
\end{align*}
Similarly, thanks to the identities 
\[
\ave{\frac{b^0 \cdot \gradxvp f}{\oc |\olvp|^2 }, (\polvp, 0), (\olv, 0)}_{\sigma S}  = \ave{\frac{b^4 \cdot \gradxvp f }{|\olvp|}, (\olvp, 0), (\polv, 0)}_{\sigma S} = 0 
\]
we obtain
\begin{align*}
& \ave{\xi _v \cdot \frac{(\olv, 0)}{|\olv|}} = - f (x,v)\avess{\frac{\partial _{x_2} f(x,\vp) }{\oc |\olv|}, (\olv, 0)}\!\!\!\!\cdot e_1 + f(x,v) \avess{\frac{\partial _{x_1} f (x, \vp)}{\oc |\olv|}, (\olv, 0)}\!\!\!\!\cdot e_2\\
& + f(x,v) \avess{\frac{\partial _{\vtp} f(x,\vp) }{|\olv|}, (\olv, 0)}\!\!\!\!\cdot e_3 + f(x,v) \avess{\frac{b^4 \cdot \gradxvp f (x,\vp)}{|\olv|\;|\olvp|}, (\olvp,0), (\olv,0)} \\
& = - \oc ^2 \intolxpvp{\sigma f \chi \frac{(\rp \cos \varphi - r )(v_3 - \vtp )}{|z|^2 + (v_3 - \vtp )^2}\left (\frac{v_3 - \vtp}{|z|} \frac{(z,0)}{|z|}, - e_3     \right )\cdot \gradoxvp f } \\
& - \oc ^2 \intolxpvp{\sigma f \chi \left ( \cos \varphi   + \frac{(r - \rp \cos \varphi )(\rp - r \cos \varphi )}{|z|^2 + (v_3 - \vtp )^2}\right )\polvpolvp \\
& \cdot \gradoxvp f   }
\end{align*}
and
\begin{align*}
& \ave{\xi _v \cdot \frac{(\polv, 0)}{|\olv|}} = - f (x,v)\avess{\frac{\partial _{x_2} f(x,\vp) }{\oc |\olv|}, (\polv, 0)}\!\!\!\!\cdot e_1 + f(x,v) \avess{\frac{\partial _{x_1} f (x, \vp)}{\oc |\olv|}, (\polv, 0)}\!\!\!\!\cdot e_2\\
& + f(x,v) \avess{\frac{\partial _{\vtp} f(x,\vp) }{|\olv|}, (\polv, 0)}\!\!\!\!\cdot e_3 + f(x,v) \avess{\frac{b^0 \cdot \gradxvp f (x,\vp)}{\oc |\olv|\;|\olvp|^2}, (\polvp,0), (\polv,0)} \\
& = - \oc ^2 \intolxpvp{\sigma f \chi \frac{\rp \cos \varphi - r }{|z|}\left (\frac{({^\perp z},0)}{|z|}, 0     \right )\cdot \gradoxvp f } \\
& + \oc ^2 \intolxpvp{\sigma f \chi \left ( \cos \varphi -   \frac{r \rp \sin ^2 \varphi }{|z|^2 + (v_3 - \vtp )^2}\right )\olvppolvp \cdot \gradoxvp f   }.
\end{align*}
Obviously $\ave{\qfpl ^- (f, f)}$ has the form in \eqref{Equ80} with
$
A^- = \left (
\begin{array}{rrr}
A_{xx} ^-  &  A_{xv} ^-\\
A_{vx}^-  &\;\;\;   A_{vv}^-
\end{array}
\right )
$
where
\begin{align}
\label{Equ82}
(A_{xx} ^-, A_{xv} ^-)  & = - \frac{\rp - r \cos \varphi }{|z|} \;\frac{({^\perp z}, 0)}{|z|} \otimes  \olvppolvp \nonumber \\
& + 
\frac{(\rp - r \cos \varphi ) (v_3 - \vtp )^2}{|z|[\;|z|^2 + (v_3 - \vtp)^2\;]}\;\frac{(z,0)}{|z|} \otimes \polvpolvp \nonumber \\
& + \left ( \cos \varphi  - \frac{r \rp  \sin ^2 \varphi }{|z|^2 + (v_3 - \vtp)^2}  \right )  \frac{(\olv, 0)}{|\olv|} \otimes \olvppolvp \nonumber \\
& + \frac{r - \rp \cos \varphi }{|z|} \;\frac{(\olv, 0)}{|\olv|} \otimes \left (\frac{({^\perp z}, 0)}{|z|}, 0 \right ) \nonumber \\
& +  \left ( \cos \varphi  + \frac{(r - \rp \cos \varphi)(\rp - r \cos \varphi) }{|z|^2 + (v_3 - \vtp)^2}  \right )  \frac{(\polv, 0)}{|\olv|} \otimes \polvpolvp \nonumber \\
& - \frac{(r - \rp \cos \varphi)(v_3 - \vtp) }{|z|^2 + (v_3 - \vtp)^2}  \;  \frac{(\polv, 0)}{|\olv|} 
\otimes \left (\frac{v_3 - \vtp}{ |z|}\;\frac{(z,0)}{|z|}, - e_3   \right ) \nonumber \\
& - E S (\;({^\perp z}, \vtp - v_3)\; )( E, -e_3 \otimes e_3 )
\end{align}
and
\begin{align}
\label{Equ83}
(A_{vx} ^-, A_{vv} ^-)  & =   \left ( \cos \varphi  - \frac{r \rp  \sin ^2 \varphi }{|z|^2 + (v_3 - \vtp)^2}  \right )  \frac{(\polv, 0)}{|\olv|} \otimes \olvppolvp \nonumber \\
& + \frac{r - \rp \cos \varphi }{|z|} \;\frac{(\polv, 0)}{|\olv|} \otimes \left (\frac{({^\perp z}, 0)}{|z|}, 0 \right ) \nonumber \\
& - \left ( \cos \varphi  + \frac{( r - \rp \cos \varphi )( \rp - r \cos \varphi )}{|z|^2 + (v_3 - \vtp)^2}  \right )  \frac{(\olv, 0)}{|\olv|} \otimes \polvpolvp \nonumber \\
& + \frac{(r -\rp \cos \varphi)(v_3 - \vtp) }{|z|^2 + (v_3 - \vtp)^2}  \;  \frac{(\olv, 0)}{|\olv|} 
\otimes \left (\frac{v_3 - \vtp}{|z|}\;\frac{(z,0)}{|z|}, - e_3   \right ) \nonumber \\
& -
\frac{(\rp - r \cos \varphi ) (v_3 - \vtp )}{|z|^2 + (v_3 - \vtp)^2}\;e_3 \otimes \polvpolvp \nonumber \\
& - e_3 \otimes e_3  \;S (\;({^\perp z}, \vtp - v_3)\; )  ( E, - e_3 \otimes e_3).
\end{align}
It is easily seen that the matrix $A^-$ writes
\begin{align}
\label{Equ86}
A^- & = \left ( \cos \varphi  - \frac{r \rp  \sin ^2 \varphi }{|z|^2 + (v_3 - \vtp)^2}  \right ) \olvpolv \otimes \olvppolvp \nonumber \\
& + \left ( \cos \varphi + \frac{( r - \rp \cos \varphi )(\rp - r \cos \varphi )}{|z|^2 + (v_3 - \vtp)^2}  \right ) \polvolv \otimes \polvpolvp \nonumber \\
& - \frac{\rp - r \cos \varphi }{|z|} \left ( \frac{({^\perp z}, 0)}{|z|}, 0\right )  \otimes  \olvppolvp \nonumber \\
& + \frac{r - \rp \cos \varphi }{|z|} \olvpolv \otimes \left (\frac{({^\perp z}, 0)}{|z|}, 0 \right ) \nonumber \\
& + 
\frac{(\rp - r \cos \varphi ) (v_3 - \vtp )}{|z|\sqrt{|z|^2 + (v_3 - \vtp)^2}}\vtz \otimes \polvpolvp \nonumber \\
& - 
\frac{(r - \rp \cos \varphi ) (v_3 - \vtp )}{|z|\sqrt{|z|^2 + (v_3 - \vtp)^2}}\polvolv \otimes \vtz\nonumber \\
& + B^- = A_1 ^- + A_2 ^- +A_3 ^- +A_4 ^- +A_5 ^- +A_6 ^- +B ^-
\end{align}
where, cf. \eqref{Equ77}
\begin{align}
\label{Equ87}
B^- & = \left ( 
\begin{array}{lll}
{^t E} S (\;({^\perp z}, \vtp - v_3)\; )  E  &  \;\;ES(\;({^\perp z}, \vtp - v_3)\; ) e_3 \otimes e_3\\
- e_3 \otimes e_3  S (\;({^\perp z}, \vtp - v_3)\; )  E  & \;\;e_3 \otimes e_3S(\;({^\perp z}, \vtp - v_3)\; ) e_3 \otimes e_3
\end{array}
\right)\\
& =  \left ( \frac{({^\perp z},0)}{|z|}, 0\right ) ^{\otimes 2}     + 
\left (\frac{v_3 - \vtp}{\sqrt{|z|^2 + (v _3 - \vtp)^2}}\frac{( z,0)}{|z|}, -\frac{|z|e_3}{\sqrt{|z|^2 + (v _3 - \vtp)^2}}  \right )^{\otimes 2} = : B_1 ^ - + B_2 ^-.\nonumber 
\end{align}
Observe that 
\begin{align}
\label{Equ84} A_1 ^- + A_3 ^- + A_4 ^- + B_1 ^- & = \frac{r\rp \sin ^2 \varphi \;(v_3 - \vtp )^2}{|z|^2 [\;|z|^2 + (v_3 - \vtp )^2 \;]}\;\olvpolv \otimes \olvppolvp \nonumber \\
& + \left [\frac{r - \rp \cos \varphi }{|z|}\olvpolv + \left ( \frac{({^\perp z},0)}{|z|}, 0\right )     \right ] \nonumber \\
& \otimes 
\left [\frac{r \cos \varphi - \rp }{|z|}\olvppolvp + \left ( \frac{({^\perp z},0)}{|z|}, 0\right )     \right ] 
\end{align}
since 
\[
\frac{r \rp \sin ^2 \varphi \;(v_3 - \vtp)^2}{|z|^2 [\;|z|^2 + (v _3 - \vtp )^2\;]} + \frac{(r - \rp \cos \varphi ) (r \cos \varphi - \rp) }{|z|^2}  = \cos \varphi  - \frac{r \rp  \sin ^2 \varphi }{|z|^2 + (v _3 - \vtp )^2}
\]
and
\begin{align}
\label{Equ85} & A_2 ^- + A_5 ^- + A_6 ^- + B_2 ^-  = \frac{r \rp  \sin ^2 \varphi }{|z|^2 } \polvolv \otimes \polvpolvp   \nonumber \\
& + \left [ \frac{(r - \rp \cos \varphi )(v_3 - \vtp)}{|z|\sqrt{|z|^2 + (v _3 - \vtp)^2}}\polvolv - 
\vtz \right ] \nonumber \\
& \otimes  \left [ \frac{(r \cos \varphi - \rp)(v_3 - \vtp)}{|z|\sqrt{|z|^2 + (v _3 - \vtp)^2}}\polvpolvp - 
\vtz \right ]
\end{align}
since
\[
\frac{r \rp \sin ^2 \varphi}{|z|^2} + \frac{(\rp \cos \varphi - r ) ( \rp - r \cos \varphi ) (v_3 - \vtp )^2}{|z|^2 \;[ \;|z|^2 + (v _3 - \vtp ) ^2\;]} = \cos \varphi + \frac{(r - \rp \cos \varphi ) (\rp - r \cos \varphi )}{|z|^2 + (v_3 - \vtp )^2}
\]
Our conclusion follows by combining \eqref{Equ86}, \eqref{Equ87}, \eqref{Equ84}, \eqref{Equ85}.
\end{proof}

\end{document}